\pdfoutput=1
\documentclass[11pt]{article}

\usepackage[margin=1in]{geometry}
\usepackage[T1]{fontenc}
\usepackage{lmodern}
\usepackage{microtype}
\usepackage{setspace}
\usepackage{natbib}
\usepackage{amsmath,amssymb,amsthm}
\allowdisplaybreaks
\usepackage{booktabs}
\usepackage{graphicx}
\usepackage{float}
\usepackage{tikz}
\usetikzlibrary{patterns}
\usepackage{hyperref}

\newtheorem{theorem}{Theorem}
\newtheorem{assumption}{Assumption}
\newtheorem{lemma}{Lemma}

\makeatletter
\newcounter{proofstep}
\AddToHook{env/proof/begin}{\setcounter{proofstep}{0}}
\newcommand{\proofstep}[1]{\par\if@newlist\else\medskip\fi\noindent
  \stepcounter{proofstep}{\emph{Step~\theproofstep. #1.}}}
\makeatother

\newcommand{\E}{\mathbb{E}}

\newcommand{\Var}{\operatorname{Var}}
\newcommand{\Cov}{\operatorname{Cov}}
\newcommand{\Pois}{\operatorname{Pois}}
\newcommand{\R}{\mathbb{R}}

\title{{Stationary Unit Roots and Inference in Softplus INGARCH Models}}
\author{Yang Lu\thanks{Department of Mathematics and Statistics, Concordia University, Montreal, Canada. Email: yang.lu@concordia.ca}, Fukang Zhu\thanks{School of Mathematics, Jilin University, Changchun, China. Email: zfk8010@163.com}}
\date{\today}
\begin{document}
\maketitle

\begin{abstract}
{A unit root makes linear count time series models such as the Poisson
INGARCH(1,1) nonstationary. We show that a softplus link changes this.
Wei{\ss}, Zhu and Hoshiyar (2022) introduced softplus INGARCH models and
gave a sufficient condition for stationarity. For Poisson, negative
binomial, gamma and related versions of these models, we obtain the
necessary and sufficient condition, which allows a unit root. At a
stationary unit root, the conditional mean has a power-law tail, and the
observations can have infinite variance or long memory. To our knowledge,
this is the first stationary unit root established for an INGARCH-type
count model. For a fixed softplus scale, the Poisson quasi-maximum
likelihood estimator is asymptotically normal both below and at the unit
root, including cases with infinite variance. Two applications show that
unit-root conclusions can depend on the softplus scale.
}
\end{abstract}

\begingroup
\small
\setstretch{1}
\noindent\textbf{{Acknowledgments}.} Lu's work is supported by NSERC
(RGPIN-2021-04144) and SSHRC (430-2024-01325).
\par\smallskip
\noindent Zhu's work is supported by the National Natural Science Foundation
of China (No.~12671320) and the Science and Technology Development Plan
Project of Jilin Province of China (No.~20250102003JC).
\par
\endgroup
\medskip

\begingroup
\small
\setstretch{1}
\par
\endgroup
\medskip

\section{Introduction}
{Count time series often show persistence close to the boundary of
stationarity. A standard model for these series is the Poisson
integer-valued generalized autoregressive conditional heteroskedasticity
(INGARCH) model of order \((1,1)\). Given the past, \(X_t\) is Poisson with
mean \(M_t=\omega+\alpha X_{t-1}+\beta M_{t-1}\). Writing
\(U_t=X_t-M_t\) gives the ARMA\((1,1)\) representation
\[
 X_t=\omega+\phi X_{t-1}+U_t-\beta U_{t-1},
 \qquad \phi=\alpha+\beta.
\]
When \(\phi=1\), the autoregressive polynomial \(1-\phi z\) has a unit
root, and we call this the \emph{unit-root case}. The standard
restrictions \(\omega>0\) and \(\alpha,\beta\geq0\) ensure positivity of
\(M_t\). With positive parameters, the model admits no stationary
solution at \(\phi=1\) \citep{michel2020limiting}.

Wei{\ss}, Zhu and Hoshiyar \citeyearpar{weiss2022softplus} introduced
softplus INGARCH models, which keep the conditional mean positive while
allowing coefficients of either sign, and hence negative
autocorrelation, as in log-linear Poisson autoregression
\citep{fokianos2011loglinear}. See also
\citet{jahn2023approximately,jahn2024artificial}. Their softplus
INGARCH\((1,1)\) model replaces the linear recursion by
\begin{equation}
\label{mrecursion}
 M_t=s_c(\omega+\alpha X_{t-1}+\beta M_{t-1}),
\end{equation}
where \(\omega,\alpha,\beta\) are real parameters, \(c>0\) is a fixed
softplus scale, and \(s_c(x)=c\log(1+e^{x/c})\), \(x\in\mathbb R\). They
show that a stationary solution exists if
\begin{equation}
\label{existing}
 \max(0,\alpha)+\max(0,\beta)<1,\qquad |\beta|<1.
\end{equation}
Because softplus is positive for every real input, the intercept may now
be negative. We keep the name unit-root case for \(\phi=1\) in this
nonlinear model, a case that condition~\eqref{existing} excludes. Can a
negative intercept make a unit root stationary?

The difficulty is that a unit root leaves only a weak restoring force.
For large \(M_t\), the softplus recursion at \(\phi=1\) behaves like a
random walk whose drift tends to \(\omega\) and whose increments have
conditional variance proportional to \(M_t\). A negative intercept
therefore pulls the process back with a force that stays bounded, while
its fluctuations grow with the level. The contraction criteria of
\citet{neumann2011absolute} and \citet{davis2016theory}, and the
geometric drift condition used by \citet{doukhan2019absolute}, do not
cover this boundary. The balance between the bounded pull and the
growing fluctuations also governs the tail of any stationary solution,
which can have infinite variance or long memory. Inference must then
cope with possibly infinite variance and, when coefficients may be
negative, with conditional means that can come arbitrarily close to
zero.

Count observations add a difficulty of their own, present at every
parameter value. Given the current conditional mean, the next one can
take only countably many values, so the transition laws from two
different states can be mutually singular. The density-based arguments
that give mixing for real-valued models therefore do not apply to the
counts.

In this paper, we show that the answer is affirmative. We study softplus
{affine\((1,1)\) processes, which extend the softplus
INGARCH\((1,1)\) model beyond the Poisson case. They keep
recursion~\eqref{mrecursion} but}
allow infinitely divisible conditional families indexed by \(M_t\), with
mean \(M_t\) and variance proportional to \(M_t\), including the Poisson,
negative binomial and gamma families. Under our assumptions on the
conditional family and the invertibility condition \(|\beta|<1\), we
obtain
\[
\text{a stationary solution exists}
\quad\Longleftrightarrow\quad
\phi<1
\quad\text{or}\quad
(\phi=1,\ \omega<0).
\]
The region \(\phi<1\) strictly contains the region defined by
condition~\eqref{existing}, and throughout it the observations are
geometrically \(\beta\)-mixing, whereas earlier geometric mixing results
for softplus models required \(|\alpha|+|\beta|<1\)
\citep{doukhan2022}. We call the case \(\phi=1\), \(\omega<0\) a
\emph{stationary unit root}. There, the stationary distribution of
\(M_t\) has a power-law tail, and the observations can have infinite
variance or long memory. To our knowledge, this is the first stationary
unit root established for an INGARCH-type count model. For a fixed
softplus scale and under our inference assumptions, the Poisson
quasi-maximum likelihood estimator (QMLE) of \((\omega,\alpha,\beta)\) is
consistent and asymptotically normal at the usual \(\sqrt n\) rate, both
below the unit root and at a stationary unit root, whenever the
observations have a finite mean. The one-sided Wald test of
\(H_0:\phi=1\) against \(H_1:\phi<1\) therefore has a standard normal
null limit, even in the infinite-variance and long-memory cases. For
Poisson observations, we also derive its local power.

Section~\ref{sec:affine-extension} establishes the stationarity, tail
and dependence results and relates them to other models.
Section~\ref{sec:likelihood-consequences} develops estimation and
unit-root inference, and Section~\ref{sec:inference-simulation}
examines their finite-sample performance by simulation.
Section~\ref{sec:empirical-applications} applies the methods to monthly
recorded-crime counts in New South Wales (NSW) and one-second TRX/USDT
transaction counts on Binance, and shows that unit-root conclusions can
depend on the softplus scale. Section~\ref{sec:conclusion} concludes.
Proofs and additional simulation results are in the appendices.
}

\section[Softplus affine(1,1) processes]{Softplus affine\((1,1)\) processes}
\label{sec:affine-extension}
\subsection{The setting}
{
Let \(G_1\) be an infinitely divisible distribution on \([0,\infty)\)
with Laplace transform \(e^{-a(u)}\), \(u\geq0\), where \(a(0)=0\).
The function \(a\) is called its \emph{Laplace exponent}.
Infinite divisibility ensures that, for every \(m\geq0\), there is a
probability distribution \(G_m\) with Laplace transform \(e^{-ma(u)}\);
in particular, \(G_0=\delta_0\). Independent variables with laws
\(G_m\) and \(G_{m'}\) have a sum with law \(G_{m+m'}\).

For a nonnegative observation process \((X_t)\) and a positive process
\((M_t)\), let \(\mathcal F_t:=\sigma(M_s,X_s:s\leq t)\).
The softplus affine\((1,1)\) model is defined by
{recursion~\eqref{mrecursion} and}
\begin{equation}
 \E(e^{-uX_t}\mid\mathcal F_{t-1})
 =\exp(-M_t a(u)),\qquad u\geq0,
 \label{eq:affine-softplus-model}
\end{equation}
{that is, \(X_t\mid\mathcal F_{t-1}\sim G_{M_t}\).}
We call \(\eta_t=\omega+\alpha X_{t-1}+\beta M_{t-1}\) the
\emph{linear index}, the input to the softplus function, so that
\(M_t=s_c(\eta_t)\).
{The softplus function is increasing and one-Lipschitz, and
\begin{equation}
 0\leq s_c(x)-\max(x,0)\leq c\log 2,
 \qquad x\in\mathbb R.
 \label{eq:softplus-positive-part-bound}
\end{equation}}%
We impose the mean-one normalization \(a'(0)=1\).\footnote{%
For an unnormalized seed with finite positive mean \(\mu=a'(0)\),
set \(\widetilde a=a/\mu\) and \(M_t^\sharp=\mu M_t\).
This leaves the conditional Laplace transform unchanged and gives
\(\widetilde a'(0)=1\). The rescaled state satisfies
\[
 M_{t+1}^\sharp
 =s_{\mu c}(\mu\omega+\mu\alpha X_t+\beta M_t^\sharp).
\]
Thus the normalization rescales the state, softplus scale, intercept,
and observation coefficient.}
Then \(M_t\) is the conditional mean of \(X_t\) given
\(\mathcal F_{t-1}\); we also call it the state.
The term \emph{affine} refers to the linear
dependence of the conditional Laplace exponent \(M_ta(u)\) on \(M_t\),
and \((1,1)\) denotes one lag of \(X_t\) and one lag of \(M_t\) in
{recursion~\eqref{mrecursion}}.

Under this model, \((M_t)\) is a time-homogeneous Markov chain on
\((0,\infty)\).
A stationary solution means a jointly strictly stationary process
\(((X_t,M_t))_{t\in\mathbb Z}\) satisfying
\eqref{mrecursion} and~\eqref{eq:affine-softplus-model}.
{Under this stationary solution, the marginal law of \(M_0\)
is an invariant probability measure for \((M_t)\).}
Uniqueness of a stationary solution always refers to its joint law.

We impose the following conditions on the conditional family
\((G_m)_{m\geq0}\).

\begin{assumption}
\phantomsection
\label{ass:affine-seed}
\begin{enumerate}
\item[(i)] \(G_1\) has finite moments of all orders and variance
\mbox{\(\sigma^2:=-a''(0)>0\)}.
\item[(ii)] One of the following two conditions holds:
  \begin{enumerate}
  \item[(a)] \(G_1(\{0\})>0\).
  \item[(b)]
  \(G_1(\{0\})=0\), and for every \(m>0\), \(G_m\) has a
  density \(g(m,x)\) such that \((m,x)\mapsto g(m,x)\) is jointly
  continuous and strictly positive on \((0,\infty)^2\).
  \end{enumerate}
\end{enumerate}
\end{assumption}

Under this assumption and the normalization, differentiating the
conditional Laplace transform at zero gives
\begin{equation}
 \E(X_t\mid\mathcal F_{t-1})=M_t,
 \qquad
 \Var(X_t\mid\mathcal F_{t-1})=\sigma^2M_t.
 \label{eq:affine-conditional-moments}
\end{equation}
Thus \(\sigma^2\) is the
\emph{conditional dispersion parameter}, the conditional variance-to-mean
ratio.
{The conditional-mean innovation
\begin{equation}
 U_t:=X_t-M_t
 \label{eq:affine-innovation}
\end{equation}
therefore satisfies \(\E(U_t\mid\mathcal F_{t-1})=0\) and
\(\E(U_t^2\mid\mathcal F_{t-1})=\sigma^2M_t\).}

Case~(a) covers compound Poisson distributions, including Poisson
and negative binomial counts; case~(b) covers continuous
distributions with jointly continuous, strictly positive densities.
The proofs of uniqueness and mixing in
Theorem~\ref{thm:affine-mixing} use observation coupling in case~(a)
and density-based minorization in case~(b).
}
Table~\ref{tab:affine-seeds} lists several standard conditional distributions satisfying Assumption~\ref{ass:affine-seed}.
\begin{table}[H]
\centering
\small
\caption{{Examples of conditional distributions with mean \(m\) and variance \(\sigma^2m\).}}
\label{tab:affine-seeds}
\begin{tabular}{lll}
\toprule
Conditional family & Laplace exponent \(a(u)\) & \(\sigma^2\)\\
\midrule
\multicolumn{3}{l}{{\textit{Case (a): positive mass at zero}}}\\
Poisson & \(1-e^{-u}\) & \(1\)\\
Negative binomial, fixed \(0<p<1\)
 & \(\dfrac{p}{1-p}\log\!\left(1+\dfrac{1-p}{p}(1-e^{-u})\right)\)
 & \(1/p\)\\
{Compound Poisson--exponential, \(\tau>0\)}
 & \(u/(1+\tau u)\) & \(2\tau\)\\
{Compound Poisson--gamma} Tweedie, \(\nu,\tau>0\)
 & \((1-(1+\tau u)^{-\nu})/(\nu\tau)\) & \((\nu+1)\tau\)\\
\addlinespace
\multicolumn{3}{l}{{\textit{Case (b): no mass at zero and a positive continuous density}}}\\
Gamma, scale \(\tau>0\)
 & \(\tau^{-1}\log(1+\tau u)\) & \(\tau\)\\
Inverse Gaussian, \(\tau>0\)
 & \((\sqrt{1+2\tau u}-1)/\tau\) & \(\tau\)\\
\bottomrule
\end{tabular}
\par\smallskip
\begin{minipage}{\textwidth}
\footnotesize
{
\emph{Notes:} At conditional mean \(m\), the negative binomial row uses
\(\operatorname{NB}(\text{size}=mp/(1-p),\ \text{probability}=p)\).
Thus its size varies with \(m\) and its variance is \(m/p\); a fixed
size \(r\) would instead give variance \(m+m^2/r\).
The compound Poisson--exponential row is the \(\nu=1\) Tweedie special
case \citep{monfort2016staying}.
}
\end{minipage}
\end{table}

\subsection{Below the unit root}
{
{Given \(M_t=m\), the observation \(X_t\) has mean \(m\) and
variance \(\sigma^2m\) by \eqref{eq:affine-conditional-moments}, so the
relative deviation \((X_t-m)/m\) has variance \(\sigma^2/m\), which
vanishes as \(m\) grows. Since
\(\omega+\alpha X_t+\beta m=\phi m+\omega+\alpha(X_t-m)\),
bound~\eqref{eq:softplus-positive-part-bound} implies
that}\footnote{This is the case \(\chi=1\) of
Lemma~\ref{lem:affine-cancellation} in Appendix~\ref{app:affine-extension}.}
\[
 \lim_{m\to\infty}
 \frac{\E(M_{t+1}\mid M_t=m)}{m}
 =\max(\phi,0).
\]
When \(\phi<1\), this limit is strictly less than one, so large values
of \(M_t\) are pulled down on average. This drift estimate underlies
the following theorem.

}

\begin{theorem}
\label{thm:affine-classification}
Suppose Assumption~\ref{ass:affine-seed} holds and \(\phi<1\).
{
A stationary solution exists. Under every stationary solution,
\(M_t\) and \(X_t\) have moments of all positive orders.
}
\end{theorem}
{For Poisson observations, \citet{weiss2022softplus} prove
stationarity under~\eqref{existing} and finite moments of all positive
orders under \(|\alpha|+|\beta|<1\). Both sufficient regions are strictly
contained in \(\{\phi<1\}\). For example,
Theorem~\ref{thm:affine-classification} covers
\((\alpha,\beta)=(1.2,-0.3)\), for which \(\phi=0.9\), \(\alpha>1\),
and \(|\alpha|+|\beta|=1.5\).}

{
{No restriction on \(\beta\) is needed for
Theorem~\ref{thm:affine-classification}. The remaining results also
require recursion~\eqref{mrecursion} to forget its starting value. If
\((M_t^{(1)})\) and \((M_t^{(2)})\) are computed from the same
observations and parameter values but from different positive starting
values, then, since \(s_c\) is one-Lipschitz,}
\begin{equation}
  |M_t^{(1)}-M_t^{(2)}|
  \leq |\beta|^t|M_0^{(1)}-M_0^{(2)}|,
  \qquad t\geq0.
  \label{eq:filter-contraction}
\end{equation}
{When \(|\beta|<1\), the filter therefore forgets its
starting value geometrically fast. We call \(|\beta|<1\) the
invertibility condition and assume it from now on.}
}

\begin{assumption}
\label{ass:invertibility}
\(|\beta|<1\).
\end{assumption}

Figure~\ref{fig:region-inclusions} compares the coefficient regions.
{
Within the invertibility strip \(|\beta|<1\), we call the part of the
{region \(\phi<1\)} not covered by the stationarity condition
\eqref{existing} of \citet{weiss2022softplus} the \emph{enlarged region}.
}

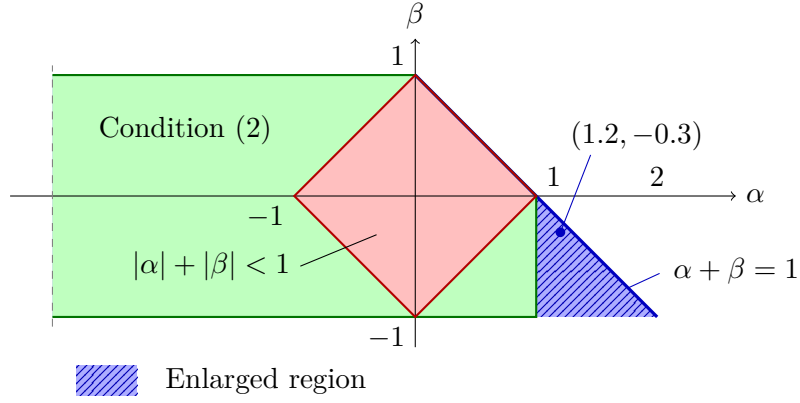
\begin{figure}[H]
\centering
{
\begin{tikzpicture}[scale=1.6]
  \def\A{3}

  \fill[blue!32]
    (-\A,-1) -- (2,-1) -- (0,1) -- (-\A,1) -- cycle;
  \fill[green!25]
    (-\A,-1) -- (1,-1) -- (1,0) -- (0,1) -- (-\A,1) -- cycle;
  \fill[red!25]
    (0,-1) -- (1,0) -- (0,1) -- (-1,0) -- cycle;

  \path[pattern=north east lines, pattern color=blue!70!black]
    (1,0) -- (2,-1) -- (1,-1) -- cycle;

  \draw[->] (-3.35,0) -- (2.65,0) node[right] {$\alpha$};
  \draw[->] (0,-1.25) -- (0,1.3) node[above] {$\beta$};
  \draw[dashed, gray] (-\A,-1.08) -- (-\A,1.08);
  \draw[blue!75!black, very thick] (2,-1) -- (0,1);
  \node[anchor=west] at (2.05,-0.62) {$\alpha+\beta=1$};
  \draw[blue!75!black] (2.02,-0.64) -- (1.76,-0.76);

  \draw[green!45!black, thick]
    (-\A,-1) -- (1,-1) -- (1,0) -- (0,1) -- (-\A,1);
  \node at (-1.9,0.55) {Condition~\eqref{existing}};
  \draw[red!70!black, thick]
    (0,-1) -- (1,0) -- (0,1) -- (-1,0) -- cycle;
  \node[anchor=east] (redlab) at (-0.95,-0.58) {$|\alpha|+|\beta|<1$};
  \draw (redlab.east) -- (-0.32,-0.32);

  \fill[blue!75!black] (1.2,-0.3) circle (0.04);
  \node[anchor=west] at (1.20,0.50) {$(1.2,-0.3)$};
  \draw[blue!75!black] (1.45,0.34) -- (1.2,-0.3);

  \node[below left] at (-1,0) {$-1$};
  \node[above right] at (1,0) {$1$};
  \node[above] at (2,0) {$2$};
  \node[above left] at (0,1) {$1$};
  \node[below left] at (0,-1) {$-1$};

  \fill[blue!32] (-2.8,-1.65) rectangle (-2.3,-1.4);
  \path[pattern=north east lines, pattern color=blue!70!black]
    (-2.8,-1.65) rectangle (-2.3,-1.4);
  \node[anchor=west] at (-2.15,-1.525) {Enlarged region};
\end{tikzpicture}
}
\caption{{Coefficient regions with \(|\beta|<1\).
Theorem~\ref{thm:affine-classification} covers the whole shaded area
\(\phi<1\). Green and red together form condition~\eqref{existing} of
\citet{weiss2022softplus}, and red alone is \(|\alpha|+|\beta|<1\).
Hatching marks the enlarged region. The plotting window is truncated,
and all defining inequalities are strict.}}
\label{fig:region-inclusions}
\end{figure}

Theorem~\ref{thm:affine-classification} establishes existence and moments.
We next consider {uniqueness in law of the stationary solution} and the rate at which
dependence between observations decays.

For a stationary observation process \((X_t)\), let
\(\mathcal F_0^X=\sigma(X_i:i\leq0)\) and
\(\mathcal F_j^{X,+}=\sigma(X_i:i\geq j)\) denote the sigma-fields generated
by the observations up to time zero and from time \(j\) onward,
respectively. Define the \(\beta\)-mixing coefficient by
\[
\beta_X^{\mathrm{mix}}(j)=
\E\left\|
\mathcal L((X_i)_{i\geq j}\mid\mathcal F_0^X)
-\mathcal L((X_i)_{i\geq j})
\right\|_{\mathrm{TV}},
\]
where $ \mathcal L((X_i)_{i\geq j}\mid\mathcal F_0^X)$ is the conditional distribution of the future given the observation past through time zero, $\mathcal L((X_i)_{i\geq j})$ is the unconditional distribution of the future,  and $\left\|\cdot \right\|_{\mathrm{TV}}$ is the total variation distance between two probability measures.
We use the convention
\[
 \|\mu-\nu\|_{\mathrm{TV}}:=\sup_A|\mu(A)-\nu(A)|,
\]
where the supremum is over measurable sets in their common space.
We say that the process $(X_t)$ is geometrically \(\beta\)-mixing if \(\beta_X^{\mathrm{mix}}(j)\) is bounded by a geometrically decreasing sequence.

\begin{theorem}
\label{thm:affine-mixing}
{Under the assumptions of Theorem~\ref{thm:affine-classification}
and Assumption~\ref{ass:invertibility}, the stationary solution is unique
in law and ergodic.} For some \(C<\infty\) and
\(\varrho_{\mathrm{mix}}\in(0,1)\),
\(\beta_X^{\mathrm{mix}}(j)\leq C\varrho_{\mathrm{mix}}^j\).
\end{theorem}

{
{In the Poisson case, \citet[Theorem~1]{weiss2022softplus}
obtain a mixing bound of order \(\varrho_{\mathrm{mix}}^{\sqrt j}\), and
\citet{doukhan2022} establish geometric mixing for softplus models under
\(|\alpha|+|\beta|<1\). More general results for nonlinear count models
\citep{neumann2011absolute,davis2016theory} require the recursion to be
Lipschitz in the state and in the observation, with constants summing to
less than one. For the softplus recursion with \(\alpha>0\), these
constants are at least \(|\beta|\) and \(\alpha\), because softplus
approaches the identity at large positive arguments. The condition
therefore fails at \((\alpha,\beta)=(1.2,-0.3)\), which
Theorem~\ref{thm:affine-mixing} covers, and on the whole line
\(\phi=1\), where \(\alpha+|\beta|=1-\beta+|\beta|\geq1\).
\citet{doukhan2019absolute} require contraction in the state only, a
linear drift condition \(\E(M_{t+1}\mid M_t=m)\leq\lambda m+B\) with
\(\lambda<1\), and a total-variation similarity condition on the
conditional family, and obtain a mixing rate of order
\(\varrho^{\sqrt j}\). By
\eqref{eq:softplus-positive-part-bound} and
\(\Var(X_t\mid M_t=m)=\sigma^2m\),
\[
 \omega+\phi m\leq\E(M_{t+1}\mid M_t=m)
 \leq\max(\omega+\phi m,0)+|\alpha|\sigma\sqrt m+c\log 2.
\]
For conditional families with positive mass at zero, their conditions
therefore hold throughout the region \(\phi<1\), and
Theorem~\ref{thm:affine-mixing} improves the rate to a geometric one.
At \(\phi=1\) the lower bound rules out their drift condition. The
contraction conditions of \citet{debaly2021iterations} and
\citet{truquet2023strong} also exclude \(\phi=1\).
Section~\ref{sec:at-unit-root} treats this case with the theory of
Markov chains with asymptotically zero drift \citep{denisov2025markov},
applied to \(\sqrt{M_t}\).}
}

{
{The proof of Theorem~\ref{thm:affine-mixing} must handle
the singular transitions described in the introduction.
\citet[Remark~3]{neumann2011absolute} gives an observation-driven Poisson
example in which the conditional-mean process fails to be strongly
mixing although the counts are absolutely regular, so proving mixing
of the conditional means can be too restrictive.
The proof is in Appendix~\ref{app:proof-affine-mixing}.}
}

\subsection{At and above the unit root}
\label{sec:at-unit-root}
{
{We now consider \(\phi=1\), where Assumption~\ref{ass:invertibility}
gives \(\alpha=1-\beta>0\).}
Unlike the standard linear Poisson
INGARCH model with a positive intercept, the softplus model can admit
a stationary solution at this boundary.
}

\begin{theorem}
\label{thm:affine-boundary}
{Suppose Assumptions~\ref{ass:affine-seed} and
\ref{ass:invertibility} hold and \(\phi=1\).
\begin{enumerate}
\item If \(\omega<0\), a stationary solution exists and is unique in law.
 Under this solution,
 \begin{enumerate}
 \item[(a)] \(\Pr(M_0>y)\sim Cy^{-\kappa}\) as \(y\to\infty\) for some
 \(C>0\), where \(\kappa:=-2\omega/(\alpha^2\sigma^2)\) is the tail index;
 \item[(b)] for every \(s>0\), the moments \(\E M_0^s\) and \(\E X_0^s\) are
 finite if and only if \(s<\kappa\);
 \item[(c)] \((X_t)\) is not geometrically \(\beta\)-mixing, and
 \(j^r\beta_X^{\mathrm{mix}}(j)\to0\) as \(j\to\infty\)
 for every \(0<r<\kappa\).
 \end{enumerate}
\item If \(\omega\geq0\), no stationary solution exists.
\end{enumerate}}
\end{theorem}

{
{For a covariance-stationary process \((X_t)\) with positive
variance, write
\[
 \rho_X(j)=\frac{\Cov(X_0,X_j)}{\Var(X_0)},\qquad j\geq0,
\]
for its autocorrelations. We use} \emph{short memory} to mean that {these}
autocorrelations are absolutely summable.
When the autocorrelations are nonnegative, we use
\emph{long memory} to mean that
\[
 \sum_{j=0}^{\infty}\rho_X(j)=\infty.
\]
}

{By part~1(b) of Theorem~\ref{thm:affine-boundary} with \(s=2\),
\((X_t)\) is covariance stationary exactly when \(\kappa>2\).}
Within that region, the next result distinguishes a
short-memory sufficient condition from a long-memory regime.

\begin{theorem}
\label{thm:kappa-two-three-long-memory}
Suppose Assumptions~\ref{ass:affine-seed} and
\ref{ass:invertibility} hold, with \(\phi=1\) and \(\omega<0\), and let
\(\kappa\) be as in Theorem~\ref{thm:affine-boundary}.
\begin{enumerate}
\item If \(\kappa>3\), then \((\rho_X(j))\) is absolutely summable.
\item If \(2<\kappa\leq3\) and \(\beta\geq0\), then {\((X_t)\) has
long memory, with \(\rho_X(j)\geq0\) for all \(j\) and
\(\rho_X(j)\geq Cj^{-(\kappa-2)}\) for some \(C>0\) and all sufficiently
large \(j\).}
\end{enumerate}
\end{theorem}

We now turn from the unit-root case to the region \(\phi>1\).
The next theorem shows that {no stationary solution exists} there, regardless of
the intercept.

\begin{theorem}
\label{thm:affine-supercritical}
Suppose Assumptions~\ref{ass:affine-seed} and
\ref{ass:invertibility} hold and \(\phi>1\).  Then, for every
\(\omega\in\mathbb R\), {the model has no stationary solution}.
\end{theorem}

{
Together, Theorems~\ref{thm:affine-classification},
\ref{thm:affine-boundary}, and~\ref{thm:affine-supercritical} show that,
under Assumptions~\ref{ass:affine-seed} and~\ref{ass:invertibility},
{a stationary solution exists} exactly when \(\phi<1\), or when
\(\phi=1\) and \(\omega<0\).
}

{
\subsection{Connections with related models}
\label{sec:literature-comparison}

{
}

\paragraph{Finite-order count autoregressions.}
\citet{costaalmost} study the stabilizing effect of negative autoregressive
coefficients ({negative feedback}).\footnote{\citet{costa2026stability}
extend this analysis to a similar third-order count model.}
Their two-lag model is
\[
 Y_t\mid(Y_{t-1},Y_{t-2})
 \sim\Pois\!\left(
   \max\{b_0+b_1Y_{t-1}+b_2Y_{t-2},0\}
 \right),\qquad b_0>0,
\]
where \(b_1,b_2\in\R\). They establish geometric ergodicity strictly
below the critical curve
\[
 b_2=
 \begin{cases}
  1, & b_1\leq0,\\
  1-b_1, & 0<b_1<2,\\
  -b_1^2/4, & b_1\geq2,
 \end{cases}
\]
and transience strictly above it, leaving the curve itself unclassified.
It is straightforward to check that this curve lies on or below the line
\(b_1+b_2=1\), with equality exactly when \(0\leq b_1\leq2\).
{Outside the unclassified boundary curve, stationarity in
their model requires \(b_1+b_2<1\), but this condition is not sufficient.}

The two models have different feedback structures. Their conditional mean
applies the positive-part link to a linear combination of two lagged
counts, whereas {recursion~\eqref{mrecursion}} combines a
lagged observation with a lagged conditional mean. {For this model,}
Theorem~\ref{thm:affine-classification} establishes stationarity whenever
\(\alpha+\beta<1\), under Assumption~\ref{ass:affine-seed}.

Their conditional mean is an explicit function of observed counts, so no
filter-invertibility condition is needed. In our model, \(|\beta|<1\)
ensures that the effect of initializing the conditional-mean filter
vanishes.

\paragraph{Dynamic Tobit models.}
{For nonnegative real-valued time series,
\citet{bykhovskaya2023networks} considers the scalar censored
autoregression}
\[
 Y_t=\max\{b_0+b_1Y_{t-1}+\varepsilon_t,0\},
\]
with iid mean-zero innovations.\footnote{
\citet{bykhovskayaDuffy2024tobit} study higher-order dynamic Tobit models
in a local-to-unity asymptotic regime.} At \(b_1=1\) and
\(b_0<0\), this is a reflected random walk with negative drift,
which is stationary under {their} innovation assumptions.
{That process is first-order Markov, as is \((M_t)\) in our model.} {The increment variances differ,
however.}
{Their} innovations have state-independent variance, whereas
{for recursion~\eqref{mrecursion} at \(\phi=1\), the conditional
variance of the increment satisfies}
\[
 \Var(M_{t+1}-M_t\mid M_t=m)
 =\alpha^2\sigma^2m+O(1),\qquad m\to\infty.
\]
The interaction between this growing conditional variance and the
negative drift determines the tail index \(\kappa\).
\par}

\paragraph{GARCH models.}
{
A comparison with the GARCH\((1,1)\) model
\(Y_t=\sqrt{h_t}\,\varepsilon_t\),
\(h_t=\omega^\ast+\alpha^\ast Y_{t-1}^2+\beta^\ast h_{t-1}\), is
also instructive.  Here \(\omega^\ast>0\), \(\alpha^\ast,\beta^\ast\geq0\),
and \((\varepsilon_t)\) is iid with \(\E\varepsilon_t^2=1\).
There, strict stationarity holds if and only if
\(\E\log(\alpha^\ast\varepsilon_0^2+\beta^\ast)<0\),
{and the boundary case where this expectation equals zero
admits no strictly stationary solution}
\citep{nelson1990stationarity,bougerolPicard1992garch}.  The integrated
case \(\alpha^\ast+\beta^\ast=1\), with \(\alpha^\ast>0\) and nondegenerate
\(\varepsilon_0^2\), lies strictly inside that region, so IGARCH is strictly
stationary with infinite variance, its heavy tails being generated
multiplicatively.  In the softplus affine\((1,1)\) model, by contrast,
{a stationary solution exists when \(\phi=1\) and
\(\omega<0\), and the marginal law of \(M_0\) under this stationary
solution has tail index
\(\kappa=-2\omega/(\alpha^2\sigma^2)\)}.
}

\paragraph{Recurrent neural networks.}
In recurrent-network terms, \(M_t\) is a scalar recurrent state, softplus is
the activation function, and the stochastic output \(X_t\) feeds back
into {recursion~\eqref{mrecursion}}. In the related literature,
\citet{leisch1998stationarity} establish stationarity and geometric
ergodicity for stable autoregressive neural networks.
\citet{zhao2020rnnlongmemory} show short memory under contraction and
motivate memory-augmented architectures.  For a constant-gate
memory-augmented gated recurrent unit, \citet{yang2025mgru} establish
polynomially decaying input--output coefficients.
{Theorem~\ref{thm:kappa-two-three-long-memory} establishes
nonsummability of the autocorrelations of the stationary observation process.}
{Filter contraction under \(|\beta|<1\) is compatible with
long memory in the endogenous output.}

\section{Inference}
\label{sec:likelihood-consequences}

\subsection{Poisson QMLE}
\label{sec:poisson-qmle}

Throughout this section, the softplus scale \(c>0\) is fixed and known.
The normalization \(a'(0)=1\) fixes the conditional-mean scale, while \(c\)
{sets the amount of smoothing around zero.}
The value \(c=1\) is the convention of
\citet{weiss2022softplus} and of the standard softplus activation in
neural-network work \citep{dugas2000incorporating}.  Moreover, estimating
\(c\) can be difficult in both limiting regimes.  As \(c\downarrow0\),
\[
 0\leq s_c(x)-\max(x,0)=c\log\bigl(1+e^{-|x|/c}\bigr)\leq c\,e^{-|x|/c},
 \qquad x\in\mathbb R,
\]
{so the link converges uniformly to the positive-part function.} For
{linear indices small in magnitude} relative to \(c\),
\(s_c(x)=c\log2+x/2+O(x^2/c)\), which confounds \(c\) with the intercept
\(\omega\). The simulations and applications examine sensitivity to
the fixed scale.

By \eqref{eq:affine-conditional-moments}, every affine family considered
here has conditional mean \(M_t\) and conditional variance
\(\sigma^2M_t\).  This common
conditional-mean specification motivates a single Poisson
quasi-likelihood.
{Write \(\theta=(\omega,\alpha,\beta)^\top\) for a generic dynamics
parameter, and let \(\theta_0=(\omega_0,\alpha_0,\beta_0)^\top\) and
\(\sigma_0^2\) denote the true dynamics parameter and conditional
dispersion parameter, respectively. {For some \(\rho\in(0,1)\), let
\(\Theta\subset\mathbb R\times\mathbb R\times[-\rho,\rho]\)
be compact.}
}
{Observe \(X_0,\ldots,X_n\), using \(X_0\) as the presample
observation for the criterion based on \(X_1,\ldots,X_n\).
Fix a positive initial value \(\widetilde M_0\) independent of
\(\theta\). For each \(\theta\in\Theta\), define
\[
 \widetilde M_t(\theta)
 =s_c(\omega+\alpha X_{t-1}+\beta\widetilde M_{t-1}(\theta)),
 \qquad t=1,\ldots,n.
\]
}
The Poisson quasi-log-likelihood criterion is
\begin{equation}
 Q_n(\theta)=\sum_{t=1}^n
 (X_t\log \widetilde M_t(\theta)-\widetilde M_t(\theta)).
 \label{eq:affine-qmle-criterion}
\end{equation}
{Let
\(\widehat\theta_n=(\widehat\omega_n,\widehat\alpha_n,\widehat\beta_n)^\top\)
be a maximizer of \eqref{eq:affine-qmle-criterion} over \(\Theta\).
We call this estimator the Poisson QMLE and define
\(\widehat\phi_n=\widehat\alpha_n+\widehat\beta_n\).
{The parameter set \(\Theta\) may contain points with
\(\phi>1\). The dispersion \(\sigma^2\) is estimated separately.}}

{
Write
\(\widetilde D_t(\theta)=\partial_\theta\widetilde M_t(\theta)\),
\(\widehat M_t=\widetilde M_t(\widehat\theta_n)\), and
\(\widehat D_t=\widetilde D_t(\widehat\theta_n)\), using the same fixed
initialization as the criterion, and define
\[
 \widehat\sigma_n^2
 =\frac{\sum_{t=1}^n(X_t-\widehat M_t)^2}
        {\sum_{t=1}^n\widehat M_t},
 \qquad
 \widehat I_{Q,n}
 =\frac1n\sum_{t=1}^n
 \frac{\widehat D_t\widehat D_t^\top}{\widehat M_t}.
\]
{When \(\widehat I_{Q,n}\) is nonsingular, we estimate the
covariance matrix of \(\widehat\theta_n\) by
\(\widehat\sigma_n^2\widehat I_{Q,n}^{-1}/n\).}
}

{%
{For \(\theta\in\Theta\), let \(M_t(\theta)\) be the value of
recursion~\eqref{mrecursion} at \(\theta\) computed from the observed past
\(X_{t-1},X_{t-2},\ldots\), so that \(M_t(\theta_0)=M_t\). Since
\(|\beta|<1\) on \(\Theta\) and \(\E X_0<\infty\),
bound~\eqref{eq:filter-contraction} makes \(M_t(\theta)\) well defined at
every \(\theta\in\Theta\), including points with \(\phi>1\), and
\(\widetilde M_t(\theta)\) converges to it.}
}
With expectation taken under the true stationary law, put
\[
 D_t(\theta)=\partial_\theta M_t(\theta),\qquad
 I_0:=
 \E\left[\frac{D_t(\theta_0)D_t(\theta_0)^\top}
                  {M_t(\theta_0)}\right].
\]
{We say that the conditional-mean filter is identified at
\(\theta_0\) over \(\Theta\) if, for every \(\theta\in\Theta\),
\(M_t(\theta)=M_t(\theta_0)\) almost surely implies \(\theta=\theta_0\).}

\begin{theorem}
\label{thm:affine-qmle}
{
Suppose Assumptions~\ref{ass:affine-seed} and \ref{ass:invertibility}
hold, the observations come from a {stationary solution of
\eqref{mrecursion} and \eqref{eq:affine-softplus-model}} at \(\theta_0\), and \(\theta_0\) is an
interior point of the compact set \(\Theta\) defined above.  Put
\(\phi_0=\alpha_0+\beta_0\), and suppose either
\begin{enumerate}
\item[(i)] \(\phi_0<1\) and \(\alpha_0\neq0\); or
\item[(ii)] \(\phi_0=1\), \(\omega_0<0\), and
\(\kappa_0:=-2\omega_0/(\alpha_0^2\sigma_0^2)>1\).
\end{enumerate}
Then the conditional-mean filter is identified at \(\theta_0\) over
\(\Theta\), \(I_0\) is positive definite,
\(\widehat\theta_n\to\theta_0\) almost surely,
\[
 \sqrt n(\widehat\theta_n-\theta_0)
 \Rightarrow N(0,\sigma_0^2I_0^{-1}),
\]
and \(\widehat\sigma_n^2\to\sigma_0^2\) and
\(\widehat I_{Q,n}\to I_0\) in probability.
}
\end{theorem}

{\citet{ahmad2016poisson} assume a positive lower bound on
\(M_t(\theta)\), uniform over \(\Theta\). Such a bound can fail when
coefficients are negative, and Theorem~\ref{thm:affine-qmle} dispenses
with it.}

{
{Theorem~\ref{thm:affine-qmle} also gives the limiting distribution
of the estimated coefficient sum
\(\widehat\phi_n=\widehat\alpha_n+\widehat\beta_n\).
Writing \(e_\phi=(0,1,1)^\top\), we have
\(\widehat\phi_n-\phi_0=e_\phi^\top(\widehat\theta_n-\theta_0)\),
and hence}
\[
 \sqrt n(\widehat\alpha_n+\widehat\beta_n-\phi_0)
 \Rightarrow
 N(0,\sigma_0^2e_\phi^\top I_0^{-1}e_\phi).
\]
{In case~(ii), \(\theta_0\) is an interior point of
\(\Theta\), so \(\Theta\) contains points with \(\phi>1\) and
\(\widehat\phi_n\) may exceed one.}
}

{
{In case~(ii), Theorem~\ref{thm:affine-boundary} and
\(\kappa_0>1\) imply \(\E M_0=\E X_0<\infty\).
This permits observations with infinite variance
(\(1<\kappa_0\leq2\)) or long memory
(\(2<\kappa_0\leq3\), \(\beta_0\geq0\)).}
{The innovation \(U_t\) in \eqref{eq:affine-innovation} is the prediction
error. {Since \(\E M_t<\infty\), the conditional moment
identities imply that \((U_t)\) is a square-integrable martingale
difference sequence with respect to \((\mathcal F_t)\), with}}
\[
 \E U_t^2=\sigma_0^2\E M_t<\infty.
\]
Thus the observations can have infinite variance while their prediction
errors remain square integrable. {The quasi-score is also a
square-integrable martingale difference.} This
structure allows a standard martingale central limit theorem for the
score and explains how the usual \(\sqrt n\) rate can coexist with long
memory in the observations.
}

{For Poisson observations, \(\sigma_0^2=1\) and the QMLE is
the conditional MLE. The exact conditional log-likelihood is
\(Q_n-\sum_{t=1}^n\log(X_t!)\), and the estimator has approximate
covariance matrix \(I_0^{-1}/n\), where \(I_0\) is the conditional Fisher
information.} {When \(\phi_0=1\), the tail index is
\(\kappa_0=-2\omega_0/\alpha_0^2\). Define its plug-in estimator by
\(\widehat\kappa_n=-2\widehat\omega_n/\widehat\alpha_n^2\).
With \(g_0=(-2/\alpha_0^2,4\omega_0/\alpha_0^3,0)^\top\), the
delta method gives
\[
 \sqrt n(\widehat\kappa_n-\kappa_0)
 \Rightarrow
 N(0,g_0^\top I_0^{-1}g_0).
\]
}

{Finally, the restriction \(\alpha_0\neq0\) in case~(i) is needed for
identification. If \(\alpha_0=0\), the stationary conditional mean is
constant, so \(\omega_0\) and \(\beta_0\) cannot be identified separately.
In case~(ii), \(\alpha_0=1-\beta_0>0\) follows from \(|\beta_0|<1\).}

\subsection{Unit-root test}
\label{sec:unit-root-test}

{
Under the assumptions of Theorem~\ref{thm:affine-qmle}, consider a
one-sided Wald test of
\[
 H_0:\ \phi_0=1,\quad \omega_0<0,\quad \kappa_0>1,
 \qquad\text{against}\qquad
 H_1:\ \phi_0<1.
\]
The null consists of stationary unit-root models with finite mean.
At \(\phi_0=1\), no stationary solution exists when \(\omega_0\geq0\)
(Theorem~\ref{thm:affine-boundary}); when \(\omega_0<0\) and
\(0<\kappa_0\leq1\), stationary solutions have infinite mean and are
not covered by Theorem~\ref{thm:affine-qmle}.

Let \(v_{\phi,0}:=\sigma_0^2e_\phi^\top I_0^{-1}e_\phi\)
denote the asymptotic variance of
\(\sqrt n(\widehat\phi_n-\phi_0)\), and let
\(\widehat v_{\phi,n}:=\widehat\sigma_n^2
e_\phi^\top\widehat I_{Q,n}^{-1}e_\phi\)
be its plug-in estimator.
The estimated standard error of \(\widehat\phi_n\) is
\(\sqrt{\widehat v_{\phi,n}/n}\). Define
\[
 T_n=\frac{\sqrt n(\widehat\phi_n-1)}
           {\sqrt{\widehat v_{\phi,n}}}.
\]
Under these null conditions,
\[
 T_n\Rightarrow N(0,1).
\]
{The one-sided \(p\)-value is therefore \(p_n=\Phi(T_n)\),
where \(\Phi\) is the standard normal distribution function. The test
with nominal level \(\zeta\in(0,1)\) rejects when \(p_n<\zeta\), that
is, when \(T_n<\Phi^{-1}(\zeta)\). At each fixed parameter in
\(H_0\), its rejection probability converges to \(\zeta\).}
}
{For every fixed parameter \(\theta_0\) with \(\phi_0<1\)
satisfying the assumptions of Theorem~\ref{thm:affine-qmle}(i),
\(T_n\to-\infty\) in probability, so the test is consistent.}

\subsection{Local alternatives to the unit root}
\label{sec:rootn-local-unit-root}
To study power near \(\phi=1\), we consider conditionally Poisson local
alternatives
\begin{equation}
 \theta_n=(\omega_0,\alpha_0,\beta_n)^\top,
 \qquad
 \beta_n=\beta_0-\frac{h}{\sqrt n},
 \qquad
 \phi_n=1-\frac{h}{\sqrt n},
 \label{eq:rootn-local-parameter}
\end{equation}
where \(h>0\) and
\[
\theta_0=(\omega_0,\alpha_0,\beta_0)^\top,
\qquad \omega_0<0,\qquad
\alpha_0+\beta_0=1.
\]
{For each sufficiently large \(n\), consider the stationary
solution \(((M_{n,t},X_{n,t}))_{t\in\mathbb Z}\) at \(\theta_n\).
Observe \(X_{n,0},\ldots,X_{n,n}\), using \(X_{n,0}\) as the presample
observation. Compute \(\widehat\theta_n\), \(\widehat I_{Q,n}\),
\(\widehat\sigma_n^2\), and \(T_n\) as in
Sections~\ref{sec:poisson-qmle}--\ref{sec:unit-root-test}, replacing
\(X_t\) by \(X_{n,t}\) and using the same fixed positive filter
initialization for every \(\theta\).}
The following theorem gives the corresponding local-power limit.

\begin{theorem}
\label{thm:rootn-local-power}
{Suppose that the assumptions of Theorem~\ref{thm:affine-qmle}
hold at \(\theta_0\) in case~(ii), and that the conditional distribution
is Poisson (so \(\sigma_0^2=1\)).} Then
\begin{equation}
 \sqrt n(\widehat\theta_n-\theta_n)
 \Rightarrow N(0,I_0^{-1}),
 \qquad
 \widehat I_{Q,n}\xrightarrow{p}I_0,
 \qquad
 \widehat\sigma_n^2\xrightarrow{p}1.
 \label{eq:rootn-local-estimator}
\end{equation}
{
The Wald statistic defined in Section~\ref{sec:unit-root-test} satisfies
\begin{equation}
 T_n\Rightarrow
 N\left(-\frac{h}{\sqrt{v_{\phi,0}}},1\right),
 \label{eq:rootn-local-wald}
\end{equation}
{with \(v_{\phi,0}=e_\phi^\top I_0^{-1}e_\phi\), and the test with nominal
level \(\zeta\) of Section~\ref{sec:unit-root-test} has limiting power}
\begin{equation}
 {\Phi\left(\Phi^{-1}(\zeta)+\frac{h}{\sqrt{v_{\phi,0}}}\right).}
 \label{eq:rootn-local-power}
\end{equation}
}
\end{theorem}

The alternatives in \eqref{eq:rootn-local-parameter} approach the stationary
{unit root} from below.  Their \(n^{-1/2}\) distance matches the
QMLE rate and produces the mean shift
{\(-h/\sqrt{v_{\phi,0}}\)} in
\eqref{eq:rootn-local-wald}.

{Other localizations concern different limiting models.
For affine processes with a positive intercept, localizations of order
\(n^{-1}\) around a nonstationary unit root lead to diffusion limits
\citep{michel2020limiting,francq2026affine}.
Corresponding limit theory for the softplus model and the remaining
parameter regimes is left for future work.}

\section{Monte Carlo evidence}
\label{sec:inference-simulation}
In this section, we examine the finite-sample behavior of the Poisson QMLE and the
{one-sided Wald test}.
Unless otherwise stated, all Monte Carlo experiments in this section and
Appendix~\ref{app:affine-simulations} use 10,000 independent replications per
design and sample size.
{At this replication count, a rejection frequency near 5\%
has a Monte Carlo standard error of approximately 0.22 percentage points.}
For the remainder of the main text, we omit the subscript \(n\) from
estimators when the sample size is clear from context.

\subsection{Poisson experiments}

{We use three designs with \(\phi<1\) (S1--S3) and three unit-root
designs (B0--B2). S1 is a conventional benchmark,} S2 approaches the unit root,
and S3 lies in the {enlarged region}.  The {unit-root} designs cover
infinite variance (B0), long memory (B1), and short memory (B2).
Table~\ref{tab:mc-poisson-designs} records their coefficients and moment regimes.

\begin{table}[htbp]
\centering
\caption{Poisson simulation designs.}
\label{tab:mc-poisson-designs}
\small
\setlength{\tabcolsep}{4pt}
\begin{tabular}{lrrrrrl}
\toprule
Design & $\omega$ & $\alpha$ & $\beta$ & $\phi$ & {Tail index $\kappa$} & Role \\
\midrule
S1 & 1 & 0.2 & 0.7 & 0.90 & --- & {Conventional benchmark} \\
S2 & 0.1 & 0.2 & 0.79 & 0.99 & --- & {Near the unit root from below} \\
S3 & 0.5 & 1.2 & $-0.25$ & 0.95 & --- & {Enlarged region} \\
B0 & $-0.1875$ & 0.5 & 0.5 & 1 & 1.5 & Infinite variance \\
B1 & $-0.3125$ & 0.5 & 0.5 & 1 & 2.5 & Long memory \\
B2 & $-0.5$ & 0.5 & 0.5 & 1 & 4 & Short memory \\
\bottomrule
\end{tabular}
\begin{minipage}{0.98\textwidth}\footnotesize
\emph{Notes:} {S labels designs below the unit root, $\phi<1$;
B (for boundary) labels stationary unit-root designs, with $\phi=1$ and
$\omega<0$.}  S3 lies outside the condition
$\max(0,\alpha)+\max(0,\beta)<1$.  {For S1--S3, all positive-order moments are finite;
the tail index \(\kappa\) is reported only for the stationary unit-root
designs B0--B2.}
\end{minipage}
\end{table}

All designs use $c=1$. Each simulated path starts at $M_0=1$ before burn-in,
and its estimation filter is reset to one at the start of the retained sample.
Estimation uses the compact parameter set
\[
 \Theta_{\mathrm{MC}}=
 \left\{(\omega,\alpha,\beta)^\top:
 -10\leq\omega\leq10,\ -1.5\leq\alpha+\beta\leq1.5,
 \ |\beta|\leq0.999\right\},
\]
which contains parameters on both sides of $\phi=1$.
{
{Each simulated path is fitted using multistart optimization
followed by numerical refinement. For an interior solution, the largest
absolute component of the average quasi-score in \((\omega,\phi,\beta)\)
coordinates must not exceed \(10^{-4}\); a fit failing this criterion is
recorded as a fitting failure.}
\par}
{
Burn-in is 20,000 observations for $\phi<1$ and 250,000 for $\phi=1$.
Estimation and coverage use the same replications, with retained sample
sizes {$n\in\{500,\;2000,\;10000\}$} for S1--S3, B1, and B2, and
{$n\in\{500,\;2000,\;8000,\;32000\}$} for B0. {Additional B1--B2
estimation runs use \(n=8{,}000\) and \(32{,}000\); the latter provide QQ
comparisons across all three unit-root designs.} The Poisson {one-sided Wald test}
experiment uses separate replications at {$n\in\{500,\;2000,\;10000\}$}.
In the Poisson experiments, standard errors are computed from the
covariance estimate
\(\widehat\sigma^2\widehat I_Q^{-1}/n\) defined in
Section~\ref{sec:poisson-qmle}.
\par}

\begin{table}[!tb]
\centering
\setstretch{1}
\small
\setlength{\tabcolsep}{9pt}
\caption{Finite-sample performance of the Poisson QMLE of \(\phi\) under Poisson observations.}
\label{tab:mc-poisson-estimation}
\begin{tabular}{llrrrrr}
\toprule
Design & $n$ & Bias & RMSE & ESD & ASE & Cov. \\
\midrule
S1 & 500 & -0.085 & 0.113 & 0.074 & 0.052 & 0.748 \\
S1 & 2,000 & -0.021 & 0.030 & 0.022 & 0.020 & 0.874 \\
S1 & 10,000 & -0.004 & 0.009 & 0.008 & 0.008 & 0.938 \\
\addlinespace[2pt]
S2 & 500 & -0.024 & 0.034 & 0.024 & 0.015 & 0.761 \\
S2 & 2,000 & -0.004 & 0.007 & 0.005 & 0.005 & 0.902 \\
S2 & 10,000 & -0.001 & 0.002 & 0.002 & 0.002 & 0.938 \\
\addlinespace[2pt]
S3 & 500 & -0.009 & 0.025 & 0.023 & 0.022 & 0.948 \\
S3 & 2,000 & -0.002 & 0.011 & 0.011 & 0.010 & 0.949 \\
S3 & 10,000 & 0.000 & 0.005 & 0.005 & 0.005 & 0.948 \\
\addlinespace[2pt]
B0 & 500 & -0.019 & 0.040 & 0.035 & 0.032 & 0.969 \\
B0 & 2,000 & -0.007 & 0.014 & 0.013 & 0.012 & 0.968 \\
B0 & 8,000 & -0.002 & 0.006 & 0.005 & 0.005 & 0.970 \\
B0 & 32,000 & -0.001 & 0.002 & 0.002 & 0.002 & 0.968 \\
\addlinespace[2pt]
B1 & 500 & -0.025 & 0.063 & 0.058 & 0.054 & 0.971 \\
B1 & 2,000 & -0.008 & 0.024 & 0.023 & 0.023 & 0.966 \\
B1 & 10,000 & -0.002 & 0.009 & 0.009 & 0.009 & 0.961 \\
\addlinespace[2pt]
B2 & 500 & -0.036 & 0.116 & 0.111 & 0.105 & 0.956 \\
B2 & 2,000 & -0.011 & 0.047 & 0.045 & 0.046 & 0.961 \\
B2 & 10,000 & -0.003 & 0.019 & 0.019 & 0.019 & 0.955 \\
\bottomrule
\end{tabular}
\par\smallskip
\begin{minipage}{0.92\textwidth}\footnotesize
\emph{Notes:} Designs are defined in Table~\ref{tab:mc-poisson-designs}.
Bias and RMSE refer to \(\widehat\phi\). ESD is the empirical standard
deviation of \(\widehat\phi\), ASE is its mean estimated standard error,
and Cov. is the empirical coverage of the nominal 95\% Wald confidence
interval for \(\phi\). Estimation uses the compact parameter set stated
in the text, without imposing \(\phi\leq1\). Table~\ref{tab:mc-poisson-parameters}
in Appendix~\ref{app:affine-simulations} reports bias and RMSE for
\(\widehat\omega\), \(\widehat\alpha\), and \(\widehat\beta\).
\end{minipage}
\end{table}

{Table~\ref{tab:mc-poisson-estimation} shows that the bias
and RMSE of \(\widehat\phi\) decrease with the sample size in every
design, and that \(\phi\) is underestimated in small samples. Coverage in
S1 and S2 approaches the nominal level from below, while S3 remains close
to nominal throughout. Intervals for the unit-root designs are generally
conservative, particularly in B0, even at $n=32{,}000$.
Table~\ref{tab:mc-poisson-parameters} in
Appendix~\ref{app:affine-simulations} reports the bias and RMSE of
\(\widehat\omega\), \(\widehat\alpha\), and \(\widehat\beta\). In small
samples, the intercept tends to be overestimated.}

For S3, the studentized estimation errors broadly follow the standard
normal reference; the QQ plots are reported in
Figure~\ref{fig:mc-qq-s3} in Appendix~\ref{app:affine-simulations}.
Figure~\ref{fig:mc-qq-boundary} {compares
estimation errors divided by their estimated standard errors with
{standard normal} quantiles.}
{At $n=32{,}000$, the QQ curves for B1 ($\kappa=2.5$) and
B2 ($\kappa=4$) broadly follow the standard normal reference.
For B0 ($\kappa=1.5$),} the studentized intercept
errors retain a small positive shift, while the studentized
$\widehat\phi$ errors show a left shift and compression (mean {$-0.18$}
and standard deviation {0.91}). The latter pattern is consistent with the
conservative B0 coverage in Table~\ref{tab:mc-poisson-estimation}.

\begin{figure}[!tb]
\centering
\includegraphics[width=0.85\textwidth]{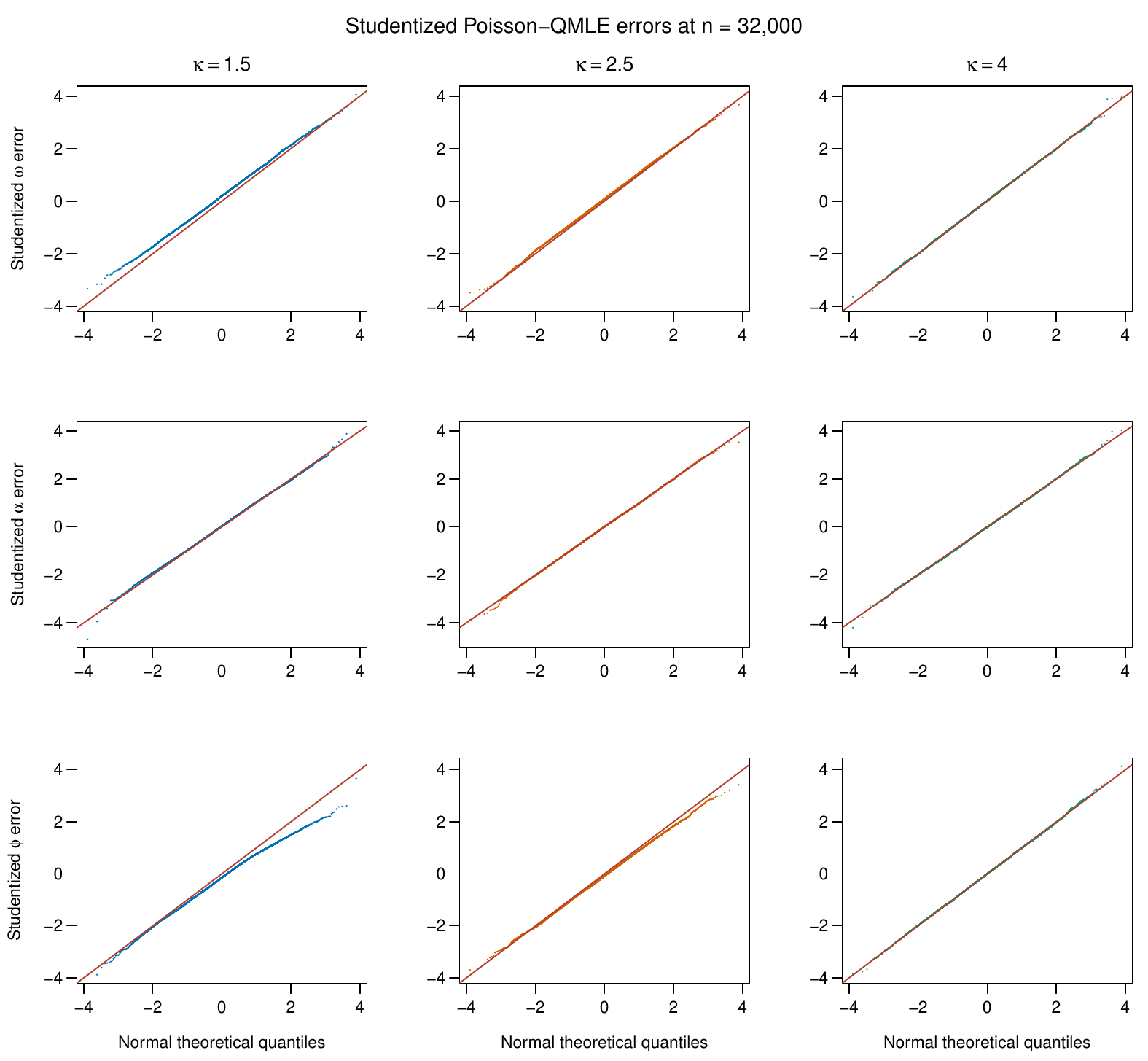}
\caption{Studentized {Poisson QMLE} errors at the stationary unit root for
$\kappa=1.5$, $2.5$, and $4$ at $n=32{,}000$.  Columns index the three
moment regimes and rows report $\omega$, $\alpha$, and the derived quantity
$\phi=\alpha+\beta$. Each panel uses the {standard normal} 45-degree
reference line.}
\label{fig:mc-qq-boundary}
\end{figure}

\subsection{Cross-family comparison}
\label{sec:mc-cross-family}

{
Table~\ref{tab:mc-family-comparison} reports the size and power of the
{one-sided Wald test} for Poisson observations at {\(n\in\{500,\;2000,\;10000\}\)} and for
{negative binomial} and {gamma} observations at {\(n\in\{2000,\;10000\}\)}.
The latter two sample sizes allow direct comparisons across all three
conditional families. {The negative binomial design uses
probability parameter \(p=0.5\), and the gamma design uses scale parameter
\(\tau=2\). Both therefore satisfy
\(\Var(X_t\mid\mathcal F_{t-1})=2M_t\).} Their {unit-root} intercepts
are chosen to match the Poisson values of \(\kappa\), so the comparison
holds the {tail index} fixed rather than the intercept.
{For each conditional family, we obtain alternatives from B2
by retaining its intercept and \(\alpha=0.5\), and lowering \(\beta\)
from \(0.5\) to \(0.49\), \(0.47\), or \(0.45\).}
\par}

\begin{table}[!htbp]
\centering
\caption{Size and power of the {one-sided Wald test at the
nominal 5\% level} across conditional families.}
\label{tab:mc-family-comparison}
\scriptsize
\setlength{\tabcolsep}{4pt}
\begin{tabular}{lrrrrrrr}
\toprule
& & \multicolumn{2}{c}{Size} & \multicolumn{4}{c}{Power} \\
\cmidrule(lr){3-4}\cmidrule(lr){5-8}
Family & \(n\)
& B1 & B2
& B2, \(\phi=.99\) & B2, \(\phi=.97\) & B2, \(\phi=.95\)
& S3, \(\phi=.95\) \\
\midrule
Poisson & 500 & 0.049 & 0.041 & 0.054 & 0.068 & 0.084 & 0.944 \\
& 2,000 & 0.055 & 0.042 & 0.067 & 0.124 & 0.187 & 1.000 \\
& 10,000 & 0.054 & 0.050 & 0.114 & 0.354 & 0.607 & 1.000 \\
\addlinespace
Negative binomial & 2,000 & 0.044 & 0.043 & 0.052 & 0.075 & 0.093 & 1.000 \\
& 10,000 & 0.048 & 0.046 & 0.072 & 0.139 & 0.242 & 1.000 \\
\addlinespace
Gamma & 2,000 & 0.043 & 0.041 & 0.051 & 0.071 & 0.095 & 1.000 \\
& 10,000 & 0.048 & 0.046 & 0.072 & 0.155 & 0.285 & 1.000 \\
\bottomrule
\end{tabular}
\begin{minipage}{0.96\textwidth}\footnotesize
\emph{Notes:} Each configuration uses 10,000 independent replications of a
series with \(n\) retained observations. The null is \(H_0:\phi=1\)
against \(H_1:\phi<1\). The size columns use B1 and B2, with
\(\kappa=2.5\) and \(4\), respectively{. The} power columns use fixed
alternatives. The {B2 alternatives} keep \(\alpha=0.5\) and set
\(\beta=0.49,0.47,0.45\), with \(\omega=-0.5\) for Poisson observations
and \(\omega=-1\) for the other families. S3 uses the coefficients in
Table~\ref{tab:mc-poisson-designs} for every family. Optimization is
unrestricted across \(\phi=1\), while retaining \(|\beta|<1\).
Standard errors are computed from \(\widehat\sigma^2\widehat I_Q^{-1}/n\).
\end{minipage}
\end{table}

{
Null rejection frequencies remain close to the nominal level across families
and sample sizes. Power against the {B2 alternatives} increases with
sample size and distance from the null. At \(\phi=0.95\), power is
substantially lower under the {B2 alternative} than under S3.
At the common sample sizes, power against the {B2 alternatives} is
lower for {negative binomial} and {gamma} observations than for Poisson
observations. Under S3, the Poisson test rejects in 94.4\% of replications
at \(n=500\), and all three families give rejection in every replication
at \(n=2{,}000\) and \(n=10{,}000\).
\par}
At \(n=10{,}000\), under the B2 alternatives with
\(\phi=0.95\), the non-Poisson samples have roughly half the average
sample mean and nearly twice the standard error of the Poisson samples.\footnote{Average
sample means are {0.97, 0.51, and 0.51} for Poisson,
{negative binomial}, and {gamma} observations, respectively. The corresponding
average estimated standard errors are {0.027, 0.052, and 0.047}.}
Their larger {dispersion parameter} and the more negative intercept required to
match \(\kappa\) help explain the lower power. {Table~\ref{tab:mc-affine-estimation}
in Appendix~\ref{app:affine-simulations} shows that estimated standard errors
track empirical standard deviations in the non-Poisson estimation
experiments. Coverage is slightly below nominal in S1 and above nominal
in B1.}

{
\subsection{Impact of \texorpdfstring{{misspecification}}{misspecification} of the softplus scale}
\label{sec:mc-c-misspecification}

The preceding experiments correctly fix \(c=1\) in both simulation and
estimation.  We now isolate what happens when the fitted scale is too large.
The data-generating process remains conditionally Poisson with true \(c_0=1\)
and
\[
 (\omega_0,\alpha_0,\beta_0)=(0.5,0.1,0.85),
 \qquad \phi_0=0.95.
\]
For the same data-generating process (DGP), we fit the model using different
values of \(c\in\{1,2,3,4\}\).  We conduct two experiments at different sample sizes.
{The long-path experiment uses one series with 200,000 retained
observations; the short-path experiment uses 10,000 independent series,
each with \(n=1{,}500\) retained observations.}  Table~\ref{tab:mc-c-misspecification} reports the estimates.

\begin{table}[H]
\centering
\caption{{QMLE and unit-root test for different fitted softplus
scales (true \(c_0=1\), \(\phi_0=0.95\)).}}
\label{tab:mc-c-misspecification}
\small
\setlength{\tabcolsep}{4pt}
\resizebox{\textwidth}{!}{%
\begin{tabular}{rrrrrrrr}
\toprule
Fitted \(c\) & Long-path \(\widehat\omega\) & Long-path \(\widehat\phi\)
 & Median \(\widehat\phi\) & {Share \(\widehat\phi>1\)}
 & {CI contains 1} & {CI above 1}
 & {Non-rejection rate}\\
\midrule
1 & {0.509} & 0.9490 & 0.9468 & 0.000 & 0.000 & 0.000 & 0.000 \\
2 & {0.424} & 0.9560 & 0.9536 & 0.000 & 0.021 & 0.000 & 0.003 \\
3 & {0.043} & 0.9840 & 0.9816 & 0.098 & 0.923 & 0.002 & 0.821 \\
4 & {$-0.658$} & 1.0303 & 1.0278 & 0.894 & 0.621 & 0.379 & 0.999 \\
\bottomrule
\end{tabular}%
}
\begin{minipage}{0.98\textwidth}\footnotesize
\emph{Notes:} {Long-path columns report QMLEs from one simulated
series with 200,000 retained observations. The remaining columns summarize
10,000 independent series with \(n=1{,}500\).} {The non-rejection rate is the proportion of
short-path replications in which the one-sided Wald test of
\(H_0:\phi=1\) against \(H_1:\phi<1\) does not reject at the nominal
5\% level.} {The columns ``Share \(\widehat\phi>1\)'', ``CI
contains 1'' and ``CI above 1'' give the proportions of short-path
replications in which \(\widehat\phi>1\), in which the nominal 95\%
Wald interval for \(\phi\) contains one, and in which the whole interval
lies above one.}
At $c=4$, 7 short-path fits reach the upper bound
{$\widehat\beta=0.999$.}
\end{minipage}
\end{table}

At \(c=1\), the long-path estimate closely recovers \(\phi_0=0.95\), and
the test rejects the unit-root null in every short-path replication.
Increasing the fitted scale lowers the estimated intercept and pushes
{the estimated coefficient sum \(\widehat\phi\) toward, then above, one.} At \(c=4\), {the non-rejection rate}
reaches 99.9\%. {A fitted scale that is too large can
therefore make the test retain the unit-root null when the true
coefficient sum is below one.}

}

\section{Empirical applications}
\label{sec:empirical-applications}

We apply the proposed estimation and testing procedures to monthly
recorded-crime counts in New South Wales and one-second TRX/USDT
transaction counts. {For crime counts, we examine whether
prolonged periods of elevated recorded incidents are compatible with
stationary unit-root dynamics. For transaction counts, we examine whether
the unit-root conclusion for trading activity is robust to the softplus
scale.}

Each specification fixes the softplus scale and estimates the dynamics
by Poisson quasi-maximum likelihood without imposing \(\phi\leq1\).
Standard errors use the covariance estimator from
Section~\ref{sec:poisson-qmle}, with estimated dispersion
\(\widehat\sigma^2\). Write \(Q_n(\theta;c)\) for the criterion in
\eqref{eq:affine-qmle-criterion} evaluated at dynamics parameter \(\theta\)
and fixed scale \(c\), and let \(\widehat\theta(c)\) denote the
corresponding fitted parameter.  Define
\[
 Q_n(c):=Q_n(\widehat\theta(c);c).
\]
For the scale grid \(\mathcal C\) displayed in the corresponding table,
we report \(\Delta Q_n(c)=Q_n(c)-\max_{d\in\mathcal C}Q_n(d)\).
{\(\Delta Q_n\) compares fit descriptively, and the
\(p\)-values are unadjusted for inspection of the scale grid.}

\subsection{NSW recorded-crime counts}
\label{sec:nsw-crime-application}

{We fix the softplus scale at \(c=1\) and apply the}
Poisson QMLE and the one-sided Wald test of
Sections~\ref{sec:poisson-qmle} and \ref{sec:unit-root-test} to the New
South Wales Bureau of Crime Statistics and Research (BOCSAR) workbook
\emph{Recorded Criminal Incidents by month--by LGA}{, which
tabulates incidents by local government area (LGA)}.
{\citet{gorgi2020betanegative} analyzes monthly Sydney
narcotics-trafficking counts from an earlier release of the BOCSAR data
using beta--negative binomial autoregressions.}
We use the March 2026 release from BOCSAR's {``Criminal incident
data'' page.}\footnote{\url{https://bocsar.nsw.gov.au/statistics-dashboards/open-datasets/criminal-offences-data.html}}
It contains monthly counts from January 1995 through March 2026, giving
$n=375$ observations for each of 62 {offense}-category or subcategory rows
and 131 LGA labels.  The initial panel therefore contains
$131\times62=8{,}122$ series.  Category totals and detailed subcategories
are both retained.  Some LGA labels are administrative rather than
geographical, and rare or recently introduced {offenses} can contain long
runs of zeros. {The data-quality conditions below screen for
sparse counts and possible gaps in coverage.}

{For each series, let \(\bar X=n^{-1}\sum_{t=1}^nX_t\)
denote its sample mean. Each filter is initialized at \(\bar X\).}
{Computation uses
$-1.5\leq\phi\leq1.5$, $|\beta|\leq0.999$, and
$|\omega|\leq200\max(\bar X,1)$.}  We call an output
\emph{finite} when the optimizer returns finite parameter estimates, a finite
criterion value, and {finite, strictly positive filtered means
at every observation used in the criterion.}  {We flag a finite
fit if any parameter estimate lies within \(10^{-5}\) of a computational
bound.  A \emph{clean numerical fit} is unflagged, and the largest absolute
component of its average quasi-score, evaluated in \((\omega,\phi,\beta)\)
coordinates, is at most \(10^{-4}\).}
A series passes the quality screen if it has a clean numerical fit and satisfies
\[
 \bar X\geq5,
 \qquad \frac1n\sum_t\mathbf 1\{X_t=0\}\leq0.20,
\]
with neither a leading nor a trailing zero run longer than 12 months.
{The conditions on the sample mean, zero frequency, and zero
runs screen for sparse counts and possible gaps in coverage.}

Of the 8,122 workbook series, 1,751 pass both numerical and data-quality
screening (Table~\ref{tab:crime-screening-funnel}).  {Two finite fits have singular estimated information matrices,
so their Wald standard errors cannot be computed.}  We use the 1,751
quality-screened series as the {analysis sample} below.
{Among the 1,751 quality-screened series, 1,082 (61.8\%)
have \(|\widehat\phi-1|\leq0.02\).}  At \(c=1\),
{the one-sided Wald test does not reject at the nominal
5\% level for 73.4\% of the screened series} (Table~\ref{tab:crime-near-unit-funnel}).
{The table reports two-sided 95\% confidence intervals and
non-rejection counts and percentages for \(H_0:\phi=1\) against
\(H_1:\phi<1\) at the nominal 5\% level.}

\begin{table}[htbp]
\centering
\caption{{Numbers of series by numerical and data-quality
status in the NSW crime panel.}}
\label{tab:crime-screening-funnel}
\begin{tabular}{lr}
\toprule
Outcome & Number of series\\
\midrule
Workbook rows & 8,122\\
Constant series, not fitted & 364\\
Finite optimizer outputs & 7,758\\
Valid Wald standard errors & 7,756\\
{Fits near computational bounds} & 3,952\\
Clean numerical fits & 3,806\\
Quality-screened series & 1,751\\
\bottomrule
\end{tabular}
\end{table}

\begin{table}[htbp]
\centering
\caption{{Estimates of \(\phi\) and unit-root test results
across NSW crime samples at \(c=1\).}}
\label{tab:crime-near-unit-funnel}
\small
\begin{tabular}{lrrrr}
\toprule
{Series group} & $N$ & 95\% CI contains 1
 & \shortstack{{Non-rejections,}\\{count (\%)}}
 & $|\widehat\phi-1|\leq0.02$\\
\midrule
All valid inference outputs & 7,756 & 6,361 (82.0\%)
 & 7,107 (91.6\%) & 1,468 (18.9\%)\\
Clean numerical fits & 3,806 & 2,780 (73.0\%)
 & 3,181 (83.6\%) & 1,461 (38.4\%)\\
Quality-screened series & 1,751 & 1,426 (81.4\%)
 & 1,285 (73.4\%) & 1,082 (61.8\%)\\
\bottomrule
\end{tabular}
\end{table}

{
{We illustrate scale sensitivity using Wollondilly / Theft /
Break and enter dwelling counts, selected from 44 series whose fits at
\(c=1\) were consistent with a stationary unit root in both the full and
pre-2020 samples. We examine how estimated persistence and its confidence
interval respond to \(c\). Appendix~\ref{app:crime-series-selection}
gives the selection criteria.}
Its 375 monthly counts show substantial serial dependence.\footnote{The
sample mean and variance are {11.17} and {53.89}{. Autocorrelations} at lags 1,
5, and 10 are {0.69, 0.60, and 0.59}.}
Table~\ref{tab:crime-wollondilly-c-grid} reports parameter estimates and
{one-sided Wald test} results across the fixed-scale grid.
{At \(c=1\), \(\widehat\sigma^2={1.90}\) gives
\(\widehat\kappa=-2\widehat\omega/(\widehat\alpha^2\widehat\sigma^2)={2.08}\).
Conditional on \(\phi=1\), this point estimate falls in the finite-variance,
long-memory regime{. We have not obtained the limiting
distribution of \(\widehat\sigma^2\), so we cannot compute an asymptotic
confidence interval for \(\kappa\).}}
{At the nominal 5\% level, the one-sided Wald test does not
reject at any displayed scale.}  The two-sided
interval also contains one for every \(c\leq2\), including all grid values
within seven quasi-log-likelihood units of the best fit at \(c=1\).
At \(c=4\), the interval lies above one{. The one-sided Wald test still does not reject},
but this fit loses 36.5 criterion units.
}

\begin{table}[htbp]
\centering
\caption{Fixed-scale sensitivity for Wollondilly / Theft / Break and enter dwelling counts.}
\label{tab:crime-wollondilly-c-grid}
\scriptsize
\setlength{\tabcolsep}{3.2pt}
\resizebox{\textwidth}{!}{%
\begin{tabular}{rrrrrrlrr}
\toprule
\(c\) & \(\widehat\omega\) & \(\widehat\alpha\) & \(\widehat\beta\)
 & \(\widehat\phi\) & {\(\mathrm{SE}(\widehat\phi)\)} & {95\% CI for \(\phi\)}
 & \shortstack{{One-sided}\\{\(p\)-value}} & \(\Delta Q_n\)\\
\midrule
0.125 & {0.010} & {0.127} & {0.870} & 0.9966 & 0.0053
 & [0.9862, 1.0070] & 0.260 &  $-0.811$\\
0.250 & {0.010} & {0.127} & {0.870} & 0.9966 & 0.0053
 & [0.9862, 1.0070] & 0.260 &  $-0.811$\\
0.500 & {0.009} & {0.127} & {0.870} & 0.9967 & 0.0053
 & [0.9863, 1.0071] & 0.268 &  $-0.760$\\
1.000 & {$-0.031$} & {0.125} & {0.874} & 0.9997 & 0.0055
 & [0.9889, 1.0104] & 0.476 &   0.000\\
1.500 & {$-0.115$} & {0.132} & {0.874} & 1.0056 & 0.0062
 & [0.9936, 1.0177] & 0.820 &  $-1.281$\\
2.000 & {$-0.234$} & {0.153} & {0.860} & 1.0131 & 0.0074
 & [0.9985, 1.0276] & 0.961 &  $-6.629$\\
4.000 & {$-1.246$} & {0.282} & {0.782} & 1.0640 & 0.0143
 & [1.0360, 1.0919] & 1.000 & $-36.481$\\
\bottomrule
\end{tabular}%
}
\begin{minipage}{0.97\textwidth}\footnotesize
\emph{Notes:} {One-sided \(p\)-values use the standard normal
approximation \(p_n=\Phi(T_n)\) for the test of \(H_0:\phi=1\) against
\(H_1:\phi<1\) (Section~\ref{sec:unit-root-test}).}
\end{minipage}
\end{table}

{
\subsection{Binance TRX/USDT transaction counts}
\label{sec:binance-application}

{We analyze the number of TRX/USDT trades executed on Binance
in each one-second interval on May 30, 2026.}\footnote{Counts
are taken from the number-of-trades column of Binance's one-second market
records: \url{https://data.binance.vision/data/spot/daily/klines/TRXUSDT/1s/TRXUSDT-1s-2026-05-30.zip}}
{TRX is the native token of the TRON network, and USDT is
Tether's US-dollar stablecoin.}
The sample contains 86,400 observations spanning a complete day, with a
sample variance-to-mean ratio of 22.35.\footnote{The sample mean and
variance are {0.87} and {19.36}{. Autocorrelations} at lags 1, 5, and 10 are
{0.16, 0.12, and 0.12}.}
Table~\ref{tab:binance-trx-c-grid} reports the QMLE and {one-sided Wald test}
over a grid of fixed softplus scales.

\begin{table}[htbp]
\centering
\caption{Fixed-scale sensitivity for one-second TRX/USDT transaction counts.}
\label{tab:binance-trx-c-grid}
\scriptsize
\setlength{\tabcolsep}{3.2pt}
\resizebox{\textwidth}{!}{%
\begin{tabular}{rrrrrrlrr}
\toprule
\(c\) & \(\widehat\omega\) & \(\widehat\alpha\) & \(\widehat\beta\)
 & \(\widehat\phi\) & {\(\mathrm{SE}(\widehat\phi)\)} & {95\% CI for \(\phi\)}
 & \shortstack{{One-sided}\\{\(p\)-value}} & \(\Delta Q_n\)\\
\midrule
0.0625 & {0.037} & {0.059} & {0.897} & 0.9568 & 0.0038
 & [0.9494, 0.9641] & $<10^{-12}$ &     0.0\\
0.125 & {0.036} & {0.060} & {0.897} & 0.9572 & 0.0039
 & [0.9496, 0.9647] & $<10^{-12}$ &   $-18.1$\\
0.250 & {0.016} & {0.066} & {0.895} & 0.9605 & 0.0043
 & [0.9520, 0.9690] & $<10^{-12}$ &  $-252.7$\\
0.500 & {$-0.111$} & {0.086} & {0.883} & 0.9696 & 0.0056
 & [0.9587, 0.9806] & $2.9\times10^{-8}$ &  $-928.5$\\
1.000 & {$-0.647$} & {0.119} & {0.892} & 1.0106 & 0.0056
 & [0.9996, 1.0216] & 0.971 & $-2020.6$\\
2.000 & {$-2.323$} & {0.190} & {0.894} & 1.0843 & 0.0057
 & [1.0731, 1.0956] & 1.000 & $-4357.1$\\
4.000 & {$-6.984$} & {0.319} & {0.894} & 1.2132 & 0.0072
 & [1.1990, 1.2273] & 1.000 & $-7660.7$\\
\bottomrule
\end{tabular}%
}
\begin{minipage}{0.97\textwidth}\footnotesize
\emph{Notes:} {One-sided \(p\)-values use the standard normal
approximation \(p_n=\Phi(T_n)\) for the test of \(H_0:\phi=1\) against
\(H_1:\phi<1\) (Section~\ref{sec:unit-root-test}).}
\end{minipage}
\end{table}

{At the nominal 5\% level, the one-sided Wald test rejects
for \(c\leq0.5\) across the displayed grid}, but not at \(c=1\){. Larger} scales place the entire confidence interval above
one.  The quasi-likelihood favors the smallest displayed scale,
\(c=0.0625\).  {Halving the smallest displayed scale leaves
\(\widehat\phi\) essentially unchanged.}\footnote{\interlinepenalty=10000
{At \(c=0.0625\), the fitted linear index has minimum {0.37},
so the softplus correction is at most approximately
\(1.7\times10^{-4}\) along the fitted path. The best fit on the displayed
grid therefore behaves almost like a linear INGARCH\((1,1)\) model.}
}
These small-scale fits give \(\widehat\phi\approx0.957\) and reject the
unit-root null, despite weak identification of the precise scale.  {The non-rejection at \(c=1\) occurs under a specification
with substantially lower quasi-log-likelihood than the small-scale fits.}
}

\section{Conclusion}
\label{sec:conclusion}
{
The central finding is that a unit root need not imply nonstationarity
in softplus affine models.
{For a large positive input, softplus is approximately the
identity, while it remains strictly positive for every real input.} A negative
intercept is therefore admissible, and it induces a negative drift that
returns the process from high levels. The resulting unit-root process is
stationary, yet it differs from a conventional short-memory process:
{the stationary distribution of \(M_t\)} has a power-law tail, and the observations can
exhibit long memory.

The same nonlinearity explains why conventional inference remains valid.
Because the process is stationary at the unit root, standard asymptotic
theory applies in place of nonstandard unit-root limit theory. A finite
mean suffices for the Poisson QMLE to be asymptotically normal at the usual
rate and for the {one-sided Wald test} to have a standard normal limiting
null distribution, even when the observations have infinite variance or
long memory.

The stationarity classification holds for every fixed \(c>0\), but the
simulations and applications show that {the estimated coefficient
sum \(\widehat\phi\) and the conclusions of unit-root tests} can be
sensitive to the chosen scale. We therefore recommend
reporting sensitivity across fixed scales. Joint estimation of the scale
\(c\) and \(\theta\) is left for future research. A further direction
is to extend the analysis to other strictly positive links used in
neural networks that are asymptotically linear with unit slope.
Studying these alternatives would clarify which features of stationary
unit roots and their power-law tails are shared across link functions
and which depend on the particular softplus specification.
\par}
{
\section*{Declaration of generative AI and AI-assisted technologies in the manuscript preparation process}
OpenAI Codex assisted the authors with language editing, manuscript
organization, notation clarification, discussion of mathematical arguments,
and modifications to the LaTeX source. The authors directed this work,
reviewed the proposed revisions, and edited them as appropriate. The authors
are fully responsible for the manuscript's content, including its
mathematical statements and proofs.
}

\bibliographystyle{apalike}
\bibliography{lib}

\clearpage
\appendix

\section[Additional simulation results]{Additional simulation results}
\label{app:affine-simulations}

{Table~\ref{tab:mc-poisson-parameters} complements
Table~\ref{tab:mc-poisson-estimation} with the bias and RMSE of
\(\widehat\omega\), \(\widehat\alpha\), and \(\widehat\beta\) in the
Poisson experiments of Section~\ref{sec:inference-simulation}. In small
samples, the intercept tends to be overestimated. For B0 at $n=8{,}000$,
the intercept RMSE of 0.042 is strongly influenced by one replication
exhibiting a prolonged period with large values of \(M_t\), which
accounts for 80.4\% of the total squared intercept error.}

\begin{table}[!htbp]
\centering
\setstretch{1}
\small
\setlength{\tabcolsep}{6pt}
\caption{Bias and RMSE of the Poisson QMLE of \(\omega\), \(\alpha\), and \(\beta\) under Poisson observations.}
\label{tab:mc-poisson-parameters}
\begin{tabular}{llrrrrrr}
\toprule
& & \multicolumn{2}{c}{$\widehat\omega$}
& \multicolumn{2}{c}{$\widehat\alpha$}
& \multicolumn{2}{c}{$\widehat\beta$} \\
\cmidrule(lr){3-4}\cmidrule(lr){5-6}\cmidrule(lr){7-8}
Design & $n$ & Bias & RMSE & Bias & RMSE & Bias & RMSE \\
\midrule
S1 & 500 & 0.858 & 1.133 & 0.039 & 0.053 & -0.124 & 0.149 \\
S1 & 2,000 & 0.211 & 0.304 & 0.014 & 0.022 & -0.035 & 0.047 \\
S1 & 10,000 & 0.041 & 0.093 & 0.003 & 0.008 & -0.007 & 0.015 \\
\addlinespace[2pt]
S2 & 500 & 0.251 & 0.372 & 0.049 & 0.063 & -0.072 & 0.087 \\
S2 & 2,000 & 0.043 & 0.067 & 0.015 & 0.023 & -0.019 & 0.027 \\
S2 & 10,000 & 0.008 & 0.019 & 0.003 & 0.007 & -0.004 & 0.008 \\
\addlinespace[2pt]
S3 & 500 & 0.047 & 0.182 & -0.011 & 0.048 & 0.002 & 0.046 \\
S3 & 2,000 & 0.010 & 0.085 & -0.003 & 0.023 & 0.000 & 0.023 \\
S3 & 10,000 & 0.002 & 0.038 & -0.001 & 0.010 & 0.000 & 0.010 \\
\addlinespace[2pt]
B0 & 500 & 0.055 & 0.273 & 0.001 & 0.052 & -0.021 & 0.065 \\
B0 & 2,000 & 0.016 & 0.065 & 0.001 & 0.025 & -0.007 & 0.029 \\
B0 & 8,000 & 0.006 & 0.042 & 0.000 & 0.013 & -0.003 & 0.014 \\
B0 & 32,000 & 0.002 & 0.009 & 0.000 & 0.006 & -0.001 & 0.007 \\
\addlinespace[2pt]
B1 & 500 & 0.038 & 0.110 & -0.002 & 0.058 & -0.023 & 0.088 \\
B1 & 2,000 & 0.013 & 0.047 & 0.000 & 0.028 & -0.008 & 0.038 \\
B1 & 10,000 & 0.003 & 0.019 & 0.000 & 0.013 & -0.002 & 0.016 \\
\addlinespace[2pt]
B2 & 500 & 0.037 & 0.134 & -0.004 & 0.069 & -0.031 & 0.142 \\
B2 & 2,000 & 0.011 & 0.058 & 0.000 & 0.035 & -0.011 & 0.063 \\
B2 & 10,000 & 0.002 & 0.025 & 0.000 & 0.015 & -0.002 & 0.027 \\
\bottomrule
\end{tabular}
\par\smallskip
\begin{minipage}{0.86\textwidth}\footnotesize
\emph{Notes:} Designs are defined in Table~\ref{tab:mc-poisson-designs}.
The replications and estimation procedure are those of
Table~\ref{tab:mc-poisson-estimation}.
\end{minipage}
\end{table}

\begin{figure}[p]
\centering
\includegraphics[width=0.82\textwidth]{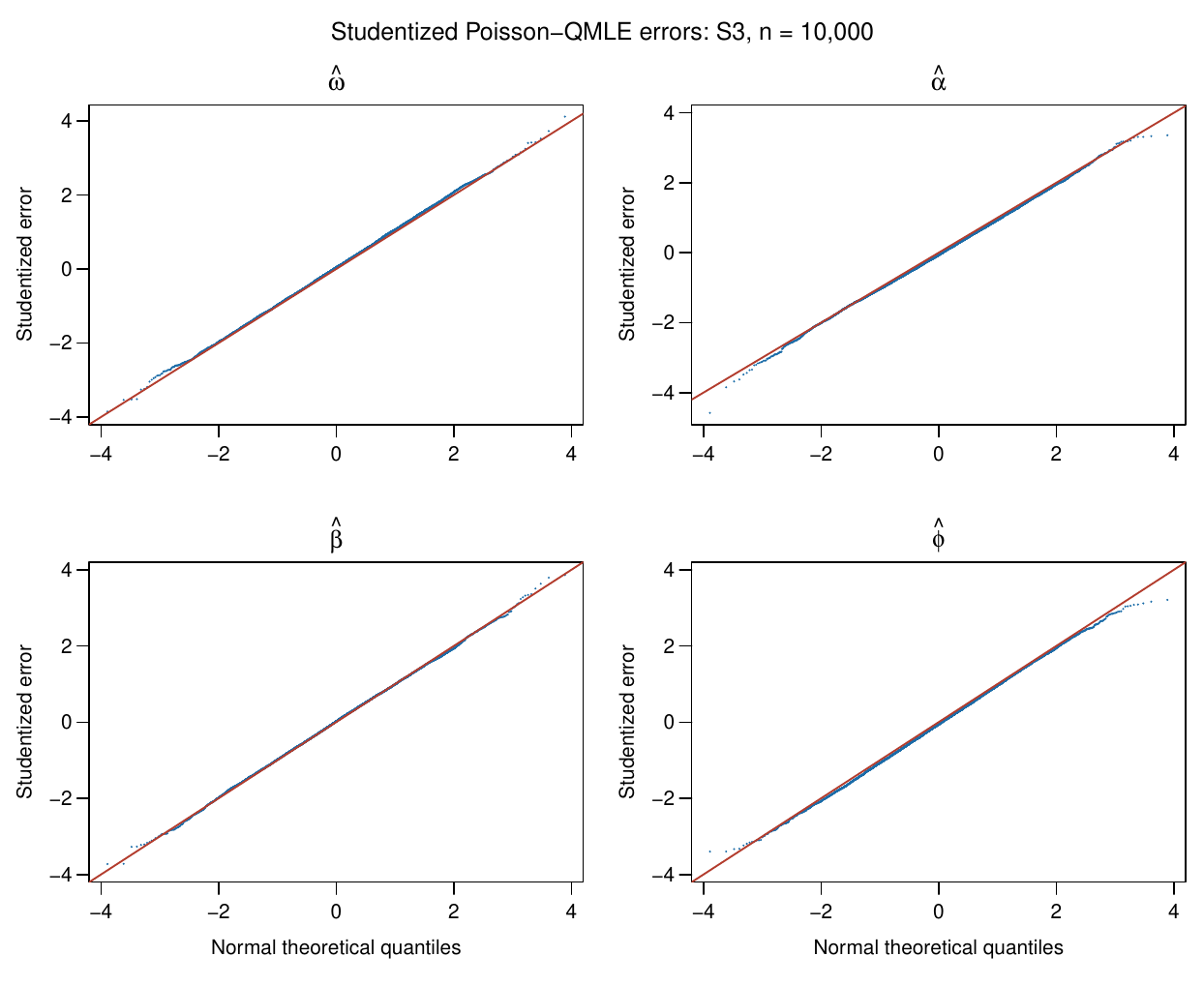}
\caption{Studentized {Poisson QMLE} errors for design S3 in the {enlarged region}
at $n=10{,}000$. Panels report $\omega$, $\alpha$, $\beta$, and
{the coefficient sum $\phi=\alpha+\beta$.} Each panel uses the
{standard normal} 45-degree reference line. {Each estimation
error is divided by its estimated standard error, computed as for the ASE
column of Table~\ref{tab:mc-poisson-estimation}.}
}
\label{fig:mc-qq-s3}
\end{figure}

{The rest of this appendix reports} additional inference results for two
non-Poisson affine conditional families. The non-Poisson designs use the negative binomial and gamma
specifications of Section~\ref{sec:mc-cross-family}. The unit-root
intercepts are $-0.625$ in B1 and $-1$ in B2, {which give the same
tail indices, $\kappa=2.5$ and $4$, as the Poisson designs.}
The B2 alternatives use \(\omega=-1\) and \(\alpha=0.5\),
with \(\beta=0.49,0.47,0.45\) for \(\phi=0.99,0.97,0.95\), respectively.
{The coefficient triples with \(\phi<1\)} are unchanged
from the corresponding Poisson designs.  We use $n=2{,}000$ and
$10{,}000$ for both estimation and the one-sided Wald test.
Table~\ref{tab:mc-naive-standardization} shows substantial overrejection
when the dispersion parameter is fixed at one.

\begin{table}[!htbp]
\centering
\setstretch{1}
\caption{Poisson QMLE inference for $\phi$ under negative binomial and gamma observations.}
\label{tab:mc-affine-estimation}
\small
\begin{tabular}{lllrrrrr}
\toprule
Family & Design & $n$ & Bias & RMSE & ESD & ASE & Cov. \\
\midrule
Negative binomial & S1 & 2,000 & -0.012 & 0.024 & 0.021 & 0.020 & 0.939 \\
Negative binomial & S1 & 10,000 & -0.002 & 0.009 & 0.008 & 0.008 & 0.945 \\
Negative binomial & S3 (enlarged region) & 2,000 & -0.003 & 0.013 & 0.013 & 0.013 & 0.947 \\
Negative binomial & S3 (enlarged region) & 10,000 & -0.001 & 0.006 & 0.006 & 0.006 & 0.956 \\
Negative binomial & B1 ($\kappa=2.5$) & 2,000 & -0.013 & 0.045 & 0.043 & 0.043 & 0.967 \\
Negative binomial & B1 ($\kappa=2.5$) & 10,000 & -0.004 & 0.018 & 0.017 & 0.017 & 0.958 \\
Negative binomial & B2 ($\kappa=4$) & 2,000 & -0.026 & 0.114 & 0.111 & 0.109 & 0.949 \\
Negative binomial & B2 ($\kappa=4$) & 10,000 & -0.006 & 0.044 & 0.043 & 0.043 & 0.949 \\
\addlinespace
Gamma & S1 & 2,000 & -0.013 & 0.024 & 0.021 & 0.020 & 0.937 \\
Gamma & S1 & 10,000 & -0.002 & 0.009 & 0.008 & 0.008 & 0.945 \\
Gamma & S3 (enlarged region) & 2,000 & -0.003 & 0.014 & 0.013 & 0.013 & 0.945 \\
Gamma & S3 (enlarged region) & 10,000 & -0.001 & 0.006 & 0.006 & 0.006 & 0.947 \\
Gamma & B1 ($\kappa=2.5$) & 2,000 & -0.012 & 0.043 & 0.041 & 0.041 & 0.968 \\
Gamma & B1 ($\kappa=2.5$) & 10,000 & -0.003 & 0.016 & 0.016 & 0.016 & 0.962 \\
Gamma & B2 ($\kappa=4$) & 2,000 & -0.022 & 0.098 & 0.096 & 0.096 & 0.957 \\
Gamma & B2 ($\kappa=4$) & 10,000 & -0.006 & 0.038 & 0.037 & 0.038 & 0.959 \\
\bottomrule
\end{tabular}
\begin{minipage}{0.94\textwidth}\footnotesize
\emph{Notes:} Design labels refer to
Table~\ref{tab:mc-poisson-designs}. Both conditional families
have dispersion parameter $\sigma^2=2$.
Unit-root intercepts are chosen to match $\kappa=2.5$ and $4$ with the
Poisson designs. ESD, ASE, and Cov. are defined in the notes to
Table~\ref{tab:mc-poisson-estimation}.
\end{minipage}
\end{table}

\begin{table}[!htbp]
\centering
\setstretch{1}
\caption{Effect of fixing the dispersion parameter at one.}
\label{tab:mc-naive-standardization}
\small
\begin{tabular}{lllrr}
\toprule
Family & Null design & $n$ & $\sigma^2$ estimated & $\sigma^2$ fixed at 1 \\
\midrule
Negative binomial & B1, $H_0$ ($\kappa=2.5$) & 2,000 & 0.044 & 0.130 \\
Negative binomial & B1, $H_0$ ($\kappa=2.5$) & 10,000 & 0.048 & 0.130 \\
Negative binomial & B2, $H_0$ ($\kappa=4$) & 2,000 & 0.043 & 0.121 \\
Negative binomial & B2, $H_0$ ($\kappa=4$) & 10,000 & 0.046 & 0.124 \\
\addlinespace
Gamma & B1, $H_0$ ($\kappa=2.5$) & 2,000 & 0.043 & 0.128 \\
Gamma & B1, $H_0$ ($\kappa=2.5$) & 10,000 & 0.048 & 0.126 \\
Gamma & B2, $H_0$ ($\kappa=4$) & 2,000 & 0.041 & 0.118 \\
Gamma & B2, $H_0$ ($\kappa=4$) & 10,000 & 0.046 & 0.124 \\
\bottomrule
\end{tabular}
\begin{minipage}{0.86\textwidth}\footnotesize
\emph{Notes:} Entries are rejection frequencies at the nominal 5\% level.
The last column fixes \(\widehat\sigma^2=1\) to illustrate the size distortion
caused by ignoring the non-Poisson dispersion parameter.
\end{minipage}
\end{table}

\begin{figure}[p]
\centering
\includegraphics[width=0.72\textwidth]{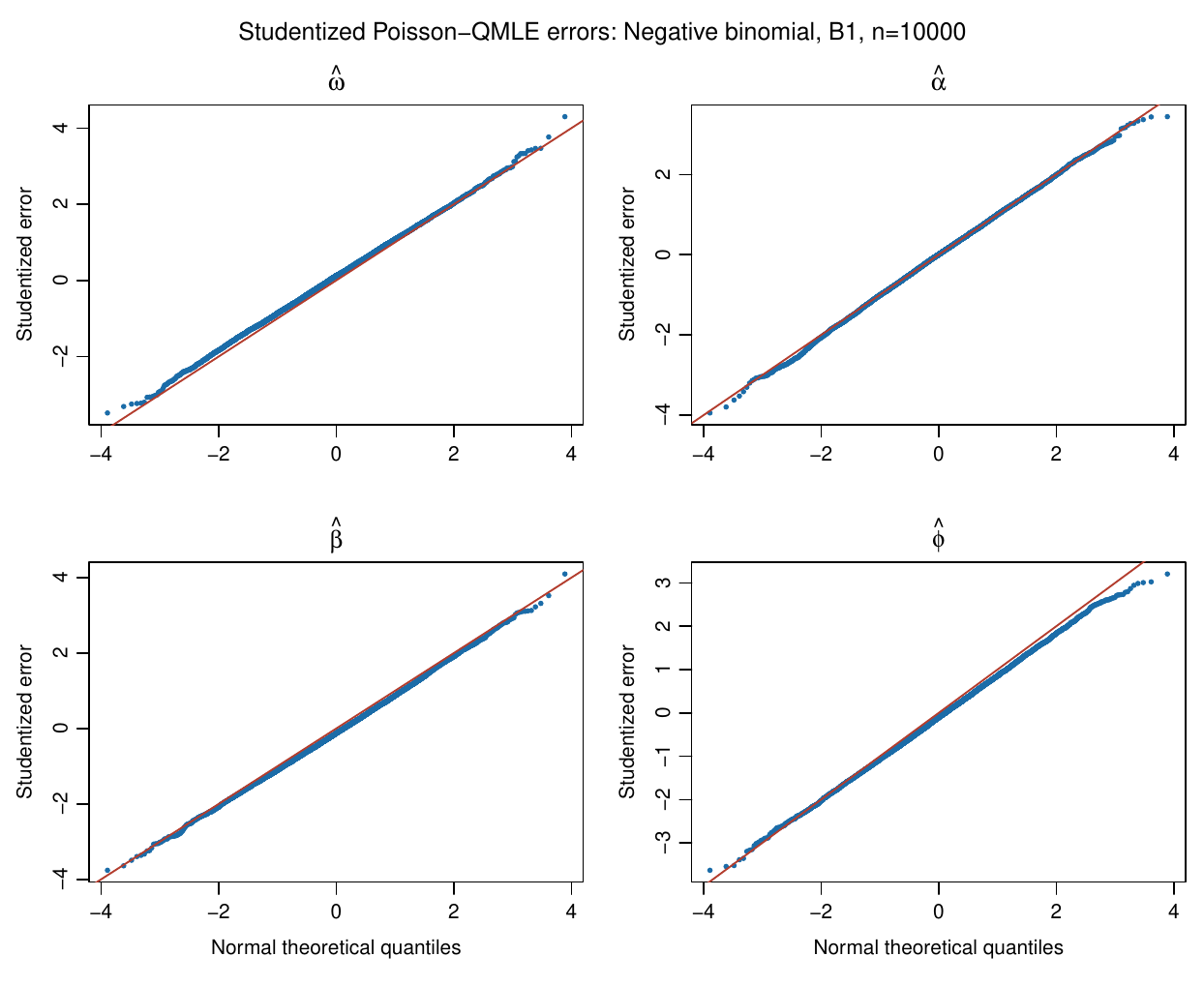}
\par\medskip
\includegraphics[width=0.72\textwidth]{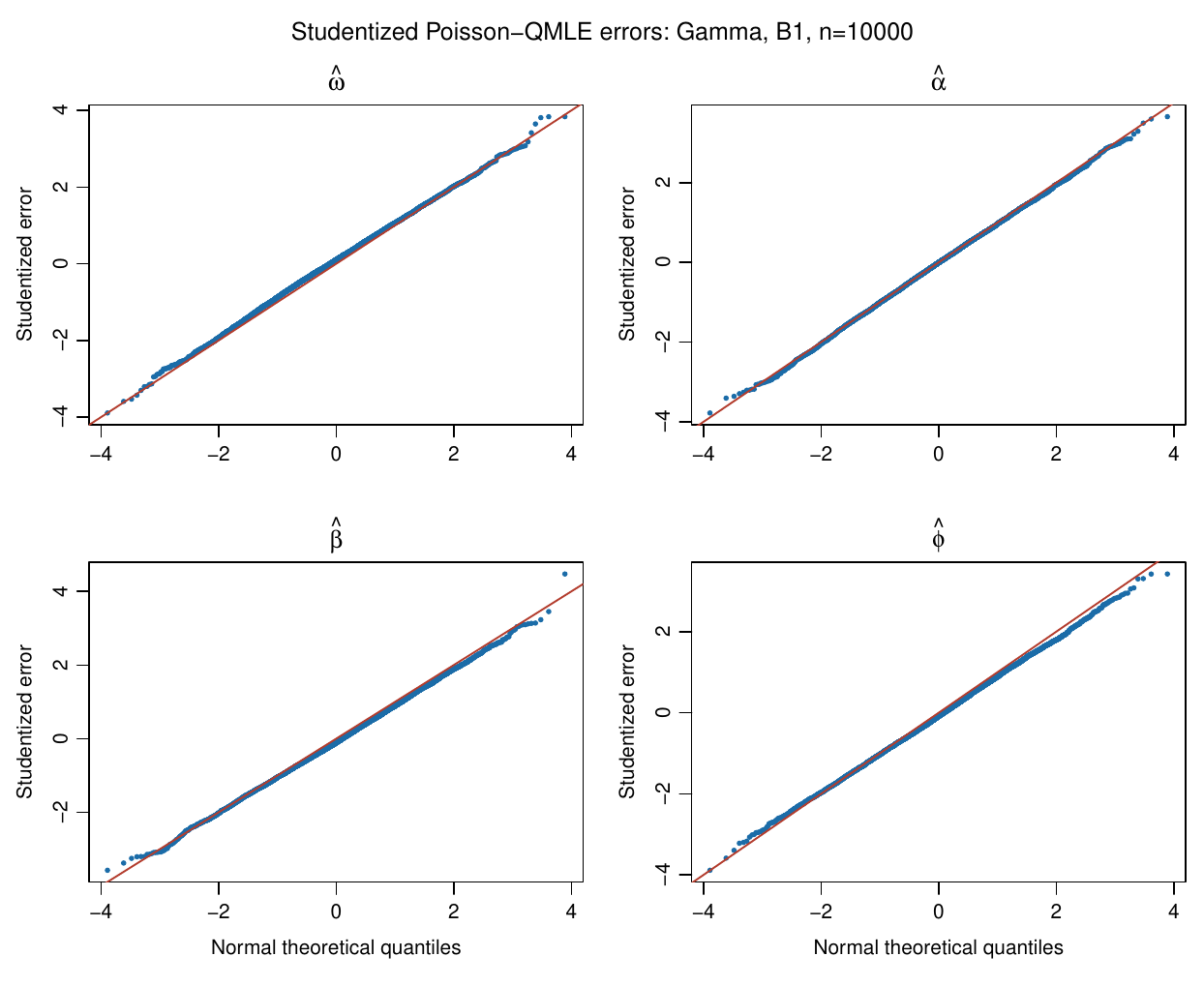}
\caption{Studentized Poisson QMLE errors in the B1 unit-root design under
negative binomial observations (upper) and gamma observations (lower), with
$n=10{,}000$. Within each four-panel block, panels show the
studentized errors for \(\omega,\alpha,\beta,\phi\), in reading order.
{Each panel uses the standard normal 45-degree reference line.
Each estimation error is divided by its estimated standard error, computed
as for the ASE column of Table~\ref{tab:mc-affine-estimation}.}
}
\label{fig:mc-qq-affine-b1}
\end{figure}

\section{\texorpdfstring{{Proofs for Section~\ref{sec:affine-extension}}}{Proofs for Section 2}}
\label{app:affine-extension}

\subsection{\texorpdfstring{{Common tools}}{Common tools}}
\label{app:auxiliary-lemmas}

{In this appendix, bounds are} for fixed model parameters
\((\omega,\alpha,\beta,c)\) and a fixed affine conditional family.
Constants may also depend on auxiliary quantities, such as
\(\epsilon,A,\chi\), fixed in the relevant statement.

Use the affine family \(G_m\) from Section~\ref{sec:affine-extension},
with the convention \(G_0=\delta_0\).  For \(m\geq0\), let
\(W_m\sim G_m\), and put
\[
 U_m=W_m-m.
\]
Thus \(\E U_m=0\) and \(\E U_m^2=\sigma^2m\).
For the observed process, use the innovation \(U_t\) defined
in \eqref{eq:affine-innovation}.  Since \(M_t\) is \(\mathcal F_{t-1}\)-measurable and
\(\mathcal L(X_t\mid\mathcal F_{t-1})=G_{M_t}\), we have
\(U_t\mid(M_t=m)\stackrel d=U_m\).
Thus a subscript \(m\) indexes a state-dependent conditional law, whereas
a subscript \(t\) indexes time.
For \(m\geq0\), we write \(\E[\cdot\mid M_0=m]\) for expectation under
the model initialized at \(M_0=m\), with \(X_0\sim G_m\).
We use analogous notation for probabilities and variances.

For fixed model parameters, let
\(P(m,A)=\Pr(M_1\in A\mid M_0=m)\) denote the transition kernel of the
Markov chain \((M_t)\), where \(A\) is any Borel subset of the state space.
Write \(P^k(m,A)=\Pr(M_k\in A\mid M_0=m)\) for its \(k\)-step
transition kernel.
For a bounded measurable or nonnegative function \(f\), write
\[
 (Pf)(m)=\int f(y)P(m,dy)
 =\E[f(M_1)\mid M_0=m],
 \qquad
 \mu f=\int f(m)\,\mu(dm),
\]
where \(\mu\) is a probability measure on the state space.  Thus
\(\mu Pf\) means \(\mu(Pf)\).  We use the same conventions for other
transition kernels and state distributions.

\begin{lemma}
\label{lem:affine-concentration}
Under Assumption~\ref{ass:affine-seed}, for every \(\epsilon>0\) and
\(A>0\), there is \(C_{\epsilon,A}<\infty\) such that, for \(m\geq1\),
\[
 \Pr(|U_m|>\epsilon m)\leq C_{\epsilon,A}m^{-A}.
\]
For every fixed \(\chi>0\),
\[
 \E(W_m+1)^\chi\leq C_\chi(1+m^\chi).
\]
For every fixed \(\chi\geq2\),
\[
 \E|U_m|^\chi\leq C_\chi(1+m^{\chi/2}).
\]
For every \(m>0\), \(G_m\) has unbounded support.
\end{lemma}

\begin{proof}
Fix an even integer \(k\geq2\).  Assumption~\ref{ass:affine-seed}(i)
gives \(\E W_1^k<\infty\).
For \(0\leq s\leq1\), the convolution identity
\(G_s*G_{1-s}=G_1\) implies that \(W_1\) has the law of
\(W_s+W_{1-s}'\), where the summands are independent and
nonnegative. Consequently,
\[
 \E W_s^k
 \leq \E(W_s+W_{1-s}')^k
 =\E W_1^k.
\]
Thus
\[
 \sup_{0\leq s\leq1}\E|W_s-s|^k
 \leq 2^{k-1}(\E W_1^k+1)<\infty.
\]

Write \(m=j_m+r_m\), where \(j_m=\lfloor m\rfloor\) and
\(r_m=m-j_m\in[0,1)\).  The convolution identity implies that
\(W_m-m\) has the law of the sum of \(j_m\) independent copies
of \(U_1\) and an
independent random variable distributed as \(U_{r_m}\).
Rosenthal's inequality gives, for every even integer \(k\geq2\),
\[
 \E|U_m|^k\leq C_k(1+m^{k/2}).
\]
For any real \(\chi\geq2\), choose an even integer \(k\geq\chi\).
Lyapunov's inequality then gives the stated moment bound of order \(\chi\).
The same decomposition, together with
\(W_m\leq m+|U_m|\), proves the stated raw-moment bound for every
\(\chi\geq2\). Lyapunov's inequality applied to the second raw moment covers
\(0<\chi<2\).

For any prescribed \(A>0\), choose \(\chi\geq2\) with \(\chi/2>A\).
Markov's inequality then gives
\[
 \Pr(|U_m|>\epsilon m)
 \leq \frac{C_\chi(1+m^{\chi/2})}{\epsilon^\chi m^\chi}
 \leq C_{\epsilon,A}m^{-A}.
\]
{Finally, suppose that \(W_m\leq y\) almost surely for some \(m>0\)
and \(y<\infty\). For every integer \(n\geq1\), the convolution identity
represents \(W_m\) as a sum of \(n\) independent nonnegative copies of
\(W_{m/n}\). If \(\Pr(W_{m/n}>y/n)>0\), all \(n\) copies would exceed
\(y/n\) with positive probability. Hence \(W_{m/n}\leq y/n\) almost surely,
and
\[
 \frac{\sigma^2m}{n}=\Var(W_{m/n})\leq\E W_{m/n}^2\leq\frac{y^2}{n^2}.
\]
This fails for large \(n\), so \(G_m\) has unbounded support.}
\end{proof}

\begin{lemma}
\label{lem:affine-state-observation-moments}
Under Assumption~\ref{ass:affine-seed}, if \(X\mid M=m\sim G_m\), then,
for every \(s>0\), \(\E X^s<\infty\) if and only if \(\E M^s<\infty\).
\end{lemma}

\begin{proof}
If \(\E M^s<\infty\), the raw-moment bound in
Lemma~\ref{lem:affine-concentration} gives \(\E X^s<\infty\).
Conversely, when \(m\geq8\sigma^2\), conditional Chebyshev gives
\(\Pr(X\geq m/2\mid M=m)\geq1/2\), and therefore
\[
 \E X^s\geq2^{-s-1}
 \E[M^s\mathbf1_{\{M\geq8\sigma^2\}}],
\]
so \(\E X^s<\infty\) implies \(\E M^s<\infty\).
\end{proof}

\begin{lemma}
\label{lem:feller-invariant-limits}
Let \(E\) be a Polish state space and \(P\) a Feller kernel on \(E\).
(i) If \(\mu_N=N^{-1}\sum_{j<N}P^j(x,\cdot)\) is tight, each weak limit
is \(P\)-invariant.  (ii) If \(P_n\) has invariant law \(\pi_n\),
\((\pi_n)\) is tight, and \(P_nf\to Pf\) locally uniformly for every
bounded Lipschitz \(f\), each weak limit of \((\pi_n)\) is \(P\)-invariant.
In either case, if \(P(x,B)=0\) for every \(x\), every such
invariant limit \(\pi\) satisfies \(\pi(B)=0\).
\end{lemma}

\begin{proof}
For (i), \(\mu_N(Pf-f)=(P^Nf(x)-f(x))/N\to0\) for bounded continuous
\(f\), and the Feller property permits passage to a weak limit.  For (ii),
tightness and local uniform convergence give
\(\pi_n(P_nf-Pf)\to0\), while \(\pi_nPf\to\pi Pf\). Pass to the limit in
\(\pi_nf=\pi_nP_nf\).  Bounded Lipschitz functions determine probability
laws.  Finally,
\(\pi(B)=\int P(x,B)\,\pi(dx)=0\).
\end{proof}

\subsection{Proof of Theorem~\ref{thm:affine-classification}}
\label{app:proof-affine-classification}

\begin{lemma}
\label{lem:affine-cancellation}
Under Assumption~\ref{ass:affine-seed}, for every \(\chi>0\), as \(m\to\infty\),
\[
 \frac{\E\!\left[s_c(\omega+\alpha W_m+\beta m)^{\chi}\right]}{m^\chi}
 \longrightarrow\bigl[\max(\phi,0)\bigr]^{\chi}.
\]
\end{lemma}

\begin{proof}
The elementary softplus bound
\[
 0\leq s_c(x)-\max(x,0)\leq c\log2,
 \qquad x\in\mathbb R,
\]
and the one-Lipschitz property of \(x\mapsto\max(x,0)\) give
\[
 \left|
 \frac{s_c(\omega+\alpha W_m+\beta m)}{m}-\max(\phi,0)
 \right|
 \leq\frac{|\omega|+|\alpha||U_m|+c\log2}{m}.
\]
By Lemma~\ref{lem:affine-concentration}, together with the exact second
moment when \(0<\chi<2\), the right-hand side converges to zero in
\(L^\chi\). Thus \(s_c(\omega+\alpha W_m+\beta m)/m\) converges to
\(\max(\phi,0)\) in \(L^\chi\).
Convergence of its moments of order \(\chi\) proves the claim.
\end{proof}

\begin{proof}[{Proof of Theorem~\ref{thm:affine-classification}}]
{Use the kernel \(P\) of Appendix~\ref{app:auxiliary-lemmas}, and set
\(V_\chi(m)=1+m^\chi\) for \(\chi>0\). Fix
\(\lambda_\chi\in(\max(\phi,0)^\chi,1)\).
Lemma~\ref{lem:affine-cancellation} gives
\(PV_\chi\leq\lambda_\chi V_\chi\) at all sufficiently large states. On the
remaining compact interval, the raw-moment bound of
Lemma~\ref{lem:affine-concentration} and the elementary softplus bound in the
proof of Lemma~\ref{lem:affine-cancellation} make \(PV_\chi\) bounded. Hence,
for some \(B_\chi<\infty\),
\[
 PV_\chi\leq\lambda_\chi V_\chi+B_\chi,\qquad \chi>0.
\]
Write \(\lambda=\lambda_1\) and \(B=B_1\).
We add the state \(0\), with \(G_0=\delta_0\), only to obtain the closed
state space \([0,\infty)\). Pointwise continuity of the conditional Laplace
transform in the state and the continuity theorem for Laplace transforms make
\(P\) Feller on \([0,\infty)\). The drift with \(\chi=1\) makes the Ces\`aro
laws tight, so Lemma~\ref{lem:feller-invariant-limits}(i) gives an invariant
probability \(\pi\). The softplus update is strictly positive, so
\(P(m,\{0\})=0\), and the same lemma gives \(\pi(\{0\})=0\). Started from
\(\pi(dm)G_m(dx)\), the joint chain \(((M_t,X_t))\) is strictly stationary.
By Kolmogorov's extension theorem, it has a two-sided version satisfying
\eqref{mrecursion} and~\eqref{eq:affine-softplus-model}, so a stationary
solution exists.
Under any stationary solution, the law of \(M_0\) is invariant for \(P\)
(Section~\ref{sec:affine-extension}).
For every invariant law \(\pi\) and every \(\chi>0\),
\citet[Lemma~7]{douc2013ergodicity} and the drift give
\(\pi(V_\chi)\leq B_\chi/(1-\lambda_\chi)<\infty\), without assuming
integrability in advance. Lemma~\ref{lem:affine-state-observation-moments}
transfers these moment bounds to the observations.}
\end{proof}

\subsection{Proof of Theorem~\ref{thm:affine-mixing}}
\label{app:proof-affine-mixing}

{A coupling kernel \(Q\) for \(P\) is a Markov kernel on pairs of
states whose marginals from \((m,m')\) are \(P(m,\cdot)\) and
\(P(m',\cdot)\). For a distance \(d\) and \(k\geq1\), write
\[
 (Q^k d)(m,m')=\E[d(M_k,M'_k)\mid M_0=m,M'_0=m'],
\]
where the pair \((M_t,M'_t)\) evolves with \(Q\). Inequalities between such
functions are understood pointwise. For probability measures \(\mu\) and
\(\nu\), \(\mathcal W_d(\mu,\nu)\) denotes the infimum of the expected
distance \(d\) over couplings of \(\mu\) and \(\nu\).}

\begin{lemma}
\label{lem:finite-activity-wasserstein-coupling}
Suppose Assumption~\ref{ass:affine-seed} holds with
case~(a) in part~(ii), and Assumption~\ref{ass:invertibility} holds.
On \(E=[0,\infty)\),
with \(G_0=\delta_0\), the state kernel \(P\) has at most one invariant
probability.  There are a bounded complete metric \(d\), generating the
usual topology, {a coupling kernel \(Q\) for \(P\), and a
constant \(a_0<1\) such that \(Qd\leq d\) everywhere and \(Qd\leq a_0d\)
wherever \(d<1\).}
For every
\(R<\infty\), there are \(N_R\geq1\) and \(\delta_R>0\) such that
\begin{equation}
 (Q^{N_R}d)(m,m')\leq(1-\delta_R)d(m,m'),
 \qquad (m,m')\in[0,R]^2.
 \label{eq:finite-activity-coupling-set}
\end{equation}
If \(\pi\) is invariant, the stationary observations satisfy
\begin{equation}
 \beta_X^{\rm mix}(n)
 \leq C\int\mathcal W_d(P^{n-1}(m,\cdot),\pi)\,\pi(dm),
 \qquad n\geq1.
 \label{eq:finite-activity-wasserstein-beta}
\end{equation}
\end{lemma}

\begin{proof}
{Put \(p_0=G_1(\{0\})\). Assumption~\ref{ass:affine-seed}(ii)(a)
gives \(p_0>0\), and \(\sigma^2>0\) gives \(p_0<1\).
Since
\(e^{-a(u)}=\E e^{-uW_1}\to p_0\) as \(u\to\infty\), dominated convergence
gives, for \(m>0\),
\[
 G_m(\{0\})=\lim_{u\to\infty}\E e^{-uW_m}
 =\lim_{u\to\infty}e^{-ma(u)}=p_0^m.
\]
The convention \(G_0=\delta_0\) gives the same formula at \(m=0\).}
For current states \(m\leq m'\), draw
\(W_m\sim G_m\) and \(V_{m'-m}\sim G_{m'-m}\) independently, and set
\(W_{m'}=W_m+V_{m'-m}\).  The convolution identity gives
\(W_{m'}\sim G_{m'}\). {We call this pairing of \(W_m\) and
\(W_{m'}\) the observation coupling.}
Define the next states by
\[
 M^+=s_c(\omega+\alpha W_m+\beta m),
 \qquad
 M'^+=s_c(\omega+\alpha W_{m'}+\beta m').
\]
{Let \(Q((m,m'),\cdot)\) be the law of \((M^+,M'^+)\), whose
marginals are \(P(m,\cdot)\) and \(P(m',\cdot)\). For \(m>m'\), reverse the
coordinates.}
The observation coupling gives
\begin{equation}
 \lVert G_m-G_{m'}\rVert_{\rm TV}
 \leq\Pr(W_m\ne W_{m'})
 ={1-p_0^{|m-m'|}}.
 \label{eq:affine-tv-similarity}
\end{equation}
At each step, apply this construction to the current pair of states.
All probabilities for the paired process below refer to this coupling.
When observations
agree, the next state distance is at most \(|\beta|\) times the
current distance.  Thus, from initial distance \(D=|m-m'|\),
\begin{equation}
 \Pr(X_j=X'_j\text{ for all }j\geq0\mid M_0=m,M'_0=m')
 \geq{p_0^{D/(1-|\beta|)}}>0.
 \label{eq:finite-activity-permanent-agreement}
\end{equation}
On this event the state distance tends to zero.  Starting from any two
ergodic invariant laws and integrating this strictly positive bound gives
a positive-probability asymptotic coupling.  By
\mbox{\citet[Theorem~1.1]{hairerMattinglyScheutzow2011}} the laws coincide.
Ergodic decomposition proves uniqueness without any moment assumption.

{Fix \(\varepsilon>0\) with
\(a_0=|\beta|+\varepsilon\log(1/p_0)<1\), and set
\(d(m,m')=1\wedge(D/\varepsilon)\).
For \(D<\varepsilon\), the coupling and
\(1-p_0^D\leq D\log(1/p_0)=\varepsilon\log(1/p_0)\,d(m,m')\) give
\[
 (Qd)(m,m')\leq|\beta|\,d(m,m')+D\log(1/p_0)
 =a_0d(m,m'),
\]
whereas for \(D\geq\varepsilon\), \(Qd\leq1=d\).
On \([0,R]^2\), choose \(N_R\) with
\(|\beta|^{N_R}R/\varepsilon\leq1/2\).}
{For \((m,m')\in[0,R]^2\), the event in
\eqref{eq:finite-activity-permanent-agreement} has probability at least
\(p_R=p_0^{R/(1-|\beta|)}\), and on this event
\(d(M_{N_R},M'_{N_R})\leq1/2\). Hence \(Q^{N_R}d\leq1-p_R/2\) on
\([0,R]^2\).}
Combining this with \(Q^{N_R}d\leq Qd\leq a_0d\) when \(d<1\)
proves \eqref{eq:finite-activity-coupling-set}, with
\(\delta_R=\min(1-a_0,p_R/2)\).

{Let \(K(m,\cdot)\) be the law of the observation future
\((X_0,X_1,\ldots)\) from state \(m\), and
\(\mu K=\int K(m,\cdot)\,\mu(dm)\).}
Equation~\eqref{eq:finite-activity-permanent-agreement} implies
\[
 \|K(m,\cdot)-K(m',\cdot)\|_{\rm TV}
 \leq{1-p_0^{D/(1-|\beta|)}}\leq C d(m,m').
\]
Integrating over a coupling gives
\(\|\mu K-\nu K\|_{\rm TV}\leq C\mathcal W_d(\mu,\nu)\).
Given the observation past through time zero, the future starting at
\(n\) has conditional law
\(\E[P^{n-1}(M_1,\cdot)K\mid\mathcal F_0^X]\).
Convexity of total variation and stationarity of \(M_1\) now give
\eqref{eq:finite-activity-wasserstein-beta}.
\end{proof}

\begin{lemma}
\label{lem:infinite-activity-density-overlap}
{Suppose Assumption~\ref{ass:affine-seed} holds with case~(b) in part~(ii),
Assumption~\ref{ass:invertibility} holds, and \(\alpha>0\). Let
\(E_\beta=(0,\infty)\) if \(\beta<0\), and \(E_\beta=(m_\beta,\infty)\) if
\(0\leq\beta<1\), where \(m_\beta\) is the fixed point of
\(m\mapsto s_c(\omega+\beta m)\). From every initial state, the chain
enters \(E_\beta\) almost surely and remains there. On \(E_\beta\), the chain
is irreducible and strongly aperiodic, and every bounded set is petite.
Every invariant law is supported on \(E_\beta\).}
\end{lemma}

\begin{proof}
Using the conditional density \(g(m,x)\) from Assumption~\ref{ass:affine-seed}, write
\[
 h_c(y)=s_c^{-1}(y)=c\log(e^{y/c}-1),
 \qquad
 f_0(m)=s_c(\omega+\beta m).
\]
The transition kernel of the Markov chain \((M_t)\) has density
\begin{equation}
 p(m,y)=\frac{h_c'(y)}{\alpha}
 g\!\left(m,\frac{h_c(y)-\omega-\beta m}{\alpha}\right)
 \mathbf 1_{\{y>f_0(m)\}}.
 \label{eq:density-state-kernel}
\end{equation}
It is jointly continuous and strictly positive on compact subsets of
\(\{(m,y):m>0,\ y>f_0(m)\}\).

{We first establish a two-step minorization.}
Fix \(R>0\).
For \(0<m\leq R\), Markov's inequality gives
\[
 \Pr(X_0\leq2R\mid M_0=m)\geq\frac12.
\]
On this event, \(M_1\) lies in the fixed compact interval
\[
 K_R=[u_R,v_R],\qquad
 u_R=s_c(\omega+\min(\beta,0)R)>0,\quad
 v_R=s_c(\omega+(2\alpha+\max(\beta,0))R)<\infty.
\]
Choose a compact interval \(J_R=[a_R,b_R]\) with
\(a_R>\sup_{z\in K_R}f_0(z)\).  By
\eqref{eq:density-state-kernel},
\[
 d_R=\inf_{z\in K_R,\,y\in J_R}p(z,y)>0.
\]
If \(\nu_R\) is normalized Lebesgue measure on \(J_R\), the Markov
property therefore yields, for every Borel set \(A\subseteq(0,\infty)\),
\begin{equation}
 \Pr(M_2\in A\mid M_0=m)\geq\frac{d_R|J_R|}{2}\nu_R(A),
 \qquad 0<m\leq R.
 \label{eq:density-two-step-overlap}
\end{equation}

We next establish irreducibility on the absorbing set
\(E_\beta\) and show that every initial state enters it almost surely. If
\(\beta<0\), then
\(f_0(z)\downarrow0\) as \(z\to\infty\).  Positivity of
\eqref{eq:density-state-kernel} shows that any state can first reach a
neighborhood of a sufficiently large \(z\) and can then reach every
positive-Lebesgue-measure subset of \((0,\infty)\).

If \(0\leq\beta<1\), \(f_0\) is a contraction and has the unique fixed
point \(m_\beta\).  Write \(f_0^{\circ n}\) for its \(n\)-fold
composition.  For \(m>m_\beta\),
 \(m_\beta\leq f_0(m)<m\), and induction in
\eqref{eq:density-state-kernel} shows that \(P^n(m,\cdot)\) has a density
positive on \((f_0^{\circ n}(m),\infty)\).  Since
\(f_0^{\circ n}(m)\to m_\beta\), every positive-Lebesgue-measure subset of
\((m_\beta,\infty)\) can eventually be reached.  For this case, \(E_\beta=(m_\beta,\infty)\) is absorbing because the transition
support satisfies \(y>f_0(m)\geq m_\beta\) for \(m\in E_\beta\).  In the preceding two-step
minorization, take \(R=m_\beta\) and choose \(J_R\subset E_\beta\).  Then, for
some \(\epsilon>0\),
\[
 \inf_{0<m\leq m_\beta}P^2(m,E_\beta)\geq\epsilon.
\]
Consequently, if \(T_{E_\beta}=\inf\{t\geq0:M_t\in E_\beta\}\), the Markov property
gives
\[
 \Pr(T_{E_\beta}>2n\mid M_0=m)\leq(1-\epsilon)^n,
 \qquad 0<m\leq m_\beta.
\]
{From the state 0, which the proof of
Theorem~\ref{thm:affine-classification} adds to close the state space, the
chain moves deterministically to \(s_c(\omega)>0\).}
Thus every initial state enters \(E_\beta\) almost surely
and remains there, and
\(P^n(m,E_\beta)\to1\) for every initial state.  Hence, for any invariant law
\(\mu\), invariance and dominated convergence yield
\[
 \mu(E_\beta)=\int P^n(m,E_\beta)\,\mu(dm)\longrightarrow1.
\]
Thus every invariant law is supported on \(E_\beta\). For \(\beta<0\), \(E_\beta=(0,\infty)\) is the whole state space.

{In either case, \(\beta<0\) or \(0\leq\beta<1\), choose a
compact interval \(D=[a,b]\subset E_\beta\) such that
\(a>\sup_{m\in D}f_0(m)\). The density in
\eqref{eq:density-state-kernel} is bounded below on \(D^2\), so
\(P(m,\cdot)\geq\epsilon_D\nu_D\) for \(m\in D\), where \(\epsilon_D>0\) and
\(\nu_D\) is normalized Lebesgue measure on \(D\).
The reachability results give irreducibility on \(E_\beta\),
\eqref{eq:density-two-step-overlap} makes every bounded set petite, and the
one-step minorization on \(D\) gives strong aperiodicity.}
\end{proof}

For later use, let \(\beta_M^{\rm mix}\) and
\(\beta_{(M,X)}^{\rm mix}\) denote the \(\beta\)-mixing
(absolute-regularity) coefficients of
the stationary processes \((M_t)\) and \(((M_t,X_t))\),
respectively.  {The state \(M_1=s_c(\omega+\alpha X_0+\beta M_0)\)
is \(\mathcal F_0\)-measurable, where \(\mathcal F_0=\sigma(M_s,X_s:s\leq0)\).
Given \(\mathcal F_0\), the future path from time \(j\) is obtained by drawing
\(M_j\) from \(P^{j-1}(M_1,\cdot)\) and then running the model forward from
\(M_j\). This second step uses the same kernel for every past, and total
variation contracts under it.}
Passing to the observation
process replaces both the past and future sigma-fields by sub-sigma-fields,
which cannot increase the \(\beta\)-mixing coefficient. Thus the first
inequality below follows from restriction to sub-sigma-fields, and the
second from total-variation contraction:
\begin{equation}
 \beta_X^{\rm mix}(j)
 \leq\beta_{(M,X)}^{\rm mix}(j)
 \leq\beta_M^{\rm mix}(j-1),
 \qquad j\geq1,
 \qquad
 \beta_M^{\rm mix}(j)
 =\int\|P^j(m,\cdot)-\pi\|_{\rm TV}\,\pi(dm).
 \label{eq:state-to-observation-beta-transfer}
\end{equation}

\begin{proof}[{Proof of Theorem~\ref{thm:affine-mixing}}]
\proofstep{Uniqueness and geometric \(\beta\)-mixing}
We first treat \(\alpha=0\). The map
\(f_0(m)=s_c(\omega+\beta m)\) is a contraction on \([0,\infty)\)
with Lipschitz constant \(|\beta|<1\), so it has a unique fixed point
\(m_*>0\). The finite first moment from
Theorem~\ref{thm:affine-classification} and stationarity give
\[
 \E|M_0-m_*|=\E|M_n-m_*|
 \leq |\beta|^n\E|M_0-m_*|\longrightarrow0.
\]
Thus \(M_t=m_*\) almost surely. Since the conditional law of \(X_t\)
given the past is the fixed law \(G_{m_*}\), the observations are
independent and identically distributed. The stationary solution is
unique and \(\beta_X^{\mathrm{mix}}(j)=0\) for \(j\geq1\).

For the remaining cases, \(\alpha\neq0\). The geometric
drift inequality from the proof of
Theorem~\ref{thm:affine-classification} holds throughout the
{region \(\phi<1\)}.

\medskip\noindent\emph{Case (a): positive mass at zero.}
Use the coupling kernel \(Q\) and metric \(d\) from
Lemma~\ref{lem:finite-activity-wasserstein-coupling}.
The same lemma gives uniqueness of the invariant law supplied by
Theorem~\ref{thm:affine-classification}.
{For \(V=V_1\), which is continuous, the geometric drift gives
\(P^nV\leq\lambda^nV+K_V\), where \(K_V=B/(1-\lambda)\). The metric
\(d\leq1\) of Lemma~\ref{lem:finite-activity-wasserstein-coupling} is
distance-like in the sense of
\citet[Definition~4.3]{hairerMattinglyScheutzow2011}. Since \(Qd\leq d\),
also \(Q^Nd\leq Qd\leq a_0d\) wherever \(d<1\),
so \(d\) is contracting for every \(P^N\) in the sense of their
Definition~4.6. Bound
\eqref{eq:finite-activity-coupling-set} makes \(\{V\leq4K_V\}\)
\(d\)-small for some \(P^N\) in the sense of their Definition~4.4.
Increasing \(N\) preserves both properties and ensures
\(\lambda^N\leq1/8\). The proof of their Theorem~4.8 uses only these drift
and coupling estimates for one transition operator and its iterates.
Applied to \(P^N\), it gives an integer \(L\geq1\) and \(q_H\in(0,1)\) such
that, for all probability measures \(\mu,\nu\),
\begin{equation}
 \mathcal W_{\widetilde d}(\mu P^L,\nu P^L)
 \leq q_H\mathcal W_{\widetilde d}(\mu,\nu),
 \qquad
 \widetilde d(x,y)=\sqrt{d(x,y)(1+V(x)+V(y))},
 \label{eq:weak-harris-block-contraction}
\end{equation}
where \(\mathcal W_{\widetilde d}\) is the infimum of the expected cost
\(\widetilde d\) over couplings. Existence and \(\pi(V)<\infty\) have
already been proved.}
For \(n=kL+r\), \(0\leq r<L\), the inequality \(Qd\leq d\),
invariance of \(\pi\), and \eqref{eq:weak-harris-block-contraction} give
\[
 \begin{aligned}
 \mathcal W_d(P^n(m,\cdot),\pi)
 &\leq\mathcal W_d(P^{kL}(m,\cdot),\pi)\\
 &\leq q_H^k\mathcal W_{\widetilde d}(\delta_m,\pi)
 \leq q_H^k\sqrt{1+V(m)+\pi(V)}.
 \end{aligned}
\]
Thus, with \(r_*=q_H^{1/L}\in(0,1)\) and a suitable \(C<\infty\),
\[
 \mathcal W_d(P^n(m,\cdot),\pi)
 \leq Cr_*^n\sqrt{1+V(m)+\pi(V)}.
\]
Integration against \(\pi\) and
\eqref{eq:finite-activity-wasserstein-beta} give geometric
\(\beta\)-mixing of \((X_t)\).

\medskip\noindent\emph{Case (b), \(\alpha>0\): a positive continuous density.}
Lemma~\ref{lem:infinite-activity-density-overlap} shows that, on \(E_\beta\), the chain is irreducible and strongly aperiodic and
every bounded set is petite. It also shows that every invariant law is
supported there.
Set \(V=V_1\).  The geometric drift can be written, for a sufficiently
large bounded set \(C\), as
\[
 PV\leq\lambda_*V+B_*\mathbf1_C,
 \qquad \lambda_*<1.
\]
The geometric-drift theorem
\citep[Theorem~15.0.1, pp.~354--355]{meyn1993markov} therefore gives a unique invariant
law \(\pi\) and constants \(K<\infty\), \(r_*\in(0,1)\) such that
\[
 \|P^j(m,\cdot)-\pi\|_{\rm TV}\leq K V(m)r_*^j.
\]
Since \(\pi(V)<\infty\) by
Theorem~\ref{thm:affine-classification}, integration and
\eqref{eq:state-to-observation-beta-transfer} give geometric
\(\beta\)-mixing of the observations.

\medskip\noindent\emph{Case (b), \(\alpha<0\): a bounded absorbing interval.}
Put \(r=\max(\beta,0)<1\) and \(b=s_c(\omega)>0\).
Nonnegativity of \(X_t\), monotonicity of \(s_c\), and its
one-Lipschitz property give
\[
 M_{t+1}\leq s_c(\omega+rM_t)\leq b+rM_t.
\]
For any initial state \(M_0=m>0\), iteration yields
\[
 M_n\leq r^n m+\frac{b(1-r^n)}{1-r},\qquad n\geq1.
\]
Fix \(R>b/(1-r)\) and put \(S_R=(0,R]\).
Since \(b+rR<R\), this set is absorbing. The preceding bound shows
that every finite initial state enters \(S_R\) after finitely many steps.
Consequently, \(P^n(m,S_R)\to1\) for every \(m>0\), and every invariant
law \(\mu\) satisfies
\[
 \mu(S_R)=\int P^n(m,S_R)\,\mu(dm)\longrightarrow1.
\]
Theorem~\ref{thm:affine-classification} supplies at least one such law.

We next establish a minorization uniform over \(S_R\), including near
zero. For \(0<m\leq R\), Markov's inequality and
\(\E(X_0\mid M_0=m)=m\) give
\[
 \Pr(X_0\leq2R\mid M_0=m)\geq\frac12.
\]
On this event, \(M_1\) belongs to the compact interval \(K=[u,v]\), where
\[
 \begin{aligned}
 u&=s_c\bigl(\omega+2\alpha R+\min(\beta,0)R\bigr)>0,\\
 v&=s_c(\omega+rR)\leq b+rR<R.
 \end{aligned}
\]
In particular, \(K\subset S_R\) is bounded away from zero.
Write \(h_c=s_c^{-1}\) and \(f_0(z)=s_c(\omega+\beta z)\).
For \(\alpha<0\), the state transition density is
\[
 p(z,y)=\frac{h_c'(y)}{|\alpha|}
 g\!\left(z,\frac{\omega+\beta z-h_c(y)}{|\alpha|}\right),
 \qquad 0<y<f_0(z),
\]
and it is zero otherwise. Choose a compact interval
\(J=[a_J,b_J]\) with
\(0<a_J<b_J<\min_{z\in K}f_0(z)\).
Since \(f_0(z)\leq b+rz<R\) on \(K\), we have \(J\subset S_R\).
Joint continuity and strict positivity of \(g\) imply
\[
 d:=\inf_{z\in K,\,y\in J}p(z,y)>0.
\]
Let \(\nu\) be normalized Lebesgue measure on \(J\). For every Borel
set \(A\subseteq S_R\), the Markov property gives
\[
 \begin{aligned}
 P^2(m,A)
 &\geq\int_K P(z,A)\,P(m,dz)\\
 &\geq d|J|\nu(A)P(m,K)
 \geq\frac{d|J|}{2}\nu(A),\qquad m\in S_R.
 \end{aligned}
\]
{Choose \(0<\epsilon<1\) no larger than \(d|J|/2\). On \(S_R\), write
\(P^2(m,\cdot)=\epsilon\nu(\cdot)+(1-\epsilon)\mathcal R(m,\cdot)\) for a
Markov kernel \(\mathcal R\). Thus \(P^2\) contracts total variation on
\(S_R\) by the factor \(1-\epsilon\). All invariant laws are supported on
\(S_R\), so the invariant law \(\pi\) is unique, and
\[
 \sup_{m\in S_R}\|P^n(m,\cdot)-\pi\|_{\mathrm{TV}}
 \leq(1-\epsilon)^{\lfloor n/2\rfloor},\qquad n\geq0.
\]
Equation~\eqref{eq:state-to-observation-beta-transfer} then gives geometric
\(\beta\)-mixing of the observations.}

\proofstep{Ergodicity of the stationary solution}
{Step~1 gives a unique invariant state law \(\pi\). For \(\alpha=0\),
this uses that every invariant law yields a stationary solution, as in the
proof of Theorem~\ref{thm:affine-classification}. Let \(\bar\pi\) be an
invariant law of the joint chain \(((M_t,X_t))\), with first marginal
\(\bar\pi_M\), and let \((M_0,X_0)\sim\bar\pi\). Since
\(X_1\mid M_1\sim G_{M_1}\), \(\bar\pi(dm,dx)=\bar\pi_M(dm)G_m(dx)\). Then
\(M_1=s_c(\omega+\alpha X_0+\beta M_0)\) has law \(\bar\pi_MP\), and
invariance gives \(\bar\pi_MP=\bar\pi_M\), so \(\bar\pi_M=\pi\) and
\(\bar\pi=\pi(dm)G_m(dx)\), as in the proof of
\citet[Proposition~8]{douc2013ergodicity}. Their Theorems~30 and~31 make the
two-sided joint process ergodic, and hence so is its observation coordinate.
Every stationary solution has this law.}
\end{proof}

\subsection{Proof of Theorem~\ref{thm:affine-boundary}}
\label{app:proof-affine-boundary}

At \(\phi=1\), {Lemma~\ref{lem:affine-cancellation}} gives
\(\E[M_1\mid M_0=m]/m\to1\), which does not determine the
sign of \(\E[M_1-m\mid M_0=m]\).
We therefore introduce the square-root chain \((R_t)\) by
\begin{equation*}
 R_t:=\sqrt{M_t}.
\end{equation*}
Since \((M_t)\) is a time-homogeneous Markov chain on
\((0,\infty)\) and \(x\mapsto\sqrt x\) is a bijection of \((0,\infty)\), so
is \((R_t)\). Under a stationary solution, \(R_t\) is defined for every
\(t\in\mathbb Z\).
The notation \(\E[\cdot\mid R_0=x]\) means expectation for the model
initialized at \(M_0=x^2\), with \(X_0\sim G_{x^2}\).
{At \(\phi=1\), \(\omega+\alpha W_m+\beta m=\omega+m+\alpha U_m\).
Hence, given \(M_0=m\), the increment \(M_1-m\) has the law of
\[
 \Delta_m:=s_c(\omega+m+\alpha U_m)-m.
\]}%
The next two lemmas show that the square-root increments have moments of
every order bounded uniformly in the initial state, and that
\[
 \E[R_1-x\mid R_0=x]=O(x^{-1}),\qquad x\to\infty.
\]
Thus \((R_t)\) has asymptotically zero drift, which is the form needed
for the recurrence and invariant-tail arguments below.

\begin{lemma}
\label{lem:affine-boundary-expansions}
{Under Assumption~\ref{ass:affine-seed}, suppose \(\phi=1\). As
\(m\to\infty\), for every \(A>0\) and every fixed \(\chi>0\),
\begin{align}
 \E\Delta_m&=\omega+O(m^{-A}),\label{eq:affine-delta-mean}\\
 \E\Delta_m^2&=\alpha^2\sigma^2m+O(1),\label{eq:affine-delta-var}\\
 \E[(m+\Delta_m)^\chi-m^\chi]
 &=\chi m^{\chi-1}\Bigl[\omega+\frac{\chi-1}{2}\alpha^2\sigma^2\Bigr]
   +O(m^{\chi-3/2}),\label{eq:affine-power-drift}\\
 \E\bigl[(\sqrt{m+\Delta_m}-\sqrt m)^2\bigr]
 &=\frac{\alpha^2\sigma^2}{4}+O(m^{-1/2}).\label{eq:affine-sqrt-variance}
\end{align}}
\end{lemma}

\begin{proof}
{Write \(Y_m=\omega+m+\alpha U_m\), so that \(\Delta_m=s_c(Y_m)-m\).}
Choose \(\epsilon>0\) so small that \(1-|\alpha|\epsilon>0\).  On
\(\{|U_m|\leq\epsilon m\}\),
\[
 Y_m=\omega+m+\alpha U_m
 \geq \omega+(1-|\alpha|\epsilon)m.
\]
Hence, for all sufficiently large \(m\), \(Y_m\) is positive and bounded
above and below by positive multiples of \(m\), so the softplus remainder
\({\mathcal E_m}=s_c(Y_m)-Y_m\) is exponentially small.  Globally,
\[
 0\leq {\mathcal E_m}\leq c\log2+\max(-Y_m,0)
       \leq C(1+m+|U_m|).
\]
Consequently, Lemma~\ref{lem:affine-concentration} and H\"older's
inequality imply that, for every fixed \(j\geq0\), \(\ell\geq1\), and
\(A>0\),
\begin{equation}
 \E[(1+|U_m|)^j{\mathcal E_m}^\ell]=O(m^{-A}).
 \label{eq:softplus-remainder-superpolynomial}
\end{equation}
{To bound expectations over events of small probability,}
put
\(T_m=1+m+|U_m|\). Lemma~\ref{lem:affine-concentration} gives
\(\E T_m^{2d}=O(m^{2d})\) for every fixed \(d\geq0\). For any events
\(B_m\) with \(\Pr(B_m)\leq C_Nm^{-N}\), H\"older's inequality gives
\begin{equation}
 \E[T_m^d\mathbf1_{B_m}]
 \leq (\E T_m^{2d})^{1/2}\Pr(B_m)^{1/2}
 \leq C_{d,N}m^{d-N/2},\qquad m\geq1.
 \label{eq:affine-holder-exceptional-event}
\end{equation}
Choosing \(N\geq2(d+A)\) makes this \(O(m^{-A})\). Apply this bound
with \(B_m=\{|U_m|>\epsilon m\}\) and \(d=j+\ell\), since
\((1+|U_m|)^j\mathcal E_m^\ell\leq CT_m^{j+\ell}\).
On \(\{|U_m|\leq\epsilon m\}\), the same integrand is bounded by an
exponential times a polynomial. These two estimates prove
\eqref{eq:softplus-remainder-superpolynomial}.

Since
\(\Delta_m=\omega+\alpha U_m+{\mathcal E_m}\),
\eqref{eq:softplus-remainder-superpolynomial} and {the centered-moment
bound of Lemma~\ref{lem:affine-concentration}}
prove \eqref{eq:affine-delta-mean}--\eqref{eq:affine-delta-var}.
They also give, for every fixed \(p\geq2\),
\begin{equation}
 \E|\Delta_m|^p\leq C_p(1+m^{p/2}).
 \label{eq:affine-delta-absolute-moments}
\end{equation}

{Fix \(\chi>0\) and \(\eta\in(0,1/2]\). By Markov's inequality
and~\eqref{eq:affine-delta-absolute-moments},
\(\Pr\{|\Delta_m|>\eta m\}\leq C_p(1+m^{p/2})(\eta m)^{-p}=O(m^{-p/2})\)
for every \(p\geq2\). On \(\{|\Delta_m|\leq\eta m\}\), Taylor's formula
gives
\[
 (m+\Delta_m)^\chi-m^\chi
 =\chi m^{\chi-1}\Delta_m
  +\frac{\chi(\chi-1)}2m^{\chi-2}\Delta_m^2
  +O(m^{\chi-3}|\Delta_m|^3),
\]
and \eqref{eq:affine-delta-absolute-moments} with \(p=3\) bounds the
expected remainder by \(O(m^{\chi-3/2})\).
On \(B_m=\{|\Delta_m|>\eta m\}\), the power increment and its quadratic
Taylor polynomial are bounded in absolute value by
\(CT_m^{\max\{\chi,2\}}\) for \(m\geq1\), since
\(0\leq m+\Delta_m\leq CT_m\) and \(|\Delta_m|\leq CT_m\).
{Since the preceding Markov bound holds for every \(p\),
\(\Pr(B_m)\leq C_Nm^{-N}\) for every \(N\), and
\eqref{eq:affine-holder-exceptional-event} with \(d=\max\{\chi,2\}\) makes
these contributions \(O(m^{-A})\) for every \(A>0\). Substituting
\eqref{eq:affine-delta-mean}--\eqref{eq:affine-delta-var} gives
\eqref{eq:affine-power-drift}.}
Finally, \(m+\Delta_m=s_c(Y_m)>0\), and the
identity
\[
 (\sqrt{m+\Delta_m}-\sqrt m)^2
 =\Delta_m-2\sqrt m\,(\sqrt{m+\Delta_m}-\sqrt m),
\]
together with \eqref{eq:affine-delta-mean} and
\eqref{eq:affine-power-drift} with \(\chi=1/2\), gives
\[
 \E(\sqrt{m+\Delta_m}-\sqrt m)^2
 =\omega-\frac{4\omega-\alpha^2\sigma^2}{4}+O(m^{-1/2})
 =\frac{\alpha^2\sigma^2}4+O(m^{-1/2}),
\]
which proves \eqref{eq:affine-sqrt-variance}.}
\end{proof}

The uniform moment bounds below control the truncation errors and
large-jump probabilities in the proofs of
Lemmas~\ref{lem:affine-general-state-boundary}
and~\ref{lem:affine-boundary-invariant-tail}.

\begin{lemma}
\label{lem:affine-uniform-root-increments}
Under Assumption~\ref{ass:affine-seed}, suppose \(\phi=1\), and define
\[
 \Delta_m^{(1/2)}=\sqrt{s_c(\omega+m+\alpha U_m)}-\sqrt m.
\]
For every fixed \(p>0\),
\[
 \sup_{m\geq0}\E|\Delta_m^{(1/2)}|^p<\infty.
\]
\end{lemma}

\begin{proof}
For \(m\geq0\), the one-Lipschitz property of \(s_c\) and
\(0\leq s_c(m)-m\leq c\log2\) give
\[
 |s_c(\omega+m+\alpha U_m)-m|
 \leq |\omega|+|\alpha||U_m|+c\log2.
\]
Thus, for \(m\geq1\),
\[
 |\Delta_m^{(1/2)}|
 =\frac{|s_c(\omega+m+\alpha U_m)-m|}
 {\sqrt{s_c(\omega+m+\alpha U_m)}+\sqrt m}
 \leq\frac{|\omega|+|\alpha||U_m|+c\log2}{\sqrt m}.
\]
For \(p\geq2\), the centered-moment bound in
Lemma~\ref{lem:affine-concentration} now gives
\(\sup_{m\geq1}\E|\Delta_m^{(1/2)}|^p<\infty\).
For \(0<p<2\), the same conclusion follows from
\(\E|\Delta_m^{(1/2)}|^p
\leq(\E|\Delta_m^{(1/2)}|^2)^{p/2}\).
For \(0\leq m\leq1\), the elementary softplus bound gives
\(|\Delta_m^{(1/2)}|^p\leq C_p(1+W_m^{p/2})\),
whose expectation is uniformly bounded by the raw-moment bound in
Lemma~\ref{lem:affine-concentration}, for every \(p>0\).
\end{proof}

For the square-root calculations at \(\phi=1\), we use the following
notation and bounds. For \(x>0\), let \(\xi(x)\) have the conditional law of
\(R_1-x\) given \(R_0=x\). Then \(\xi(x)\overset d=\Delta_{x^2}^{(1/2)}\),
and Lemma~\ref{lem:affine-uniform-root-increments} gives
\[
 \bar\mu_q:=\sup_{x>0}\E|\xi(x)|^q<\infty\qquad\text{for every }q>0.
\]
As in \citet{denisov2025markov}, we truncate the increments at the level
\(s(x):=\min\{\sqrt x,x/2\}\), which is increasing, satisfies
\(s(x)\leq x/2\), and equals \(\sqrt x\) for \(x\geq4\). For
\(j\in\{1,2\}\), write
\[
 \mu_j(x)=\E\xi(x)^j,\qquad
 \mu_j^{[s]}(x)=\E\!\left[\xi(x)^j\mathbf1_{\{|\xi(x)|\leq s(x)\}}\right].
\]
For \(x\geq4\), on the event \(\{|\xi(x)|>\sqrt x\}\) we have
\(|\xi(x)|^j\leq|\xi(x)|^{j+4}x^{-2}\). Hence
\begin{equation}
 |\mu_j(x)-\mu_j^{[s]}(x)|
 \leq\E\!\left[|\xi(x)|^j\mathbf1_{\{|\xi(x)|>\sqrt x\}}\right]
 \leq\bar\mu_{j+4}\,x^{-2},
 \qquad j\in\{1,2\},\ x\geq4,
 \label{eq:sqrt-truncation}
\end{equation}
and Markov's inequality gives, for every \(q>0\),
\begin{equation}
 \Pr\{|\xi(x)|\geq\sqrt x\}\leq\bar\mu_q\,x^{-q/2},\qquad x>0.
 \label{eq:sqrt-large-jump}
\end{equation}

The next lemma shows that the chain visits arbitrarily large states
almost surely, as required by Theorem~8.2 of \citet{denisov2025markov}
in the proof of Lemma~\ref{lem:affine-boundary-invariant-tail}.

\begin{lemma}
\label{lem:affine-boundary-nonconfinement}
Under Assumptions~\ref{ass:affine-seed} and \ref{ass:invertibility}, suppose
\(\phi=1\).  For every initial state \(m\geq0\) and every intercept \(\omega\),
\[
 \Pr\!\left(\limsup_{t\to\infty}M_t=\infty\,\middle|\,M_0=m\right)=1.
\]
\end{lemma}

\begin{proof}
Here \(\alpha=1-\beta>0\).  Fix \(K>0\) and put
\(\ell_K=s_c(\omega+\min(\beta,0)K)>0\).
If \(M_t\leq K\), then \(M_{t+1}\geq\ell_K\).  For \(m\geq\ell_K\),
the convolution identity makes \(G_m\) stochastically dominate
\(G_{\ell_K}\).  Choose \(z_K\) with
\(s_c(\omega+\alpha z_K+\min(\beta,0)K)>K\).
The unbounded support of \(G_{\ell_K}\) then gives, uniformly over
\(0\leq m\leq K\),
\[
 \Pr(\max(M_1,M_2)>K\mid M_0=m)
 \geq G_{\ell_K}([z_K,\infty))=:\delta_K>0.
\]
The Markov property bounds the probability of remaining in \([0,K]\)
for the next \(2n\) steps by \((1-\delta_K)^n\).
For each deterministic integer \(r\geq0\), apply this bound
conditionally on \(M_r\) and let the number of two-step blocks tend to
infinity to obtain
\[
 \Pr(M_t\leq K\text{ for every }t\geq r\mid M_0=m)=0.
\]
Taking the countable intersection over \(r\geq0\) and
\(K\in\mathbb N\), almost surely every tail of the path exceeds every
integer threshold, and hence \(\limsup_{t\to\infty}M_t=\infty\).
\end{proof}

\begin{lemma}
\label{lem:affine-general-state-boundary}
Suppose Assumptions~\ref{ass:affine-seed} and
\ref{ass:invertibility} hold, \(\phi=1\), and \(\omega\geq0\).  Then the Markov chain \((M_t)\) has
no invariant probability measure.
\end{lemma}

\begin{proof}
Use the notation \(\xi(x)\), \(s(x)\), \(\mu_j\) and \(\mu_j^{[s]}\)
introduced after Lemma~\ref{lem:affine-uniform-root-increments}.
{With \(m=x^2\), taking \(\chi=1/2\) in~\eqref{eq:affine-power-drift}
and using~\eqref{eq:affine-sqrt-variance}, we obtain}
\(\mu_1(x)=(4\omega-\alpha^2\sigma^2)/(8x)+O(x^{-2})\) and
\(\mu_2(x)=\alpha^2\sigma^2/4+O(x^{-1})\), and by~\eqref{eq:sqrt-truncation}
the same expansions hold for \(\mu_1^{[s]}\) and \(\mu_2^{[s]}\). Therefore,
for any fixed \(\delta\in(0,1)\) and all large \(x\),
\begin{equation*}
 \frac{2\mu_1^{[s]}(x)}{\mu_2^{[s]}(x)}
 =\frac{4\omega/(\alpha^2\sigma^2)-1}{x}+O(x^{-2})
 \geq-\frac1x-\frac{1-\delta}{x\log x}.
\end{equation*}
The same estimates imply
\(\mu_1(x)\geq-C/x\) for large \(x\),
\(\sup_{x\geq0}\mu_2(x)<\infty\), and
\(\liminf_{x\to\infty}\mu_2^{[s]}(x)>0\).
Moreover, \(\bar\mu_3<\infty\) and~\eqref{eq:sqrt-large-jump}
with \(q=6\) give, since \(\{\xi(x)\leq-s(x)\}\subseteq\{|\xi(x)|\geq\sqrt x\}\)
for \(x\geq4\),
\[
 \E[|\xi(x)|^3\mathbf1_{\{-s(x)\leq\xi(x)\leq0\}}]=O(1)=o(x/\log x),
 \qquad
 \Pr\{\xi(x)\leq-s(x)\}=o((x^2\log x)^{-1}).
\]
{The ratio bound is the drift condition of
\citet[Corollary~2.12]{denisov2025markov} with a single logarithmic
correction. The remaining estimates are their conditions (2.60)--(2.62),
(2.65) and (2.66), with truncation level \(s(x)\leq x/2\).}
That
corollary yields an \(r_0<\infty\) such that
\begin{equation}
 \E[T_r\mid R_0=x]=\infty\quad(x>r\geq r_0),
 \qquad T_r=\inf\{t\geq0:R_t\leq r\}.
 \label{eq:affine-infinite-mean-passage}
\end{equation}

If \(\pi_R\) were invariant, strict positivity of softplus would give
\(\pi_R(\{0\})=0\).  Choose \(r\geq r_0\) with
\(\pi_R(A)>0\), where \(A=(0,r]\), and let
\(T_A^+=\inf\{t\geq1:R_t\in A\}\).
For every \(x\in A\), the unbounded support of \(G_{x^2}\) and
\(\alpha>0\) give \(\Pr(R_1>r\mid R_0=x)>0\).  The Markov property and
\eqref{eq:affine-infinite-mean-passage} imply
\(\E[T_A^+\mid R_0=x]=\infty\).
Let \(\Pr_{\pi_R}\) denote probability under the two-sided stationary law
with marginal distribution \(\pi_R\).  Then
\[
\begin{split}
 \int_A\E[T_A^+\mid R_0=x]\,\pi_R(dx)
 &=\sum_{j\geq0}\Pr\nolimits_{\pi_R}\{R_0\in A,R_1,\ldots,R_j\notin A\}\\
 &=\sum_{j\geq0}\Pr\nolimits_{\pi_R}\{R_{-j}\in A,R_{-j+1},\ldots,R_0\notin A\}
 \leq1.
\end{split}
\]
For \(j=0\) the event is \(\{R_0\in A\}\). The events in the last sum
are disjoint, specifying the most recent visit to \(A\) at or before
time zero.  This contradicts \(\pi_R(A)>0\) and proves the claim.
\end{proof}

Lemma~\ref{lem:affine-boundary-invariant-tail} gives the stationary
tail and moment classification in
Theorem~\ref{thm:affine-boundary}. Together with
Lemma~\ref{lem:kappa-two-three-persistence}, it also rules out
geometric \(\beta\)-mixing.

\begin{lemma}
\label{lem:affine-boundary-invariant-tail}
Suppose Assumptions~\ref{ass:affine-seed} and
\ref{ass:invertibility} hold, with \(\phi=1\), \(\omega<0\), and
\(\kappa=-2\omega/(\alpha^2\sigma^2)\).
Under any stationary solution, there is a constant \(C>0\) such that
\begin{equation}
 \Pr(M_0>y)\sim Cy^{-\kappa},\qquad y\to\infty.
 \label{eq:affine-boundary-state-tail}
\end{equation}
{For every \(q>0\), the moments \(\E M_0^q\) and
\(\E X_0^q\) are finite if and only if \(q<\kappa\).}
\end{lemma}

\begin{proof}
In this proof we use the notation of \citet[Section~8.2]{denisov2025markov}
for the truncation level \(s\), the functions
\(r,p,\mathcal R,\mathcal U\) and the constant \(b\). Their \(\mu\) is our
\(\mu_R\).

\proofstep{Drift and variance of \((R_t)\)}
Put
\[
 b=\frac{\alpha^2\sigma^2}{4},\qquad
 \mu_R=\frac{\alpha^2\sigma^2-4\omega}{8}.
\]
{With \(m=x^2\), taking \(\chi=1/2\) in~\eqref{eq:affine-power-drift}
and using~\eqref{eq:affine-sqrt-variance}, we obtain}
\(\mu_1(x)=-\mu_R/x+O(x^{-2})\) and \(\mu_2(x)=b+O(x^{-1})\).
By~\eqref{eq:sqrt-truncation}, the same expansions hold for
\(\mu_1^{[s]}\) and \(\mu_2^{[s]}\). Since
\(\kappa=-2\omega/(\alpha^2\sigma^2)\), we have
\(2\mu_R/b=1-4\omega/(\alpha^2\sigma^2)=2\kappa+1\), and therefore
\begin{equation}
 \frac{2\mu_1^{[s]}(x)}{\mu_2^{[s]}(x)}=-\frac{2\kappa+1}{x}+O(x^{-2}),
 \qquad x\to\infty.
 \label{eq:sqrt-drift-ratio}
\end{equation}

\proofstep{The functions \(r\), \(p\), \(\mathcal R\) and \(\mathcal U\)}
Fix \(\eta\in(0,\min(1,2\kappa))\) and set, for \(x\geq0\),
\[
 r(x)=\frac{2\kappa+1}{1+x},\qquad p(x)=(1+x)^{-1-\eta}.
\]
The function \(r\) is decreasing, \(r(x)\to0\) and
\(xr(x)\to2\kappa+1=2\mu_R/b\). The function \(p\) is positive, decreasing
and integrable on \([0,\infty)\), and \(r'(x)=O(x^{-2})\) and
\(p'(x)=O(x^{-2})\). Since \((2\kappa+1)/x-r(x)=(2\kappa+1)/\{x(1+x)\}\)
and \(\eta<1\), \eqref{eq:sqrt-drift-ratio} gives
\[
 \frac{2\mu_1^{[s]}(x)}{\mu_2^{[s]}(x)}=-r(x)+O(x^{-2})=-r(x)+o(p(x)).
\]
{This is condition (8.5) of \citet{denisov2025markov}, and the
bounds on \(r'\) and \(p'\) above are their (8.6).}
The functions defined in {\citet[eqs.~(8.8)--(8.9)]{denisov2025markov}} are then explicit.
For \(x>0\),
\[
 \mathcal R(x)=\int_0^x r(z)\,dz=(2\kappa+1)\log(1+x),
 \qquad
 \mathcal U(x)=\int_0^x e^{\mathcal R(z)}\,dz
 =\frac{(1+x)^{2\kappa+2}-1}{2\kappa+2},
\]
and \(\mathcal R(x)=\mathcal U(x)=0\) for \(x\leq0\). In particular,
\(0\leq\mathcal U(y)\leq(1+|y|)^{2\kappa+2}\) for every real \(y\), and
\(p(x)e^{\mathcal R(x)}=(1+x)^{2\kappa-\eta}\).

\proofstep{Hypotheses of Theorem~8.2 and Corollary~8.3 of \citet{denisov2025markov}}
{Theorem~8.2 of \citet{denisov2025markov} concerns a recurrent
chain \((R_t)\) whose stationary measure \(\pi_R\) satisfies
\(\pi_R(-\infty,x]<\infty\) for every \(x\) and which, from every initial
state, has \(\limsup_tR_t=\infty\) almost surely, their (8.10). It requires
the truncated moments to satisfy \(\mu_2^{[s]}(x)\to b>0\) and
\(x\mu_1^{[s]}(x)\to-\mu_R\) with \(\mu_R>-b/2\), for an increasing
truncation level \(s(x)=o(x)\), their (8.4). It also requires the expansion
(8.5) of the ratio \(2\mu_1^{[s]}/\mu_2^{[s]}\), the regularity condition
(8.6), and the four tail and integrability conditions (8.11)--(8.14) on the
increments. Corollary~8.3 adds \(2\mu_R>b\) and positive recurrence.}
\begin{enumerate}
\item[(i)] \emph{Recurrence.} {\citet[eq.~(7.33)]{denisov2025markov}} call a chain recurrent if, for some
\(\hat x\), it enters \((-\infty,\hat x]\) almost surely from every initial
state \(x>\hat x\).
By~\eqref{eq:affine-delta-mean}, there is \(m_1\) such that
\(\E[M_1-m\mid M_0=m]\leq\omega/2\) for all \(m\geq m_1\). Fix
\(K\geq m_1\), let \(\tau_K=\inf\{t\geq0:M_t\leq K\}\), and start from
\(M_0=m>K\). Each \(M_t\) is integrable by
Lemma~\ref{lem:affine-concentration} and induction, and
\(\bigl(M_{t\wedge\tau_K}-(\omega/2)(t\wedge\tau_K)\bigr)_{t\geq0}\) is a supermartingale. As
\(M_t\geq0\) and \(\omega<0\),
\[
 \frac{|\omega|}2\,\E[n\wedge\tau_K\mid M_0=m]
 \leq m-\E[M_{n\wedge\tau_K}\mid M_0=m]\leq m.
\]
Letting \(n\to\infty\) gives \(\E[\tau_K\mid M_0=m]\leq2m/|\omega|<\infty\).
Hence \((R_t)\) is recurrent in this sense, with \(\hat x=\sqrt K\).
\item[(ii)] \emph{Invariant measure.} The law \(\pi_R\) of \(R_0\) under the
stationary solution is an invariant probability measure of \((R_t)\), so
\(\pi_R(-\infty,x]\leq1\) for every \(x\). By (i) and (ii), \((R_t)\) is
recurrent with a finite invariant measure, which is positive recurrence in
the sense of \citet[Section~7.4]{denisov2025markov}.
\item[(iii)] \emph{Condition (8.10).}
Lemma~\ref{lem:affine-boundary-nonconfinement} gives
\(\limsup_{t\to\infty}R_t=\infty\) almost surely from every initial state.
\item[(iv)] \emph{Conditions (8.4)--(8.6).} The truncation level \(s\) is
increasing and \(s(x)=o(x)\), \(\mu_2^{[s]}(x)\to b>0\),
\(x\mu_1^{[s]}(x)\to-\mu_R\), and \(\mu_R-b/2=-\omega/2>0\). Conditions
(8.5)--(8.6) were verified in {Step~2}.
{\item[(v)] \emph{Conditions (8.11)--(8.14).} Since
\(0\leq\mathcal U(y)\leq(1+|y|)^{2\kappa+2}\leq2^{2\kappa+1}(1+|y|^{2\kappa+2})\),
and by \eqref{eq:sqrt-large-jump} with \(q=6\) for \(x\geq4\),
\[
 \begin{aligned}
 \sup_{x>0}\E\,\mathcal U(\xi(x))
 &\leq2^{2\kappa+1}(1+\bar\mu_{2\kappa+2}),\\
 \Pr\{|\xi(x)|>s(x)\}&\leq\bar\mu_6x^{-3}=o(p(x)/x),\\
 \E[|\xi(x)|^3\mathbf1_{\{|\xi(x)|\leq s(x)\}}]&\leq\bar\mu_3=o(x^2p(x)),\\
 \E[\mathcal U(\xi(x))\mathbf1_{\{\xi(x)>s(x)\}}]
 &\leq\sup_{y>0}\E\,\mathcal U(\xi(y))=o(p(x)e^{\mathcal R(x)}).
 \end{aligned}
\]
The first bound makes the ratio in (8.11) bounded, and the other three are
(8.12)--(8.14). The comparisons hold because
\(p(x)/x\asymp x^{-2-\eta}\), \(x^2p(x)\asymp x^{1-\eta}\), and
\(p(x)e^{\mathcal R(x)}=(1+x)^{2\kappa-\eta}\), with \(\eta<1\) and
\(\eta<2\kappa\).}
\end{enumerate}

\proofstep{Conclusion}
Since \(2\mu_R>b\) and \((R_t)\) is positive recurrent,
Corollary~8.3 of {\citet{denisov2025markov}} applies with \(\rho_R=2\mu_R/b+1=2\kappa+2\). It gives a
constant \(C_R>0\) such that
\[
 \Pr(R_0>x)\sim\frac{C_R}{\rho_R-2}\,\frac{x^2}{\mathcal U(x)}
 \sim\frac{(\kappa+1)C_R}{\kappa}\,x^{-2\kappa},
 \qquad x\to\infty.
\]
Since \(M_0=R_0^2\), \(\Pr(M_0>y)=\Pr(R_0>\sqrt y)\), which
proves~\eqref{eq:affine-boundary-state-tail} with
\(C=(\kappa+1)C_R/\kappa\). For \(q>0\),
\(\E M_0^q=\int_0^\infty qy^{q-1}\Pr(M_0>y)\,dy\), which is finite if and
only if \(q<\kappa\). By Lemma~\ref{lem:affine-state-observation-moments},
the same holds for \(\E X_0^q\).
\end{proof}

\begin{lemma}
\label{lem:kappa-two-three-persistence}
{Suppose Assumptions~\ref{ass:affine-seed} and
\ref{ass:invertibility} hold, with \(\phi=1\) and \(\omega<0\).}
There are \(a_*,p_*,m_0>0\) such that, for every
\(m\geq m_0\) and every integer \(0\leq j\leq a_*m\),
\begin{equation}
 \Pr(M_k\in[m/2,2m]\text{ for all }0\leq k\leq j,\,
             X_0\geq m/2,\ X_j\geq m/4\mid M_0=m)\geq p_*.
 \label{eq:kappa-two-three-persistence}
\end{equation}
\end{lemma}

\begin{proof}
Equation~\eqref{eq:affine-delta-mean} and positivity of the softplus
remainder imply, for all sufficiently large \(x\),
\[
 \omega\leq\E(M_{t+1}-M_t\mid M_t=x)\leq\omega/2<0.
\]
Since \(s_c\) is one-Lipschitz and
\(\Var(X_t\mid M_t=x)=\sigma^2x\),
\[
 \Var(M_{t+1}\mid M_t=x)\leq\alpha^2\sigma^2x.
\]

{Choose \(a_*>0\) so small that}
\[
 |\omega|a_*<\frac18,\qquad
 32\alpha^2\sigma^2a_*<\frac14.
\]
Starting from \(M_0=m\), let
\[
 \tau_m=\inf\{k\geq0:M_k\notin[m/2,2m]\},
 \qquad k_m=\lfloor a_*m\rfloor,
\]
and define the stopped martingale \((\mathcal M_\ell)_{\ell\geq0}\) by
\[
 \mathcal M_\ell=\sum_{k=0}^{\ell-1}
 [M_{k+1}-M_k-\E(M_{k+1}-M_k\mid M_k)]
 \mathbf1_{\{k<\tau_m\}}.
\]
For all large \(m\),
\[
 \E[\mathcal M_{k_m}^2\mid M_0=m]\leq2\alpha^2\sigma^2mk_m.
\]
Enlarge \(m_0\), if necessary, so that every point of
\([m/2,2m]\) lies in the region where the preceding drift bounds hold.
Doob's inequality then gives
\[
 \Pr\!\left(\max_{\ell\leq k_m}|\mathcal M_\ell|\geq m/4\,\middle|\,M_0=m\right)\leq\frac14.
\]
Put
\(b(x)=\E(M_{t+1}-M_t\mid M_t=x)\). On
\(\{\max_{\ell\leq k_m}|\mathcal M_\ell|<m/4\}\), a first exit
\(\tau_m\leq k_m\) would satisfy, by {the definition of
\(\mathcal M_\ell\)}
and
the drift bounds before exit,
\[
 \begin{aligned}
 M_{\tau_m}&=m+\mathcal M_{\tau_m}+\sum_{k<\tau_m}b(M_k),\\
 \frac{5m}{8}&<m-\frac m4-|\omega|k_m
 \leq M_{\tau_m}<m+\frac m4=\frac{5m}{4}.
 \end{aligned}
\]
This contradicts \(M_{\tau_m}\notin[m/2,2m]\). Define
\[
 E_{m,j}:=\{\tau_m>j\}
 =\{M_k\in[m/2,2m]\text{ for every }0\leq k\leq j\}.
\]
We have proved
\(\Pr(E_{m,k_m}\mid M_0=m)\geq3/4\), and hence
\(\Pr(E_{m,j}\mid M_0=m)\geq3/4\) for every \(0\leq j\leq k_m\).

Chebyshev's inequality and
\(\Var(X_0\mid M_0=m)=\sigma^2m\) give
\[
 \Pr(X_0<m/2\mid M_0=m)\leq\frac{4\sigma^2}{m}.
\]
For \(1\leq j\leq k_m\), the event \(E_{m,j}\) is
\(\mathcal F_{j-1}\)-measurable because \(M_j\) is determined before
\(X_j\) is drawn. On this event, \(M_j\geq m/2\), so
\[
 \Pr(X_j<m/4\mid\mathcal F_{j-1})
 \leq\Pr(|X_j-M_j|\geq M_j/2\mid\mathcal F_{j-1})
 \leq\frac{4\sigma^2}{M_j}\leq\frac{8\sigma^2}{m}.
\]
Taking expectations after multiplying by \(\mathbf1_{E_{m,j}}\) gives
\[
 \Pr(E_{m,j},X_j<m/4\mid M_0=m)\leq\frac{8\sigma^2}{m}.
\]
A union bound therefore yields
\[
 \Pr(E_{m,j},X_0\geq m/2,X_j\geq m/4\mid M_0=m)
 \geq\frac34-\frac{12\sigma^2}{m}.
\]
For \(j=0\), the initial Chebyshev bound suffices. Increasing \(m_0\)
so that \(m_0\geq48\sigma^2\) proves
\eqref{eq:kappa-two-three-persistence} with \(p_*=1/2\).
\end{proof}

\begin{proof}[{Proof of Theorem~\ref{thm:affine-boundary}}]
Use the transition kernel \(P\) and {the functions
\(V_\chi(m)=1+m^\chi\)}
from the proof of Theorem~\ref{thm:affine-classification}.

\proofstep{Existence and moment classification when \(\omega<0\)}
{Fix \(0<s<\kappa\) and take \(\chi=s+1\). Since \(s<\kappa\) is
equivalent to \(\omega+s\alpha^2\sigma^2/2<0\),
equation~\eqref{eq:affine-power-drift} yields}
\[
 (PV_\chi)(m)\leq V_\chi(m)-C_sm^s+B_s
\]
for suitable \(B_s,C_s>0\).  Summing this inequality along the chain and
dividing by the number of steps bounds the Ces\`aro average of
\(M_t^s\).  As in the proof of
Theorem~\ref{thm:affine-classification}, extend \(G_0=\delta_0\){, so
that the kernel is Feller on \([0,\infty)\) and}
\(P(m,\{0\})=0\).  The moment bound makes the Ces\`aro laws tight, so
Lemma~\ref{lem:feller-invariant-limits} gives an invariant law on
\((0,\infty)\). As in the proof of
Theorem~\ref{thm:affine-classification}, this invariant law yields a
stationary solution. Lemma~\ref{lem:affine-boundary-invariant-tail} applies to
this invariant law and gives both the stated tail asymptotic and the exact
moment classification, including the critical order \(s=\kappa\).

\proofstep{Nonnegative intercepts}
Lemma~\ref{lem:affine-general-state-boundary} proves part~2.

\proofstep{Uniqueness and polynomial mixing}
Suppose first that Assumption~\ref{ass:affine-seed}(ii)(a) holds.
Lemma~\ref{lem:finite-activity-wasserstein-coupling} gives uniqueness
for every \(\kappa>0\).  For the mixing assertion, fix
{\(0<r<\kappa\)}, choose \(r<s<\kappa\), and then choose
\(s+1<a<\kappa+1\).  The power drift at exponent \(a\) gives
\[
 PV_a\leq V_a-c_aV_a^{1-1/a}+B_a.
\]
{Set \(V=V_a\geq1\) and \(\psi_a(v)=c_av^{1-1/a}\), which is
increasing and concave with \(\psi_a(0)=0\) and \(\psi_a'(v)\to0\). The
preceding drift is \(PV\leq V-\psi_a(V)+B_a\). Choose \(u_a\geq1\) so that
\(\psi_a(u_a)>2B_a\), and put
\[
 \Delta_a=\{V_a\leq u_a\}^2,\qquad
 \varphi_a(v)=\left(1-\frac{2B_a}{\psi_a(u_a)}\right)\psi_a(v).
\]
By \citet[Theorem~4]{durmusFortMoulines2016}, this single-chain drift gives
their pair-drift condition H2 on the coupling set \(\Delta_a\), with rate
function \(\varphi_a\). Choose \(R<\infty\) with
\(\{V_a\leq u_a\}\subset[0,R]\).
Lemma~\ref{lem:finite-activity-wasserstein-coupling} then gives their
coupling condition H1 on \(\Delta_a\), namely \(Qd\leq d\) everywhere and
\(Q^{N_R}d\leq(1-\delta_R)d\) on \(\Delta_a\). Since}
\(H_{\varphi_a}^{-1}(n)\asymp n^a\), where
\(H_{\varphi_a}(v)=\int_1^vdu/\varphi_a(u)\), their Theorem~3(i) gives
\[
 \mathcal W_d(P^n(m,\cdot),\pi)
 \leq C_a\left[\frac{V_a(m)}{n^a}
       +\frac1{n^{a-1}}+\left(\frac{\log n}{n}\right)^a\right],
 \qquad n\geq3.
\]
Because \(d\leq1\), the first term can be replaced by
\(1\wedge(C_aV_a(m)/n^a)\).  Consequently,
\[
 \int\mathcal W_d(P^n(m,\cdot),\pi)\,\pi(dm)
 \leq C\left[n^{-s}\pi(V_a^{s/a})+n^{-(a-1)}
                    +(\log n/n)^a\right]=O(n^{-s}).
\]
Here \(\pi(V_a^{s/a})<\infty\) follows from the {stationary moment
of order \(s<\kappa\) proved above.}
Equation~\eqref{eq:finite-activity-wasserstein-beta}
and \(s>r\) now give \(n^r\beta_X^{\rm mix}(n)\to0\).

Suppose now that Assumption~\ref{ass:affine-seed}(ii)(b) holds.  Here
\(\alpha=1-\beta>0\).  Lemma~\ref{lem:infinite-activity-density-overlap}
shows that the chain is irreducible and strongly aperiodic on \(E_\beta\), that bounded sets are petite, and that every invariant law is
supported there.  Fix any \(0<r<\kappa\), put
\(\chi=r+1\), and take a sufficiently large bounded set \(C\).  The power
drift has the form
\[
 PV_\chi\leq V_\chi-C_\chi V_\chi^{1-1/\chi}
             +B_\chi\mathbf1_C.
\]
The positive Harris recurrence criterion
\citep[Theorem~11.3.4, p.~265]{meyn1993markov} gives positive Harris recurrence
and hence uniqueness.  With
\(\varphi(v)=C_\chi v^{1-1/\chi}\), the rate associated with this drift is
determined by
\[
 H_\varphi(v)=\int_1^v\frac{du}{\varphi(u)}
 =\frac{\chi}{C_\chi}(v^{1/\chi}-1),
 \qquad
 r_\varphi(j)=\varphi(H_\varphi^{-1}(j))
 \asymp j^{\chi-1}=j^r.
\]
{With \(V=V_\chi\), the bounded petite set \(C\), and this
\(\varphi\), the preceding drift is the subgeometric drift condition of
\citet{meitz2021subgeometric}. For the \(\psi\)-irreducible, aperiodic
chain, their Theorem~2(d) gives, in the stationary case and without
requiring \(\pi(V_\chi)<\infty\), that
\(r_\varphi(j)\beta_M^{\rm mix}(j)\to0\). The rate calculation therefore
gives \(j^r\beta_M^{\rm mix}(j)\to0\), even when \(\pi(V_\chi)=\infty\).}
By \eqref{eq:state-to-observation-beta-transfer},
\(j^r\beta_X^{\rm mix}(j)\to0\) as well.
{In both cases, therefore, \(j^r\beta_X^{\rm mix}(j)\to0\) for every
\(0<r<\kappa\).}

\proofstep{Lack of geometric \(\beta\)-mixing}
{For the stationary observations, write
\[
 \alpha_X^{\rm mix}(j)
 =\sup_{\substack{A\in\mathcal F_0^X\\ B\in\mathcal F_j^{X,+}}}
   |\Pr(A\cap B)-\Pr(A)\Pr(B)|
\]
for the strong-mixing coefficient.  It satisfies
\(\alpha_X^{\rm mix}(j)\leq\beta_X^{\rm mix}(j)\).
Let \(a_*,p_*,m_0\) be supplied by
Lemma~\ref{lem:kappa-two-three-persistence}.  Fix \(b>1/a_*\) and
\(q\in(\kappa/2,\kappa)\), and, for integers \(n\geq1\), set
\[
 A_n=\{X_0\geq bn/2\},\qquad C_n=\{X_n\geq bn/4\}.
\]
For all large \(n\), every \(m\geq bn\) satisfies \(m\geq m_0\) and
\(n\leq a_*m\). Given \(M_0=m\), the observations
\(X_0,X_1,\ldots\) follow the model initialized at \(M_0=m\), and for
\(m\geq bn\) the event in Lemma~\ref{lem:kappa-two-three-persistence} at
horizon \(j=n\) is contained in \(A_n\cap C_n\). Writing \(C>0\) for the
constant in the exact stationary tail, we obtain
\[
 \Pr(A_n\cap C_n)\geq p_*\Pr(M_0\geq bn)
 \sim p_*Cb^{-\kappa}n^{-\kappa}.
\]
Lemma~\ref{lem:affine-boundary-invariant-tail} gives \(\E X_0^q<\infty\), so
Markov's inequality and stationarity give
\[
 \Pr(A_n)\Pr(C_n)\leq(\E X_0^q)^2\left(\frac{b^2n^2}8\right)^{-q}
 =o(n^{-\kappa}),
\]
because \(2q>\kappa\). Since \(A_n\in\mathcal F_0^X\) and
\(C_n\in\mathcal F_n^{X,+}\), there is a constant
\(C_{\mathrm{mix}}>0\) such that, for all large \(n\),
\[
 \beta_X^{\rm mix}(n)\geq\alpha_X^{\rm mix}(n)
 \geq\Pr(A_n\cap C_n)-\Pr(A_n)\Pr(C_n)
 \geq C_{\mathrm{mix}}n^{-\kappa}.
\]}
This polynomial lower bound rules out geometric \(\beta\)-mixing
of \((X_t)\).
In particular, \(e^{\lambda_{\rm mix}n}\beta_X^{\rm mix}(n)\to\infty\)
for every \(\lambda_{\rm mix}>0\).
\end{proof}

\subsection{Proof of Theorem~\ref{thm:kappa-two-three-long-memory}}
\label{app:kappa-two-three}

\leavevmode{Part~1 follows from Davydov's inequality. For part~2,
association bounds the covariance below by that of the upper-tail parts of
\(X_0\) and \(X_j\). The stationary tail and persistence at large states make
this lower bound of order \(j^{2-\kappa}\), which is not summable when
\(2<\kappa\leq3\).}

{A random vector \(Y\) is associated if \(\Cov(f(Y),g(Y))\geq0\)
for all coordinatewise nondecreasing \(f\) and \(g\) for which the covariance
exists \citep{esary1967association}.}

\begin{lemma}
\label{lem:kappa-two-three-association}
{Under the assumptions of part~2 of
Theorem~\ref{thm:kappa-two-three-long-memory}, every finite vector of the
stationary observation process is associated.}
\end{lemma}

\begin{proof}
The convolution identity \(G_{m'}=G_m*G_{m'-m}\), \(m'\geq m\), shows
that \(G_m\) is stochastically increasing in \(m\).  Hence its
generalized quantile \(F_m^{\leftarrow}(u)\), \(0<u<1\), may be chosen jointly
measurable and nondecreasing in both arguments.  Indeed, weak continuity
of \(m\mapsto G_m\) makes \((m,x)\mapsto G_m([0,x])\) a Borel
distribution function, and one may take
\[
 F_m^{\leftarrow}(u)=\inf\{x\geq0:G_m([0,x])\geq u\}.
\]
The definition makes \(F_m^{\leftarrow}\) nondecreasing in \(u\), while stochastic
ordering makes it nondecreasing in \(m\).  Let
\(\mathcal U_0,\mathcal U_1,\ldots\) be
independent uniform random
variables, independent of an invariant initial state \(M_0\), and put
\(X_t=F_{M_t}^{\leftarrow}(\mathcal U_t)\) and
\(M_{t+1}=s_c(\omega+\alpha X_t+\beta M_t)\).
Because \(\alpha>0\), \(\beta\geq0\), and \(s_c\) is increasing, every
finite observation vector is a coordinatewise nondecreasing function of
the independent variables \(M_0,\mathcal U_0,\mathcal U_1,\ldots\).  Independent variables
are associated, and association is preserved by coordinatewise
nondecreasing transformations. See \citet{esary1967association}.
\end{proof}

\begin{proof}[{Proof of Theorem~\ref{thm:kappa-two-three-long-memory}}]
For part~1, suppose \(\kappa>3\) and choose \(q\in(3,\kappa)\).
Theorem~\ref{thm:affine-boundary} gives \(\E|X_0|^q<\infty\) and
\(j^q\beta_X^{\rm mix}(j)\to0\).  Since the strong-mixing coefficient
satisfies \(\alpha_X^{\rm mix}(j)\leq\beta_X^{\rm mix}(j)\), Davydov's
covariance inequality \citep{davydov1968convergence} yields
\[
 |\Cov(X_0,X_j)|
 \leq12\|X_0\|_q^2(\alpha_X^{\rm mix}(j))^{1-2/q}
 =o(j^{-(q-2)}).
\]
Since \(q-2>1\) and \(X_0\) has positive finite
variance, division by \(\Var(X_0)\) and summation over \(j\) prove
part~1.

For part~2, Theorem~\ref{thm:affine-boundary} gives a stationary solution
that is unique in law {with}
\(\E X_0^2<\infty\){, and}
\[
 \Var(X_0)\geq\E[\Var(X_0\mid\mathcal F_{-1})]
 =\sigma^2\E M_0>0,
\]
{so \(\rho_X(j)\) is well defined. By
Lemma~\ref{lem:kappa-two-three-association},}
\(\Cov(X_0,X_j)\geq0\), and hence \(\rho_X(j)\geq0\).

The maps \(x\mapsto\max(x-j,0)\) and
\(x\mapsto\min(x,j)\) are increasing and satisfy
\(x=\max(x-j,0)+\min(x,j)\). Expanding the covariance and
using association to show that the other three terms are nonnegative gives
\begin{equation}
 \Cov(X_0,X_j)
 \geq\Cov(\max(X_0-j,0),\max(X_j-j,0)).
 \label{eq:kappa-two-three-truncation-covariance}
\end{equation}

Choose \(A\geq8\) such that \(A^{-1}\leq a_*\), where
\(a_*\) is from
Lemma~\ref{lem:kappa-two-three-persistence}.
For all sufficiently large \(j\), conditional on
\(M_0=m\geq Aj\), that lemma applies at horizon \(j\).  On its event,
\[
 \max(X_0-j,0)\geq\frac{3m}{8},
 \qquad
 \max(X_j-j,0)\geq\frac m8.
\]
Integrating this conditional bound over \(M_0\geq Aj\) gives
\[
 \E\!\left[\max(X_0-j,0)\max(X_j-j,0)\right]
 \geq
 C_*j^2\Pr(M_0\geq Aj),
\]
where \(C_*=3p_*A^2/64>0\).  The finite second moment gives
\[
 \E\max(X_0-j,0)
 \leq\E[X_0\mathbf1_{\{X_0>j\}}]
 \leq j^{-1}\E X_0^2.
\]
By stationarity, substitution in
\eqref{eq:kappa-two-three-truncation-covariance} gives, for all
sufficiently large \(j\),
\begin{equation}
 \Cov(X_0,X_j)
 \geq
 C_*j^2\Pr(M_0\geq Aj)
 -(\E X_0^2)^2j^{-2}.
 \label{eq:kappa-two-three-integrated-lower-bound}
\end{equation}

Lemma~\ref{lem:affine-boundary-invariant-tail} gives
\[
 j^2\Pr(M_0\geq Aj)
 \asymp j^{2-\kappa}.
\]
Since \(2<\kappa\leq3\), we have \(j^{-2}=o(j^{2-\kappa})\).
Hence
\eqref{eq:kappa-two-three-integrated-lower-bound}, followed by division by
the positive finite variance of \(X_0\), gives
\[
 \rho_X(j)\geq Cj^{-(\kappa-2)}
\]
for all sufficiently large \(j\).  This lower bound is nonsummable when
\(2<\kappa\leq3\), establishing the long-memory conclusion
in part~2.
\end{proof}

\subsection{Proof of Theorem~\ref{thm:affine-supercritical}}
\label{app:proof-affine-supercritical}

{Chebyshev's inequality shows that, from large states, one step
multiplies the state by at least a fixed factor \(\lambda>1\) with
probability close to one. Under an invariant law, the tail probabilities
at the points \(\lambda^nx\) would then stay bounded away from zero, which is
impossible.}

\begin{proof}[{Proof of Theorem~\ref{thm:affine-supercritical}}]
Since \(\phi>1\) and \(\beta<1\),
\(\alpha=\phi-\beta>0\).  Choose \(\lambda\in(1,\phi)\) and put
\[
 \epsilon=\frac{\phi-\lambda}{2\alpha}>0.
\]
Conditional on \(M_t=y\),
\[
 X_t-y\stackrel{d}{=}U_y.
\]
On the event \(\{X_t-y\geq-\epsilon y\}\), the inequality
\(s_c(x)\geq x\) gives
\[
 M_{t+1}
 =s_c(\omega+\phi y+\alpha(X_t-y))
 \geq \omega+(\phi-\alpha\epsilon)y
 =\omega+\frac{\phi+\lambda}{2}y.
\]
Consequently, there is \(m_{\rm esc}<\infty\) such that, for
\(y\geq m_{\rm esc}\),
\[
 X_t-y\geq-\epsilon y
 \quad\Longrightarrow\quad M_{t+1}\geq\lambda y.
\]
{Chebyshev's inequality and \(\Var(U_y)=\sigma^2y\) give
\[
 \Pr(M_1<\lambda y\mid M_0=y)
 \leq \Pr(U_y<-\epsilon y)
 \leq \frac{\sigma^2}{\epsilon^2y},
 \qquad y\geq m_{\rm esc}.
\]
Suppose that an invariant probability \(\pi\) existed, and put
\(\bar F(x)=\pi((x,\infty))\). Lemma~\ref{lem:affine-concentration} gives
unbounded support of \(G_m\). Because \(\alpha>0\) and \(s_c\) is
increasing and unbounded, \(P(m,(x,\infty))>0\) for every \(m>0\) and
\(x<\infty\), so invariance gives \(\bar F(x)>0\). For \(x\geq m_{\rm esc}\)
and \(m>x\), the event \(\{M_1\geq\lambda m\}\) is contained in
\(\{M_1>\lambda x\}\). Invariance and the preceding bound therefore give
\[
 \bar F(\lambda x)
 \geq\int_{(x,\infty)}\left(1-\frac{\sigma^2}{\epsilon^2m}\right)\pi(dm)
 \geq\left(1-\frac{\sigma^2}{\epsilon^2x}\right)\bar F(x).
\]
Fix \(x\geq m_{\rm esc}\) with \(\sigma^2\leq\epsilon^2x(1-\lambda^{-1})/2\).
Iteration gives, for every \(n\geq1\),
\[
 \bar F(\lambda^nx)
 \geq\bar F(x)\prod_{k=0}^{n-1}\left(1-\frac{\sigma^2}{\epsilon^2\lambda^kx}\right)
 \geq\bar F(x)\left(1-\sum_{k=0}^{n-1}\frac{\sigma^2}{\epsilon^2\lambda^kx}\right)
 \geq\frac{\bar F(x)}2.
\]
The left-hand side tends to zero as \(n\to\infty\), a contradiction. Hence
no invariant probability exists. A strictly stationary solution of the full
model would make the marginal law of \(M_0\) invariant for the Markov chain
\((M_t)\), so no such solution exists either.}
\end{proof}

\section{\texorpdfstring{{Proofs for Section~\ref{sec:likelihood-consequences}}}{Proofs for Section 3}}
\label{app:inference-proofs}

\subsection{\texorpdfstring{{Notation and auxiliary lemmas}}{Notation and auxiliary lemmas}}
\label{app:inference-notation}

\leavevmode{The following lemma constructs the filter \(M_t(\theta)\) of
Section~\ref{sec:poisson-qmle} at every \(\theta\in\Theta\), including
points with \(\phi>1\). A solution of recursion~\eqref{mrecursion} at
\(\theta\), computed from \((X_t)\), is causal if \(M_t(\theta)\) is the same
measurable function of \(X_{t-1},X_{t-2},\ldots\) for every \(t\).}

\begin{lemma}
\label{lem:unrestricted-candidate-filter}
Let \((X_t)\) be any stationary nonnegative process with
\(\E X_0<\infty\).  {Suppose that
\(\Theta\subset\mathbb R\times\mathbb R\times[-\rho,\rho]\) is compact
for some \(\rho\in(0,1)\).}
\begin{enumerate}
\item[(i)]
For every \(\theta\in\Theta\), recursion~\eqref{mrecursion},
evaluated at \(\theta\) using the observations \((X_t)\), has a unique stationary causal solution
\((M_t(\theta))\).
For each \(t\), \(M_t(\theta)\) is continuous in
\(\theta\) almost surely and satisfies
\[
 \sup_{\theta\in\Theta}M_t(\theta)
 \leq C\sum_{j=0}^\infty\rho^j(1+X_{t-1-j})<\infty
 \quad\text{a.s.}
\]
\item[(ii)]
For two solutions \((M_t^{(1)}(\theta))\) and \((M_t^{(2)}(\theta))\) of recursion~\eqref{mrecursion} using the same sequence \((X_t)\) and arbitrary finite initial values,
\[
 \sup_{\theta\in\Theta}
 |M_t^{(1)}(\theta)-M_t^{(2)}(\theta)|
 \leq\rho^t\sup_{\theta\in\Theta}
 |M_0^{(1)}(\theta)-M_0^{(2)}(\theta)|.
\]
\item[(iii)]
If, for some \(\theta\in\Theta\), a jointly stationary process
\(((X_t,M_t^*))\) with \(\E M_0^*<\infty\) satisfies
recursion~\eqref{mrecursion} at \(\theta\), then
\(M_t^*=M_t(\theta)\) almost surely, {and}
\(M_t^*\) is measurable with respect to
{\(\mathcal F_{t-1}^{X}=\sigma(X_s:s\leq t-1)\)}.
\end{enumerate}
\end{lemma}

\begin{proof}
This is the uniformly contractive case of the stochastic-recurrence
construction in \citet{bougerol1993} and
\citet[Theorem~2.8]{straumannMikosch2006}.  Because those results assume
ergodic driving maps, we give a direct proof for merely stationary \((X_t)\)
that also records the explicit uniform bounds.
For \(F_{t,\theta}(m)=s_c(\omega+\alpha X_{t-1}+\beta m)\), the
one-Lipschitz property of \(s_c\) gives, uniformly in \(\theta\),
\[
 |F_{t,\theta}(m)-F_{t,\theta}(m')|
 \leq\rho|m-m'|,
 \qquad
 F_{t,\theta}(0)\leq C(1+X_{t-1}).
\]
Successive backward iterates from zero differ by at most
\(C\rho^j(1+X_{t-1-j})\), whose sum has finite expectation.  Hence
they converge uniformly in \(\theta\), almost surely, to a
{causal solution}
\((M_t(\theta))\).  The tail error is bounded by
\(C\sum_{j\geq k}\rho^j(1+X_{t-1-j})\to0\) almost surely.  Each finite
iterate is continuous in \(\theta\), so \(M_t(\cdot)\) is continuous.
Iteration of \(F_{t,\theta}(m)\leq F_{t,\theta}(0)+\rho m\) gives the
stated envelope, and forward contraction gives the initialization bound.

Let \((\bar M_t)\) be either \((M_t^*)\) or another stationary causal solution at the
same parameter.  {The pair \((\bar M_t,M_t(\theta))\)}
is jointly stationary, and for every
\(j\geq1\),
\[
 |\bar M_t-M_t(\theta)|
 \leq\rho^j
 |\bar M_{t-j}-M_{t-j}(\theta)|.
\]
For \((\bar M_t)=(M_t^*)\), take expectations and let \(j\to\infty\).  For an
arbitrary stationary causal solution \((\bar M_t)\), the right side converges to zero in
probability because the factor following \(\rho^j\) has a fixed finite
distribution.  Thus \(\bar M_t=M_t(\theta)\) almost surely in both cases.
\end{proof}

For the model-specific lemmas and inference proofs, fix \(c>0\) and
work under {the stationary law} at
\(\theta_0=(\omega_0,\alpha_0,\beta_0)^\top\) in
Theorem~\ref{thm:affine-qmle}.  Expectations are under this law, and
\(\sigma_0^2\) denotes the true dispersion parameter from
Section~\ref{sec:poisson-qmle}.  Write \(M_t\) for the true conditional
mean and {\(M_t(\theta)\) for the filter of
Lemma~\ref{lem:unrestricted-candidate-filter}, computed from the same
observations \((X_t)\). We call it the stationary filter, to distinguish it
from the initialized filter \(\widetilde M_t(\theta)\).}
In this setting, \(\E X_0=\E M_0<\infty\), so
Lemma~\ref{lem:unrestricted-candidate-filter} gives
\(M_t(\theta_0)=M_t\) almost surely.
All derivatives are with respect to \(\theta=(\omega,\alpha,\beta)^\top\).
Define {the linear index \(\eta_t(\theta)\) and the filter
error \(\delta_t(\theta)\) by}
\[
 \eta_t(\theta)=\omega+\alpha X_{t-1}+\beta M_{t-1}(\theta),
 \qquad \delta_t(\theta)=M_t(\theta)-M_t,
\]
\[
 \begin{aligned}
 D_t(\theta)&=\partial_\theta M_t(\theta),&
 D_t^{(2)}(\theta)&=\partial_{\theta\theta^\top}^2M_t(\theta),\\
 A_t(\theta)&=\partial_\theta\eta_t(\theta),&
 J_t(\theta)&=\partial_{\theta\theta^\top}^2\eta_t(\theta),
 \end{aligned}
\]
\[
 e_t(\theta)=(1,X_{t-1},M_{t-1}(\theta))^\top,
 \qquad e_\beta=(0,0,1)^\top.
\]
For the quasi-log-likelihood and its derivatives, use
\[
 \ell_t^Q(\theta)=X_t\log M_t(\theta)-M_t(\theta),
 \qquad S_t(\theta)=\partial_\theta\ell_t^Q(\theta),
\]
\[
 \mathcal H_t(\theta)=-\partial_{\theta\theta^\top}^2\ell_t^Q(\theta),
 \qquad K_t(\theta)=\frac{D_t(\theta)D_t(\theta)^\top}{M_t(\theta)}.
\]
Tildes denote the corresponding quantities computed from a fixed,
positive, parameter-independent initial value \(\widetilde M_0\).

{The inference proofs use the following elementary bounds on the
softplus function.}

\begin{lemma}
\label{lem:softplus-calculus}
Fix \(c>0\), and set
\(p_c=s_c'\), \(r_c=(\log s_c)'\), \(r_c'=(\log s_c)''\), and
\(j_c=p_c^2/s_c\).  For a constant \(C=C(c)\) and all \(x\),
\begin{align}
&0<p_c\leq1,\quad 0\leq s_c''\leq(4c)^{-1},\quad
 p_c/s_c\leq c^{-1},\quad s_c''/s_c\leq c^{-2},
 \label{eq:softplus-self-regularization}\\
&p_c\leq C\frac{s_c}{1+s_c},\quad
 r_c\leq\frac C{1+s_c},\quad
 s_c''+|r_c'|\leq C\frac{s_c}{(1+s_c)^3},
 \label{eq:qmle-sharp-softplus-derivatives}\\
&0\leq j_c\leq c^{-1},\qquad |j_c'|\leq3c^{-2},
 \label{eq:softplus-information-calculus}
\end{align}
{and \(|\log s_c(x)|\leq C(1+|x|)\).}
\end{lemma}

\begin{proof}
Set \(u=e^{x/c}\) and \(v=\log(1+u)\).  Direct differentiation gives
\[
 s_c=cv,\qquad p_c=\frac{u}{1+u},\qquad
 s_c''=\frac{u}{c(1+u)^2},\qquad
 r_c'=\frac{u(v-u)}{c^2(1+u)^2v^2}.
\]
Use \(u/(1+u)\leq v\).  For \(u\leq1\), use
\(u/2\leq v\leq u\) and \(0\leq u-v\leq u^2/2\). For \(u\geq1\), use
\(v\geq\log2\) and \(\sup_{u\geq1}v^2/u<\infty\).  These prove the derivative
bounds.  The comparisons \(ce^{x/c}/2\leq s_c(x)\leq ce^{x/c}\) for
\(x\leq0\) and \(c\log2\leq s_c(x)\leq x+c\log2\) for \(x\geq0\)
give the log bound.  Finally,
\(j_c'=2p_c s_c''/s_c-p_c^3/s_c^2\), which gives
\eqref{eq:softplus-information-calculus}.
\end{proof}

The consistency argument uses an integrable upper bound on the likelihood
contrast, allowing its negative part to have infinite expectation.
The next lemma gives the model-specific envelope.
{We use the convention \(0\log0=0\).}

\begin{lemma}
\label{lem:qmle-contrast-envelope}
{Let \(M_t\) be the true conditional mean and, for
\(\theta\in\Theta\), let \(M_t(\theta)\) be the filter of
Lemma~\ref{lem:unrestricted-candidate-filter}.}
For
\[
 L_t(\theta)
 =X_t\log\frac{{M_t(\theta)}}{M_t}
  -({M_t(\theta)}-M_t),
\]
one has
\begin{equation}
 L_t(\theta)\leq
 \bar L_t:=X_t\log\frac{X_t}{M_t}-X_t+M_t
 \leq\frac{(X_t-M_t)^2}{M_t},
 \qquad \E \bar L_t\leq\sigma_0^2.
 \label{eq:qmle-contrast-upper-envelope}
\end{equation}
{The conditional expectation of \(L_t(\theta)\) is}
\begin{equation}
 \E[L_t(\theta)\mid\mathcal F_{t-1}]
 =-M_t \mathfrak h\!\left(\frac{{M_t(\theta)}}{M_t}\right),
 \qquad \mathfrak h(r)=r-1-\log r.
 \label{eq:qmle-conditional-contrast}
\end{equation}
\end{lemma}

\begin{proof}
For fixed \(x\geq0\) and \(m>0\), maximizing
\(x\log(\mu/m)-(\mu-m)\) over \(\mu>0\) gives its value at \(\mu=x\), with the
usual boundary interpretation at \(x=0\).  Also
\(y\log y-y+1\leq(y-1)^2\) for \(y\geq0\).  These two observations give
the inequalities in \eqref{eq:qmle-contrast-upper-envelope}. Conditional
expectation and
\(\E[(X_t-M_t)^2\mid\mathcal F_{t-1}]=\sigma_0^2M_t\) give its final
claim.  Taking conditional expectation of the contrast gives
\eqref{eq:qmle-conditional-contrast}.
\end{proof}

\begin{lemma}
\label{lem:affine-qmle-envelopes}
Under {the stationary law at \(\theta_0\)} in
Theorem~\ref{thm:affine-qmle}, let the parameters range over its compact
set {\(\Theta\subset\mathbb R\times\mathbb R\times[-\rho,\rho]\), where \(\rho\in(0,1)\).}
For the stationary filters of
Lemma~\ref{lem:unrestricted-candidate-filter}, choose
\(\bar\rho\in(\sqrt{\rho},1)\) and define
\[
 Z_t=\sum_{j\geq0}\bar\rho^j(1+X_{t-1-j}).
\]
There is a deterministic \(C<\infty\), independent of \(t\) and
\(\theta\), such that almost surely
\begin{equation}
 \sup_{\theta\in\Theta}M_t(\theta)\leq CZ_t,
 \label{eq:qmle-level-envelope}
\end{equation}
\begin{equation}
 \begin{aligned}
 \sup_{\theta\in\Theta}
 [\|D_t(\theta)\|+\|A_t(\theta)\|]&\leq CZ_t,\\
 \sup_{\theta\in\Theta}
 [\|D_t^{(2)}(\theta)\|+\|J_t(\theta)\|]&\leq CZ_t^2,
 \end{aligned}
 \label{eq:qmle-filter-derivative-envelopes}
\end{equation}
and
\begin{align}
 \sup_{\theta\in\Theta}\|K_t(\theta)\|
 &\leq CZ_t^2,\label{eq:qmle-full-information-envelope}\\
 \sup_{\theta\in\Theta}
 \|\partial_{\theta\theta^\top}^2\ell_t^Q(\theta)\|
 &\leq C(1+X_t)(1+Z_t^2).
 \label{eq:qmle-full-hessian-envelope}
\end{align}
{If \(X_0\in L^2\), the right sides of
\eqref{eq:qmle-level-envelope}--\eqref{eq:qmle-full-information-envelope} are
integrable. If \(X_0\in L^3\), so is the right side of
\eqref{eq:qmle-full-hessian-envelope}.}
\end{lemma}

\begin{proof}
The calculus bounds used throughout this proof are collected in
Lemma~\ref{lem:softplus-calculus}.
Lemma~\ref{lem:unrestricted-candidate-filter} proves existence, uniqueness, and
initialization stability of \(M_t(\theta)\) for every \(\theta\in\Theta\),
including parameters with \(\alpha+\beta\geq1\).
Its filter bound gives
\eqref{eq:qmle-level-envelope} for the \(Z_t\) defined in the statement.
For filters with two starting values, one of which may be the stationary
random initialization, the lemma gives
\begin{equation*}
 \sup_\theta|M_t(\theta)-\widetilde M_t(\theta)|
 \leq\rho^t\sup_\theta|M_0(\theta)-\widetilde M_0(\theta)|.
\end{equation*}

{We next show that \(M_t(\theta)\) is twice continuously
differentiable. Since \(\bar\rho^2>\rho\), we may fix
\(\rho'\in(\rho,\bar\rho^2)\) and a compact box \(\Theta_{\rm box}\) whose
interior contains \(\Theta\) and whose \(\beta\)-coordinates lie in
\([-\rho',\rho']\). Lemma~\ref{lem:unrestricted-candidate-filter} applies to
\(\Theta_{\rm box}\) with \(\rho'\) in place of \(\rho\).
For \(k\geq1\), let \(M_{t,k}(\theta)\) be the value at
time \(t\) of recursion~\eqref{mrecursion}, evaluated at \(\theta\) and
started from zero at time \(t-k\). These backward iterates are infinitely
differentiable in \(\theta\), and their first and second derivatives satisfy
the recursions displayed below, with zero initial values. The iteration
argument given below for \(D_t\), \(J_t\), and \(D_t^{(2)}\), with \(\rho'\)
in place of \(\rho\), bounds these derivatives on \(\Theta_{\rm box}\) by
\(CZ_t\) and \(CZ_t^2\), uniformly in \(k\).
The iterates \(M_{t,k}\) and \(M_{t,k+1}\) solve the same recursion from
time \(t-k\) onward. At that time their values differ by at most
\(C(1+X_{t-k-1})\), and their first and second derivatives differ by at
most \(C(1+X_{t-k-1})\) and \(C(1+X_{t-k-1})^2\). Subtract the two sets of
recursions, use that \(s_c'\) and \(s_c''\) are bounded and Lipschitz
(\(|s_c'''|\leq c^{-2}\)), and iterate with \(Z_{t-j}\leq\bar\rho^{-j}Z_t\).
This gives
\[
 \begin{aligned}
 &\sup_{\theta\in\Theta_{\rm box}}
 \bigl[|M_{t,k+1}-M_{t,k}|
 +\|\partial_\theta M_{t,k+1}-\partial_\theta M_{t,k}\|
 +\|\partial_{\theta\theta^\top}^2M_{t,k+1}
   -\partial_{\theta\theta^\top}^2M_{t,k}\|\bigr]\\
 &\qquad\leq C(k+1)^2\left(\frac{\rho'}{\bar\rho^2}\right)^k
 (1+X_{t-k-1})^2Z_t^2.
 \end{aligned}
\]
Since \(\E\log(1+X_0)<\infty\) and \(\bar\rho^2>\rho'\),
\citet[Lemma~2.1]{straumannMikosch2006} makes the right-hand side summable
in \(k\), almost surely. Hence the iterates and their first two derivatives
converge uniformly on \(\Theta_{\rm box}\). The limit \(M_t(\theta)\) is
therefore twice continuously differentiable on the interior of
\(\Theta_{\rm box}\), which contains \(\Theta\), and passing to the limit in
the recursions for the iterates gives}
\[
 D_t(\theta)=s_c'(\eta_t(\theta))
 [e_t(\theta)+\beta D_{t-1}(\theta)]
\]
and
\[
 J_t(\theta)=\beta D_{t-1}^{(2)}(\theta)
  +e_\beta D_{t-1}(\theta)^\top+D_{t-1}(\theta)e_\beta^\top,
\]
\[
 D_t^{(2)}(\theta)=s_c''(\eta_t(\theta))
 A_t(\theta)A_t(\theta)^\top
 +s_c'(\eta_t(\theta))J_t(\theta),
 \qquad A_t(\theta)=e_t(\theta)+\beta D_{t-1}(\theta).
\]
{Since \(0<s_c'\leq1\), \(0\leq s_c''\leq1/(4c)\), and
\(|\beta|\leq\rho\), iteration of the first recursion yields
\(\sup_\theta\|D_t(\theta)\|\leq CZ_t\). In the recursion for
\(D_t^{(2)}\), the coefficient of \(D_{t-1}^{(2)}\) is at most \(\rho\) in
absolute value, and the remaining terms are bounded by \(CZ_t^2\). Dropping
the first \(j\) terms of \(Z_t\) gives \(Z_{t-j}\leq\bar\rho^{-j}Z_t\), hence
\[
 \sup_\theta\|D_t^{(2)}(\theta)\|
 \leq C\sum_{j\geq0}\rho^jZ_{t-j}^2
 \leq CZ_t^2\sum_{j\geq0}\Bigl(\frac{\rho}{\bar\rho^2}\Bigr)^j,
\]
and the last series converges because \(\bar\rho^2>\rho\). The recursion for
\(J_t\) then gives \(\sup_\theta\|J_t(\theta)\|\leq CZ_t^2\).}
The definition of \(A_t\) also gives
\(\sup_\theta\|A_t(\theta)\|\leq CZ_t\), completing
\eqref{eq:qmle-filter-derivative-envelopes}. For every \(p\geq1\)
with \(X_0\in L^p\), Minkowski's inequality gives
\[
 \|Z_t\|_{L^p}\leq\frac{1+\|X_0\|_{L^p}}{1-\bar\rho}.
\]
Moreover, \(K_t(\theta)=j_c(\eta_t(\theta))A_t(\theta)A_t(\theta)^\top\)
and \(0\leq j_c\leq c^{-1}\) give
\eqref{eq:qmle-full-information-envelope}.

{Using
\eqref{eq:softplus-self-regularization} in
\(\mathcal H_t(\theta)\), the preceding bounds yield
\eqref{eq:qmle-full-hessian-envelope}.}
For example, before taking suprema over \(\theta\), the two Hessian terms are
bounded by constant multiples of \(X_t\|A_t(\theta)\|^2\) and
\((1+X_t)[\|A_t(\theta)\|^2+\|J_t(\theta)\|]\), respectively.
{Second moments make \(Z_t^2\) integrable.} If \(X_0\in L^3\),
H\"older's inequality also gives
\(\E[X_tZ_t^2]\leq\|X_t\|_{L^3}\|Z_t\|_{L^3}^2<\infty\),
so the full Hessian bound is integrable.

\end{proof}

\begin{lemma}
\label{lem:qmle-initialization}
Under the assumptions of Theorem~\ref{thm:affine-qmle}, let
\(M_t(\theta)\) and \(\widetilde M_t(\theta)\) be the stationary and
initialized filters computed from the same observations \((X_t)\),
{where the initial value \(\widetilde M_0>0\) is a constant or an
almost surely finite random variable that does not depend on \(\theta\).}
On
every compact subset \(\mathcal K\subseteq\Theta\),
\[
 \sum_{t=1}^\infty\sup_{\theta\in \mathcal K}
 \left[
 |\widetilde\ell_t^Q(\theta)-\ell_t^Q(\theta)|
 +\|\widetilde S_t(\theta)-S_t(\theta)\|
 +\|\widetilde K_t(\theta)-K_t(\theta)\|
 \right]<\infty
 \quad\text{a.s.}
\]
\end{lemma}

\begin{proof}
{The proof follows \citet[Section~2]{straumannMikosch2006} and
keeps explicit bounds for all three sums.}
In this proof, filter quantities are evaluated at a generic
\(\theta\in\mathcal K\), whose argument is suppressed.
Untilded quantities use \(M_t(\theta)\), and tilded quantities use \(\widetilde M_t(\theta)\).  Since \(\widetilde M_0>0\) is
independent of \(\theta\), \(\widetilde D_0=0\). Subsequent derivatives
are obtained by differentiating
recursion~\eqref{mrecursion}, evaluated at \(\theta\).
Let
\[
 B_0=1+\sup_{\theta\in \mathcal K}\left[
 |\widetilde M_0(\theta)-M_0(\theta)|
 +\|\widetilde D_0(\theta)-D_0(\theta)\|\right].
\]
The stationary envelopes and fixed initialization give \(B_0<\infty\) almost
surely.
Lemma~\ref{lem:unrestricted-candidate-filter} gives
\[
 \sup_{\theta\in \mathcal K}|\widetilde M_t-M_t|\leq B_0\rho^t,
 \qquad
 \sup_{\theta\in \mathcal K}|\widetilde\eta_t-\eta_t|
 \leq B_0\rho^t.
\]
{By \eqref{eq:softplus-self-regularization}, \(\log s_c\) is
\(1/c\)-Lipschitz and \(s_c\) is one-Lipschitz, which gives}
\[
 \sup_{\theta\in \mathcal K}|\widetilde\ell_t^Q-\ell_t^Q|
 \leq CB_0\rho^t(1+X_t).
\]

Use \(Z_t\) from Lemma~\ref{lem:affine-qmle-envelopes}, with
\(\bar\rho\in(\sqrt{\rho},1)\).  Since
\(\E Z_0=(1+\E X_0)/(1-\bar\rho)<\infty\), this envelope requires only
the assumed first moment.
Subtracting their recursions, using boundedness of \(s_c''\), and iterating
gives
\begin{equation*}
 \sup_{\theta\in \mathcal K}
 [\|\widetilde A_t-A_t\|+\|\widetilde D_t-D_t\|]
 \leq CB_0\rho^t\sum_{j=1}^t(1+Z_j).
\end{equation*}
Use \(p_c,r_c,j_c\) as defined in Lemma~\ref{lem:softplus-calculus}.  Since
\[
 S_t(\theta)=(X_t r_c(\eta_t)-p_c(\eta_t))A_t,
 \qquad
 K_t(\theta)=j_c(\eta_t)A_tA_t^\top,
\]
the preceding bounds and \(\sup_{\theta\in \mathcal K}\|A_t(\theta)\|\leq CZ_t\)
give
\begin{align}
 \sup_{\theta\in \mathcal K}\|\widetilde S_t-S_t\|
 &\leq CB_0\rho^t(1+X_t)
       \left[1+Z_t+\sum_{j=1}^t(1+Z_j)\right],
 \label{eq:qmle-score-initialization-bound}\\
 \sup_{\theta\in \mathcal K}\|\widetilde K_t-K_t\|
 &\leq C\left[
 B_0\rho^t(1+Z_t)
 \left(1+Z_t+\sum_{j=1}^t(1+Z_j)\right)
 +B_0^2\rho^{2t}
   \left(\sum_{j=1}^t(1+Z_j)\right)^2
 \right].
 \label{eq:qmle-information-initialization-bound}
\end{align}
{Since \(X_t\) and \(Z_t\) are integrable, the stationary
sequences \((1+X_t)\) and \((1+Z_t)\) have finite logarithmic
moments. By \citet[Lemma~2.1]{straumannMikosch2006}, the sequences
\(\rho^{t/2}(1+Z_t)\) and \(\rho^{t/2}\sum_{j=1}^t(1+Z_j)\) converge to
zero exponentially fast almost surely, and hence so does
\(\rho^t[1+Z_t+\sum_{j=1}^t(1+Z_j)]\). A further application of the same
lemma, with this factor or with \(\rho^t\) multiplying \(1+X_t\) or
\(1+Z_t\), shows that the criterion bound and the majorants in
\eqref{eq:qmle-score-initialization-bound}--
\eqref{eq:qmle-information-initialization-bound} are almost surely
summable. The last term in
\eqref{eq:qmle-information-initialization-bound} is the square of an
exponentially vanishing sequence. The factors \(B_0,B_0^2\) are almost
surely finite, also when \(\widetilde M_0\) is random. This proves all three
assertions under a first moment.}
\end{proof}

The generic Hessian bound in Lemma~\ref{lem:affine-qmle-envelopes}
uses a third observation moment, which need not exist at the stationary
unit root. The next two lemmas obtain the local Hessian and information
envelopes needed there using only a first stationary moment.

\begin{lemma}
\label{lem:qmle-rare-relative-innovation}
Let \(((M_t,X_t))\) be the stationary solution under
Assumptions~\ref{ass:affine-seed} and \ref{ass:invertibility}, with
\(\phi_0=1\), \(\omega_0<0\), and \(\E M_t<\infty\).
For the innovations \((U_t)\) in \eqref{eq:affine-innovation},
define the weighted innovation sum
\[
 \mathsf V_t(r)=1+\sum_{j=0}^\infty r^j|U_{t-1-j}|,\qquad 0<r<1.
\]
Fix \(\varrho\in(0,1)\), and write \(\mathsf V_t=\mathsf V_t(\varrho)\).
Then \(\mathsf V_t\in L^2\){, and there is \(C_0<\infty\) such that,
for every \(\chi\geq1\),}
\begin{equation}
 \E\left[\mathsf V_t^{\chi}
 \mathbf1_{\{\mathsf V_t>{C_0}(1+M_t)\}}\right]<\infty.
 \label{eq:qmle-rare-relative-moment}
\end{equation}
\end{lemma}

\begin{proof}
At the unit root, \(\alpha_0=1-\beta_0>0\) and
\(M_{t+1}=s_c(\omega_0+M_t+\alpha_0U_t)\).
The affine conditional variance gives
\(\E U_t^2=\sigma_0^2\E M_t<\infty\), so Minkowski's inequality gives
\(\mathsf V_t\in L^2\).

Choose \(\varrho_0\in(\varrho,1)\) and then \(\epsilon>0\) so that
\(1-\alpha_0\epsilon>\varrho_0\).  Set
\[
 B_t=\mathbf1_{\{|U_t|>\epsilon(1+M_t)\}}.
\]
Since \(s_c(x)\geq x\), on \(\{B_t=0\}\),
\[
 1+M_{t+1}\geq(1-\alpha_0\epsilon)(1+M_t)+\omega_0
             \geq \varrho_0(1+M_t)+\omega_0.
\]
{Put \(K_0=-\omega_0/\varrho_0>0\). The following inequality
holds on both \(\{B_t=0\}\) and \(\{B_t=1\}\).}
\begin{equation}
 1+M_t\leq \varrho_0^{-1}(1+M_{t+1})+{K_0
 +(1+M_t+|U_t|)B_t}.
 \label{eq:qmle-backward-good-bad}
\end{equation}
Define
\[
 \Sigma_t=\sum_{j\geq0}\varrho^j(1+M_{t-1-j}),
 \qquad
 \Xi_t=\sum_{j\geq0}\varrho^j
 (1+M_{t-1-j}+|U_{t-1-j}|)B_{t-1-j}.
\]
Multiplying \eqref{eq:qmle-backward-good-bad} at time \(t-1-j\) by
\(\varrho^j\) and summing gives
\[
 \Sigma_t\leq \varrho_0^{-1}(1+M_t+\varrho\Sigma_t)
 +{\frac{K_0}{1-\varrho}+\Xi_t},
\]
and hence \(\Sigma_t\leq C_\Sigma(1+M_t+\Xi_t)\) for some finite
positive deterministic constant \(C_\Sigma\), depending only on the fixed
model parameters and the choices of \(\varrho,\varrho_0,\epsilon\).
Set \(C_{\rm rel}=1+\epsilon C_\Sigma\).  Splitting the innovations according
to \(B_{t-1-j}\) yields
\begin{equation*}
 \mathsf V_t\leq1+\epsilon\Sigma_t+\Xi_t
 \leq C_{\rm rel}(1+M_t+\Xi_t).
\end{equation*}
If \(C_0>2C_{\rm rel}\), then on
\(\{\mathsf V_t>C_0(1+M_t)\}\),
\begin{equation}
 \mathsf V_t\leq2C_{\rm rel}\Xi_t.
 \label{eq:qmle-V-by-Xi}
\end{equation}

{Fix \(\chi\geq1\).}
Since \(\mathcal L(U_t\mid M_t=m)=\mathcal L(U_m)\),
Lemma~\ref{lem:affine-concentration} gives, for arbitrary \(p\geq2\),
\[
 \E(|U_t|^p\mid M_t=m)\leq C_p(1+m^{p/2}).
\]
Choose \(p>\chi\) and \(p\geq2(\chi-1)\).  Markov's inequality gives, uniformly
for \(m\geq1\),
\[
 (1+m)^{\chi}\Pr(B_t=1\mid M_t=m)
 +\E(|U_t|^{\chi}B_t\mid M_t=m)
 \leq C_{\chi}(1+m).
\]
The uniform small-state conditional moments in
Lemma~\ref{lem:affine-concentration} give the same bound for \(0<m<1\).
Consequently,
\begin{equation*}
 \E[(1+M_t+|U_t|)^{\chi}B_t\mid M_t=m]
 \leq C_{\chi}(1+m).
\end{equation*}
{Taking expectations and using \(\E M_t<\infty\) gives
\((1+M_t+|U_t|)B_t\in L^\chi\). Minkowski's inequality and stationarity
then give
\[
 \|\Xi_t\|_{L^\chi}
 \leq\sum_{j\geq0}\varrho^j
 \|(1+M_{t-1-j}+|U_{t-1-j}|)B_{t-1-j}\|_{L^\chi}
 =\frac{\|(1+M_t+|U_t|)B_t\|_{L^\chi}}{1-\varrho}<\infty.
\]}
Equation~\eqref{eq:qmle-V-by-Xi} proves
\eqref{eq:qmle-rare-relative-moment}.
\end{proof}

\begin{lemma}
\label{lem:boundary-qmle-hessian-envelope}
Under the assumptions of Theorem~\ref{thm:affine-qmle}, suppose that
\(\phi_0=1\) and \(\E M_t<\infty\).  Use the inference notation from
Section~\ref{app:inference-notation}.  There is a compact convex
neighborhood \(\mathcal N\subseteq\operatorname{int}(\Theta)\)
of \(\theta_0\) such that
\begin{equation}
 \E\sup_{\theta\in \mathcal N}
 \|\partial_{\theta\theta^\top}^2\ell_t^Q(\theta)\|<\infty,
 \qquad
 \E\sup_{\theta\in \mathcal N}
 \frac{\|D_t(\theta)\|^2}{M_t(\theta)}<\infty.
 \label{eq:qmle-boundary-local-envelopes}
\end{equation}
\end{lemma}

\begin{proof}
Use the true state \(M_t\) and the innovations \(U_t\) defined above.
Choose \(\bar\rho\in(\sqrt{\rho},1)\), {where
\(\Theta\subset\mathbb R\times\mathbb R\times[-\rho,\rho]\),}
and put
\(\mathsf V_t=\mathsf V_t(\bar\rho)\). {Lemma~\ref{lem:qmle-rare-relative-innovation} gives}
\(\mathsf V_t\in L^2\) and, for each \(\chi\geq1\), the restricted
moment bound \eqref{eq:qmle-rare-relative-moment} on
\(\{\mathsf V_t>C_0(1+M_t)\}\). Since \(s_c(x)\geq x\),
recursion~\eqref{mrecursion} at the true unit-root parameter gives
\[
 M_{t-1}\leq M_t+|\omega_0|+\alpha_0|U_{t-1}|.
\]
Iteration yields, for every \(k\geq1\),
\begin{equation}
 M_{t-k}\leq M_t+k|\omega_0|
       +\alpha_0\sum_{\ell=1}^k|U_{t-\ell}|.
 \label{eq:qmle-backward-level-bound}
\end{equation}
Moreover, \(X_{t-k}\leq M_{t-k}+|U_{t-k}|\).
Recalling \(Z_t\) from Lemma~\ref{lem:affine-qmle-envelopes}, substitute
\eqref{eq:qmle-backward-level-bound} with \(k=j+1\) and interchange the two
nonnegative sums.  Because
\[
 \sum_{j\geq0}\bar\rho^j
       \sum_{\ell=1}^{j+1}|U_{t-\ell}|
 =\frac1{1-\bar\rho}
   \sum_{\ell\geq1}\bar\rho^{\ell-1}|U_{t-\ell}|,
\]
and \(\sum_{j\geq0}(j+1)\bar\rho^j<\infty\), this gives
\begin{equation}
 Z_t\leq C(M_t+\mathsf V_t).
 \label{eq:qmle-boundary-current-state-envelope}
\end{equation}

Using \(X_{t-1}=M_{t-1}+U_{t-1}\) and \(\phi_0=1\), we have
\[
 \eta_t(\theta)-\eta_t(\theta_0)
 = (\omega-\omega_0)+(\alpha-\alpha_0)U_{t-1}
   +(\phi-1)M_{t-1}+\beta \delta_{t-1}(\theta).
\]
For sufficiently small \(\varepsilon>0\), define the closed ball
\[
 \mathcal N_\varepsilon
 =\{\theta\in\mathbb R^3:\|\theta-\theta_0\|\leq\varepsilon\}
 \subseteq\operatorname{int}(\Theta).
\]
The one-Lipschitz property of softplus and the uniform bound on \(|\beta|\)
then give
\begin{align}
 \sup_{\theta\in \mathcal N_\varepsilon}|\delta_t(\theta)|
 &\leq C\varepsilon(M_t+\mathsf V_t),
 \label{eq:qmle-local-filter-boundary}\\
 \sup_{\theta\in \mathcal N_\varepsilon}\|A_t(\theta)\|
 &\leq C(M_t+\mathsf V_t),\qquad
 \sup_{\theta\in \mathcal N_\varepsilon}\|J_t(\theta)\|
 \leq C(M_t+\mathsf V_t)^2.
 \label{eq:qmle-local-derivatives-boundary}
\end{align}
Indeed, the preceding identity gives the one-step bound
\[
 \sup_{\theta\in \mathcal N_\varepsilon}|\delta_t(\theta)|
 \leq C\varepsilon(1+M_{t-1}+|U_{t-1}|)
 +\rho\sup_{\theta\in \mathcal N_\varepsilon}|\delta_{t-1}(\theta)|.
\]
Iteration and \eqref{eq:qmle-backward-level-bound} prove
\eqref{eq:qmle-local-filter-boundary}.  Iterating the derivative recursions
in the proof of Lemma~\ref{lem:affine-qmle-envelopes}, using
\eqref{eq:qmle-boundary-current-state-envelope}, proves
\eqref{eq:qmle-local-derivatives-boundary}. Here
\(\bar\rho^2>\rho\) controls the quadratic forcing.

We next prove that
\(\sup_{\theta\in \mathcal N_\varepsilon}\|J_t(\theta)\|\) is integrable.
The chain rule gives
\(D_t(\theta)=s_c'(\eta_t(\theta))A_t(\theta)\).  Define
\[
 D_t^*=\sup_{\theta\in \mathcal N_\varepsilon}\|D_t(\theta)\|,
 \qquad
 J_t^*=\sup_{\theta\in \mathcal N_\varepsilon}\|J_t(\theta)\|.
\]
{By Lemma~\ref{lem:affine-qmle-envelopes}, \(D_t^*\leq CZ_t\), and
\(\E Z_t<\infty\) because \(\E X_0<\infty\). Hence \(\E D_t^*<\infty\).}
Split \(\mathcal N_\varepsilon\) according as
\(M_t(\theta)\geq M_t/2\).  On this set,
\eqref{eq:qmle-sharp-softplus-derivatives} and
\eqref{eq:qmle-local-derivatives-boundary} give the bound below. On its
complement, \eqref{eq:qmle-local-filter-boundary} gives
\(M_t+\mathsf V_t\leq C\mathsf V_t\), and boundedness of \(s_c''\) gives the same bound:
\[
 \sup_{\theta\in \mathcal N_\varepsilon}
 s_c''(\eta_t(\theta))\|A_t(\theta)\|^2
 \leq C(1+\mathsf V_t^2).
\]
The recursion for \(J_t(\theta)\) becomes
\[
 J_t(\theta)=e_\beta D_{t-1}(\theta)^\top
 +D_{t-1}(\theta)e_\beta^\top
 +\beta\bigl[s_c''(\eta_{t-1}(\theta))
 A_{t-1}(\theta)A_{t-1}(\theta)^\top
 +s_c'(\eta_{t-1}(\theta))J_{t-1}(\theta)\bigr].
\]
Taking suprema gives \(J_t^*\leq\rho J_{t-1}^*+F_{t-1}\) for an integrable stationary
\(F_t\).  Backward iteration gives
\[
 J_t^*\leq\rho^kJ_{t-k}^*
 +\sum_{j=0}^{k-1}\rho^jF_{t-1-j}.
\]
Stationarity makes the first term converge to zero in probability, hence
almost surely along a deterministic subsequence.  Passing to that
subsequence gives
\(J_t^*\leq\sum_{j\geq0}\rho^jF_{t-1-j}\) almost surely.  The right side
has expectation \((1-\rho)^{-1}\E F_0\), and therefore
\begin{equation}
 \E J_t^*<\infty.
 \label{eq:qmle-integrable-J}
\end{equation}

It remains to transfer these bounds to {\(\mathcal H_t(\theta)\)
and \(K_t(\theta)\).}
Use \(p_c,r_c,r_c'\) from
Lemma~\ref{lem:softplus-calculus}.  Suppressing the parameter argument \(\theta\) in \(\mathcal H_t,K_t,\eta_t,A_t,J_t\), direct
differentiation gives
\[
 \mathcal H_t=-X_t r_c'(\eta_t)A_tA_t^\top
     +s_c''(\eta_t)A_tA_t^\top
     +(p_c(\eta_t)-X_t r_c(\eta_t))J_t,
 \qquad K_t=\frac{p_c(\eta_t)^2}{M_t(\theta)}A_tA_t^\top.
\]
Split \(\mathcal N_\varepsilon\) into
\[
 \mathcal C_t=\{\theta\in \mathcal N_\varepsilon:M_t(\theta)\geq M_t/2\},
 \qquad \mathcal D_t=\mathcal N_\varepsilon\setminus\mathcal C_t.
\]
For \(\mathcal A_t\in\{\mathcal C_t,\mathcal D_t\}\), let
\(\mathcal H_{t,\mathcal A}^*=\sup_{\theta\in\mathcal A_t}
\|\mathcal H_t(\theta)\|\) and define \(K_{t,\mathcal A}^*\) analogously,
with an empty supremum equal to zero.  {These restricted
suprema are measurable.}\footnote{On the relatively open set
\(\mathcal D_t\), the supremum of a continuous function equals its supremum
over a fixed countable dense subset. On the closed set \(\mathcal C_t\), a
measurable maximum theorem applies.}
On \(\mathcal C_t\),
Lemma~\ref{lem:softplus-calculus},
\eqref{eq:qmle-local-derivatives-boundary}, and
\(\E(1+X_t\mid\mathcal F_{t-1})=1+M_t\) give
\[
 \E(\mathcal H_{t,\mathcal C}^*\mid\mathcal F_{t-1})
 +K_{t,\mathcal C}^*
 \leq C\left[\frac{(M_t+\mathsf V_t)^2}{1+M_t}+J_t^*\right].
\]
If \(\mathcal D_t\ne\varnothing\),
\eqref{eq:qmle-local-filter-boundary} gives
\(M_t\leq C\varepsilon\mathsf V_t\) and
\(M_t+\mathsf V_t\leq C\mathsf V_t\).  The globally bounded quantities
\(r_c,r_c',s_c'',p_c\), and \(p_c^2/s_c\) therefore give
\[
 \E(\mathcal H_{t,\mathcal D}^*\mid\mathcal F_{t-1})
 +K_{t,\mathcal D}^*
 \leq C(1+M_t)\mathsf V_t^2
       \mathbf1_{\{\mathcal D_t\ne\varnothing\}}.
\]
{Take \(C_0\) from Lemma~\ref{lem:qmle-rare-relative-innovation}}
and then decrease
\(\varepsilon\), if necessary, so that, on \(\{M_t>1\}\),
\(\{\mathcal D_t\ne\varnothing\}\subseteq
\{\mathsf V_t>C_0(1+M_t)\}\).  Hence the last display is bounded by
\[
 C\left[\mathsf V_t^2+
 \mathsf V_t^3\mathbf1_{\{\mathsf V_t>C_0(1+M_t)\}}\right].
\]
Finally,
\((M_t+\mathsf V_t)^2/(1+M_t)
 \leq C(1+M_t+\mathsf V_t^2)\).  Taking expectations and using
\(\E M_t<\infty\), \(\mathsf V_t\in L^2\),
\eqref{eq:qmle-integrable-J}, and
\eqref{eq:qmle-rare-relative-moment} {with \(\chi=3\)} proves both envelopes in
\eqref{eq:qmle-boundary-local-envelopes} with
\(\mathcal N=\mathcal N_\varepsilon\) for this fixed \(\varepsilon\).
\end{proof}

\subsection{Proof of Theorem~\ref{thm:affine-qmle}}
\label{app:proof-affine-qmle}

\begin{proof}[{Proof of Theorem~\ref{thm:affine-qmle}}]
We use the notation of Section~\ref{app:inference-notation}.
{In the region \(\phi<1\)}, Theorem~\ref{thm:affine-classification} supplies all
moments.  At the unit root, Theorem~\ref{thm:affine-boundary} and
\(\kappa_0>1\) give \(\E X_t=\E M_t<\infty\), which is the only
stationary level moment used in the unit-root argument.
{Theorems~\ref{thm:affine-mixing} and
\ref{thm:affine-boundary} give a unique invariant state law in cases~(i)
and~(ii), respectively, so the ergodicity step in the proof of
Theorem~\ref{thm:affine-mixing} makes the joint stationary process ergodic
in both cases.}
Lemmas~\ref{lem:affine-qmle-envelopes} and
\ref{lem:boundary-qmle-hessian-envelope} provide a compact convex
neighborhood \(\mathcal N\subseteq\operatorname{int}(\Theta)\)
of \(\theta_0\) such that
\begin{equation}
 \E\sup_{\theta\in\mathcal N}[\|\mathcal H_t(\theta)\|+\|K_t(\theta)\|]<\infty.
 \label{eq:qmle-common-local-envelopes}
\end{equation}

\proofstep{Consistency and identification}
Set \(L_t(\theta)=\ell_t^Q(\theta)-\ell_t^Q(\theta_0)\).
Lemma~\ref{lem:qmle-contrast-envelope} gives an integrable upper envelope
for \(\sup_{\theta\in\Theta}L_t(\theta)\), even when the negative part of
the contrast is not integrable for some \(\theta\in\Theta\).  Conditionally,
\begin{equation*}
 \E[L_t(\theta)\mid\mathcal F_{t-1}]
 =-M_t(\theta_0)
 \mathfrak h\!\left(\frac{M_t(\theta)}{M_t(\theta_0)}\right)\leq0,
\end{equation*}
with equality only when the two conditional means agree.  Equality \(M_t(\theta)=M_t(\theta_0)\) almost surely, together with injectivity of \(s_c\), gives
\[
 (\omega-\omega_0)+(\alpha-\alpha_0)X_{t-1}
 +(\beta-\beta_0)M_{t-1}(\theta_0)=0\quad\text{a.s.}
\]
The regressors \(1,X_{t-1},M_{t-1}(\theta_0)\) are linearly independent.
Indeed, since \(M_{t-1}(\theta_0)\) is \(\mathcal F_{t-2}\)-measurable,
conditional variance given \(\mathcal F_{t-2}\) forces the coefficient on
\(X_{t-1}\) in any almost-sure linear relation to vanish, because
\(\Var(X_{t-1}\mid\mathcal F_{t-2})
=\sigma_0^2M_{t-1}(\theta_0)>0\).
Moreover, \(M_{t-1}(\theta_0)\) is nonconstant: otherwise injectivity of
\(s_c\), stationarity, and \(\alpha_0\ne0\) would make \(X_{t-2}\)
constant, contradicting the same positive conditional variance.  Hence the remaining two
coefficients also vanish, and the preceding relation gives
\(\omega=\omega_0\), \(\alpha=\alpha_0\), and \(\beta=\beta_0\).  Thus
\(\E L_t(\theta)<0\), possibly with value \(-\infty\), at every
\(\theta\ne\theta_0\).  Continuity of \(M_t(\theta)\) in \(\theta\) makes \((L_t)\) a
stationary and ergodic sequence of random continuous functions on
\(\Theta\).
To apply \citet[Theorem~33(ii)]{douc2013ergodicity}, define
the stationary and initialized averaged contrasts by
\[
 \mathcal L_n(\theta)=\frac1n\sum_{t=1}^nL_t(\theta),\qquad
 \widetilde{\mathcal L}_n(\theta)
 =\frac1n\left[Q_n(\theta)-\sum_{t=1}^n\ell_t^Q(\theta_0)\right].
\]
{Its conditions (C3)--(C5) require an integrable upper envelope,
almost sure upper semicontinuity of \(L_0\), and almost sure uniform
convergence of \(\widetilde{\mathcal L}_n-\mathcal L_n\) to zero, with
upper semicontinuous initialized criteria. With \(x^+=\max\{x,0\}\),
Lemma~\ref{lem:qmle-contrast-envelope} gives
\[
 \sup_{\theta\in\Theta}(L_0(\theta))^+
 \leq\frac{U_0^2}{M_0},\qquad
 \E\!\left[\frac{U_0^2}{M_0}\right]=\sigma_0^2,
\]
which is (C3), without requiring integrability of the negative part.
Continuity of the stationary and initialized filters gives (C4) and the
upper semicontinuity of the initialized criteria.
Lemma~\ref{lem:qmle-initialization} bounds
\(\sup_{\theta\in\Theta}|\widetilde{\mathcal L}_n-\mathcal L_n|\) by
\(O(n^{-1})\) almost surely, which is (C5).}
{Theorem~33(ii) of \citet{douc2013ergodicity} then states}
that the distance of any measurable maximizing
selection to \(\operatorname*{argmax}_{\theta\in\Theta}\E L_0(\theta)\)
converges to zero almost surely. The initialized contrast has the same
maximizers as \(Q_n\), since the subtracted criterion is parameter
independent, and continuity and compactness ensure that maximizers exist.
The identification argument above makes the population maximizer set
\(\{\theta_0\}\). Thus \(\widehat\theta_n\to\theta_0\) almost surely
for any measurable maximizing selection.

\proofstep{Score central limit theorem}
At the truth, abbreviate \(D_t=D_t(\theta_0)\) and
\(S_t=S_t(\theta_0)\).  The quasi-score
\[
 S_t=\left(\frac{X_t}{M_t}-1\right)D_t
\]
is a martingale difference, and the affine conditional variance gives
\begin{equation*}
 \E(S_tS_t^\top\mid\mathcal F_{t-1})
 =\sigma_0^2\frac{D_tD_t^\top}{M_t}.
\end{equation*}
The information part of \eqref{eq:qmle-common-local-envelopes} makes
\(S_t\) square integrable and \(I_0=\E K_t(\theta_0)\) finite.

The information matrix is also automatically positive definite.  If
\(v^\top I_0v=0\), then
\(v^\top D_t=0\) almost surely.  Since \(s_c'(\eta_t)>0\), the derivative
recursion gives
\[
 v_\omega+v_\alpha X_{t-1}+v_\beta M_{t-1}
 +\beta_0v^\top D_{t-1}=0\quad\text{a.s.}
\]
Since \(v^\top D_{t-1}=0\) almost surely, the regressor linear
independence established in the identification step gives \(v=0\).

{Fix \(v\in\mathbb R^3\setminus\{0\}\). The variables
\(n^{-1/2}v^\top S_t\), \(1\leq t\leq n\), form a square-integrable
martingale-difference array with respect to \((\mathcal F_t)\). Because the
array comes from a single martingale, the nesting condition of
\citet{hallheyde1980} holds automatically. Its conditional variances sum to
\[
 \sigma_0^2v^\top\left[
 \frac1n\sum_{t=1}^n\frac{D_tD_t^\top}{M_t}
 \right]v
 \longrightarrow\sigma_0^2v^\top I_0v
 \quad\text{a.s.}
\]
by ergodicity. For the conditional Lindeberg condition, fix
\(\varepsilon>0\) and use
\((v^\top S_t)^2\mathbf1_{\{|v^\top S_t|>\varepsilon\sqrt n\}}
\leq\|v\|^2\|S_t\|^2\mathbf1_{\{\|S_t\|>\varepsilon\sqrt n/\|v\|\}}\).
By stationarity, the conditional expectations of the right-hand side,
averaged over \(t\leq n\), have expectation
\(\|v\|^2\E[\|S_0\|^2\mathbf1_{\{\|S_0\|>\varepsilon\sqrt n/\|v\|\}}]\to0\),
so Markov's inequality gives convergence to zero in probability.
Corollary~3.1 (p.~58) of \citet{hallheyde1980} then gives the scalar normal
limit for every \(v\), and the Cram\'er--Wold device gives}
\begin{equation}
 n^{-1/2}\sum_{t=1}^nS_t
 \Rightarrow N(0,\sigma_0^2I_0).
 \label{eq:qmle-score-clt}
\end{equation}

\proofstep{Uniform laws of large numbers for the Hessian and information matrix}
Let \(C(\mathcal N,\mathbb R^{3\times3})\) denote the space of continuous
matrix-valued functions on \(\mathcal N\), equipped with the supremum norm
\(\|F\|_\infty=\sup_{\theta\in\mathcal N}\|F(\theta)\|\).
The process \(((\mathcal H_t,K_t))\) is stationary and ergodic as a sequence of
random elements in the Cartesian product
\(C(\mathcal N,\mathbb R^{3\times3})^2\).  Therefore,
\eqref{eq:qmle-common-local-envelopes} and {the uniform strong law of large numbers for
matrix-valued functions \citep[Theorem~2.7 and Section~2.4]{straumannMikosch2006}} give
\begin{equation}
 \sup_{\theta\in\mathcal N}\left\|
 n^{-1}\sum_{t=1}^n\mathcal H_t(\theta)-\E \mathcal H_t(\theta)
 \right\|
 +\sup_{\theta\in\mathcal N}\left\|
 n^{-1}\sum_{t=1}^nK_t(\theta)-\E K_t(\theta)
 \right\|\longrightarrow0\quad\text{a.s.}
 \label{eq:qmle-uniform-local-hessian}
\end{equation}
Because the convergence takes place in \(C(\mathcal N,\mathbb R^{3\times3})^2\), both
mean functions \(\E \mathcal H_t(\theta)\) and \(\E K_t(\theta)\) are continuous
on \(\mathcal N\).  Differentiating the quasi-score at the truth and taking conditional
expectations gives \(\E \mathcal H_t(\theta_0)=\E K_t(\theta_0)=I_0\).  Hence, whenever \(d_n=o_p(1)\) and the segment
\(\{\theta_0+u d_n:0\leq u\leq1\}\) lies in \(\mathcal N\) with probability tending
to one,
\begin{equation}
 \int_0^1 n^{-1}\sum_{t=1}^n
 \mathcal H_t(\theta_0+u d_n)\,du\ \xrightarrow{p}\ I_0.
 \label{eq:qmle-integrated-hessian}
\end{equation}

\proofstep{Initialization and the score expansion}
{By consistency, \(\widehat\theta_n\in\mathcal N\subseteq\operatorname{int}(\Theta)\)
with probability tending to one. Lemma~\ref{lem:qmle-initialization}, applied
with \(\mathcal K=\mathcal N\), bounds the difference between the initialized
and stationary score sums by \(O_p(1)\), uniformly on \(\mathcal N\). On this
event, put \(d_n=\widehat\theta_n-\theta_0\) and
\[
 \bar{\mathcal H}_n=\int_0^1\frac1n\sum_{t=1}^n
 \mathcal H_t(\theta_0+u d_n)\,du.
\]
The initialized first-order condition and the fundamental theorem of calculus,
applied to the stationary vector score, give
\[
 0=n^{-1/2}\sum_{t=1}^nS_t(\theta_0)-\bar{\mathcal H}_n\sqrt n\,d_n+o_p(1).
\]
By \eqref{eq:qmle-integrated-hessian}, \(\bar{\mathcal H}_n\xrightarrow{p}I_0\),
which is positive definite. Together with \eqref{eq:qmle-score-clt}, this gives
\[
 \sqrt n\,d_n=I_0^{-1}n^{-1/2}\sum_{t=1}^nS_t(\theta_0)+o_p(1)
\]
and hence \(\sqrt n(\widehat\theta_n-\theta_0)\Rightarrow N(0,\sigma_0^2I_0^{-1})\).}

\proofstep{Feasible information and dispersion}
The uniform law \eqref{eq:qmle-uniform-local-hessian} for \(K_t\), the
continuity of \(\E K_t(\theta)\) on \(\mathcal N\), and consistency show that the
stationary sample information evaluated at the estimator converges in
probability to \(I_0\).
Lemma~\ref{lem:qmle-initialization} makes the total difference between
the initialized and stationary sample-information sums \(O_p(1)\).
Division by \(n\) proves consistency of \(\widehat I_{Q,n}\).

It remains to estimate \(\sigma_0^2\).  Evaluate the stationary-filter
discrepancy \(\delta_t(\theta)\) at \(\widehat\theta_n\).  On the event \(\widehat\theta_n\in\mathcal N\), convexity of
\(\mathcal N\) and the fundamental theorem of calculus applied to the scalar path
\(u\mapsto M_t(\theta_0+u d_n)\) give
\begin{equation}
 \delta_t(\widehat\theta_n)
 =\left[\int_0^1D_t(\theta_0+u d_n)\,du\right]^{\!\top}d_n.
 \label{eq:qmle-filter-difference-integral}
\end{equation}
{Lemma~\ref{lem:affine-qmle-envelopes} gives
\(\sup_{\theta\in\Theta}\|D_t(\theta)\|\leq CZ_t\), and
\(\E Z_0=(1+\E X_0)/(1-\bar\rho)<\infty\) in both cases. Thus
\eqref{eq:qmle-filter-difference-integral} gives
\(|\delta_t(\widehat\theta_n)|\leq C\|d_n\|Z_t\). For every
\(\varepsilon>0\), stationarity gives
\[
 \Pr\Bigl(\max_{t\leq n}Z_t>\varepsilon n\Bigr)
 \leq n\Pr(Z_0>\varepsilon n)
 \leq\varepsilon^{-1}\E[Z_0\mathbf1_{\{Z_0>\varepsilon n\}}]
 \longrightarrow0,
\]
and the ergodic theorem gives \(n^{-1}\sum_{t\leq n}Z_t\to\E Z_0\) almost
surely. Hence
\[
 \frac1{n^2}\sum_{t=1}^nZ_t^2
 \leq\left(\frac1n\max_{t\leq n}Z_t\right)
       \left(\frac1n\sum_{t=1}^nZ_t\right)\xrightarrow{p}0.
\]
Since the already proved root-\(n\) rate gives \(n\|d_n\|^2=O_p(1)\),
\[
 \frac1n\sum_{t=1}^n\delta_t(\widehat\theta_n)^2
 \leq C(n\|d_n\|^2)
       \left[\frac1{n^2}\sum_{t=1}^nZ_t^2\right]=o_p(1).
\]
This uses only the first observation moment, in both cases.}
Using \(U_t\) from \eqref{eq:affine-innovation}, the Cauchy--Schwarz inequality and the ergodic
theorem then give
\[
 n^{-1}\sum_{t=1}^n|U_t\delta_t(\widehat\theta_n)|=o_p(1).
\]
Only the innovation's second moment is used here, and it follows from the
state's first moment:
\[
 \E(U_t^2\mid\mathcal F_{t-1})=\sigma_0^2M_t,\qquad
 \E M_t<\infty.
\]
Thus
\[
 \frac{n^{-1}\sum_{t=1}^n(X_t-M_t)^2}
      {n^{-1}\sum_{t=1}^nM_t}\longrightarrow\sigma_0^2
\quad\text{a.s.},
\]
where \(\E M_t>0\) because \(M_t>0\) almost surely.  Replacing \(M_t\) by
the fitted stationary filter changes neither average asymptotically. In
particular, \(n^{-1}\sum_t|\delta_t(\widehat\theta_n)|=o_p(1)\).
For the initialized fitted filter, on
\(\{\widehat\theta_n\in\mathcal N\}\), the bounds in the proof
of Lemma~\ref{lem:qmle-initialization} give
\[
 \sum_{t\geq1}
 \left|(X_t-\widetilde M_t(\widehat\theta_n))^2
       -(X_t-M_t(\widehat\theta_n))^2\right|
 \leq CB_0\sum_{t\geq1}\rho^t
       (X_t+Z_t+B_0\rho^t)<\infty\quad\text{a.s.},
\]
and also
\[
 \sum_{t\geq1}|\widetilde M_t(\widehat\theta_n)-M_t(\widehat\theta_n)|
 \leq B_0\sum_{t\geq1}\rho^t<\infty.
\]
This proves \(\widehat\sigma_n^2\xrightarrow{p}\sigma_0^2\)
and completes the proof.
\end{proof}
\subsection{Proof of Theorem~\ref{thm:rootn-local-power}}
\label{app:proof-rootn-local-power}

\leavevmode{Let \(\pi_n\) and \(\pi_0\) be the invariant state laws at
\(\theta_n\) and \(\theta_0\). The stationary laws at the two parameters start
from different initial laws, and \(\pi_n\) has no closed form. After proving
\(\pi_n\Rightarrow\pi_0\), we therefore establish LAN for an auxiliary
experiment that starts from \(\pi_0\) at every parameter, so that the initial
law cancels from the likelihood ratio. A coupling of the two initial laws then
transfers the limits to the stationary solution at \(\theta_n\).}

\begin{proof}[{Proof of Theorem~\ref{thm:rootn-local-power}}]
\proofstep{Invariant laws}
Let \(P_n\) and \(P_0\) be the state transition kernels at \(\theta_n\) and
\(\theta_0\).  Given the current state \(m\), use the same
\(W_m\sim G_m\) in both transitions.  Monotonicity of softplus gives
\begin{equation}
 s_c\!\left(\omega_0+\alpha_0W_m+
       \left(\beta_0-\frac h{\sqrt n}\right)m\right)
 \leq s_c(\omega_0+\alpha_0W_m+\beta_0m).
 \label{eq:rootn-kernel-order}
\end{equation}
Thus \(P_nV\leq P_0V\) for every increasing \(V\), and
Theorems~\ref{thm:affine-classification} and
\ref{thm:affine-mixing} give a unique invariant law \(\pi_n\) for all
sufficiently large \(n\).  Fix \(n_0\) so that these conclusions hold for
every \(n\geq n_0\).

Fix \(p\) with \(1<p<\kappa_0\).  The power-drift inequality at \(\theta_0\), applied
to the function \(V_{p+1}\) from the proof of
Theorem~\ref{thm:affine-classification}, gives
\[
 P_0V_{p+1}(m)\leq V_{p+1}(m)-C_pm^p+B_p.
\]
For every sufficiently large \(n\), \(\phi_n<1\), so \(\pi_n\) has moments of all positive orders.  Integrating the preceding inequality
under \(\pi_n\) and using \eqref{eq:rootn-kernel-order} yields
\begin{equation*}
 \sup_{n\geq n_0}\int m^p\,\pi_n(dm)<\infty.
\end{equation*}
Hence \((\pi_n)\) is tight on \([0,\infty)\), where the Poisson kernel is
extended at zero by taking the conditional observation to be zero.  For
bounded Lipschitz \(f\),
\[
 \sup_{m\leq R}|P_nf(m)-P_0f(m)|
 \leq \operatorname{Lip}(f)hR/\sqrt n.
\]
If \(\pi_{n_k}\Rightarrow\pi\), invariance, this locally uniform kernel
convergence, tightness, and the Feller property satisfy the hypotheses of
part~(ii) of Lemma~\ref{lem:feller-invariant-limits}. Hence
\(\pi P_0=\pi\).  Moreover,
softplus is strictly positive, so \(P_0(m,\{0\})=0\) for every \(m\geq0\),
and the same lemma gives \(\pi(\{0\})=0\). By uniqueness in
Theorem~\ref{thm:affine-boundary}, \(\pi=\pi_0\). Thus
\begin{equation}
 \pi_n\Rightarrow\pi_0.
 \label{eq:rootn-invariant-convergence-proof}
\end{equation}
Let \(F_n^{-1}\) and \(F_0^{-1}\) be the generalized quantile functions of
\(\pi_n\) and \(\pi_0\), and let \(U\) be uniform on \((0,1)\).  The quantile
coupling
\[
 A_n=F_n^{-1}(U),\qquad A=F_0^{-1}(U)
\]
then satisfies \(A_n\sim\pi_n\), \(A\sim\pi_0\), and
\(A_n\to A\) almost surely by \eqref{eq:rootn-invariant-convergence-proof}.

\proofstep{LAN under a common initial law}
For fixed \(u\in\mathbb R^3\), put
\(\vartheta_{n,u}=\theta_0+u/\sqrt n\). This belongs to the interior
neighborhood used in the proof of Theorem~\ref{thm:affine-qmle} for all
large \(n\).  Define an auxiliary law \(\mathsf Q_{n,u}\) as follows.
Draw
\[
 M_0^\circ\sim\pi_0,
 \qquad X_0\mid M_0^\circ\sim\operatorname{Pois}(M_0^\circ),
\]
with this joint initial law fixed for every \(u\).  Define
\((\check M_t(\theta))\) by recursion~\eqref{mrecursion} at parameter
\(\theta\), starting from \(\check M_0(\theta)=M_0^\circ\) for every
\(\theta\). Thus \(\partial_\theta\check M_0(\theta)=0\).
Under \(\mathsf Q_{n,u}\), conditionally on
\(M_0^\circ,X_0,\ldots,X_{t-1}\), let
\(X_t\sim\operatorname{Pois}(\check M_t(\vartheta_{n,u}))\).
For every \(\theta\), \(\check M_t(\theta)>0\), so these
conditional Poisson laws are mutually absolutely continuous.  The initial
law is the same under \(\mathsf Q_{n,u}\) and \(\mathsf Q_{n,0}\), and
therefore cancels from their joint likelihood ratio.  With
\[
 \check\ell_t^Q(\theta)
 =X_t\log\check M_t(\theta)-\check M_t(\theta),
\]
that exact ratio is
\begin{equation*}
 \Lambda_n(u)
 :=\log\frac{d\mathsf Q_{n,u}}{d\mathsf Q_{n,0}}
       (M_0^\circ,X_0,\ldots,X_n)
 =\sum_{t=1}^n
   [\check\ell_t^Q(\vartheta_{n,u})
     -\check\ell_t^Q(\theta_0)].
\end{equation*}

{Realize \(\mathsf Q_{n,0}\) on the two-sided stationary null process,
and compute \(\ell_t^Q\), \(S_t\), and \(\mathcal H_t\) from its stationary
filter \((M_t(\theta))\), with the notation of
Section~\ref{app:inference-notation}. Lemma~\ref{lem:qmle-initialization},
applied with the random initial value \(\widetilde M_0=M_0^\circ\), gives
\[
 B_{\rm init}:=\sum_{t=1}^\infty\sup_{\theta\in\mathcal N}
 \|\check S_t(\theta)-S_t(\theta)\|<\infty
 \quad\mathsf Q_{n,0}\text{-a.s.},
 \qquad \check S_t=\partial_\theta\check\ell_t^Q,
\]
where \(\mathcal N\) is the compact neighborhood used in the proof of
Theorem~\ref{thm:affine-qmle}. Define the initialized central sequence
\(\Delta_n^{\mathrm{LAN}}=n^{-1/2}\sum_{t=1}^n\check S_t(\theta_0)\).
For all large \(n\), the segment from \(\theta_0\) to \(\vartheta_{n,u}\)
lies in \(\mathcal N\), and the fundamental theorem of calculus gives
\[
 \begin{aligned}
 \Bigl|\Lambda_n(u)-\sum_{t=1}^n
 [\ell_t^Q(\vartheta_{n,u})-\ell_t^Q(\theta_0)]\Bigr|
 &\leq\|u\|B_{\rm init}/\sqrt n,\\
 \Bigl\|\Delta_n^{\mathrm{LAN}}-n^{-1/2}\sum_{t=1}^nS_t(\theta_0)\Bigr\|
 &\leq B_{\rm init}/\sqrt n.
 \end{aligned}
\]
Taylor's formula with integral remainder,
\eqref{eq:qmle-uniform-local-hessian}, continuity of
\(\E\mathcal H_t(\theta)\) at \(\theta_0\), and
\(\E\mathcal H_t(\theta_0)=I_0\) give
\[
 \sum_{t=1}^n[\ell_t^Q(\vartheta_{n,u})-\ell_t^Q(\theta_0)]
 =\frac{u^\top}{\sqrt n}\sum_{t=1}^nS_t(\theta_0)-\tfrac12u^\top I_0u+o_p(1).
\]
Since \(\sigma_0^2=1\), these displays and \eqref{eq:qmle-score-clt} give the
LAN expansion}
\begin{equation}
 \Lambda_n(u)=u^\top\Delta_n^{\mathrm{LAN}}-\tfrac12u^\top I_0u+o_p(1),
 \qquad \Delta_n^{\mathrm{LAN}}\Rightarrow N(0,I_0),
 \label{eq:rootn-common-lan}
\end{equation}
under \(\mathsf Q_{n,0}\).

The exponential of the Gaussian limit in
\eqref{eq:rootn-common-lan} has expectation one.  Le Cam's first lemma
therefore makes \(\mathsf Q_{n,u}\) contiguous with respect to
\(\mathsf Q_{n,0}\), and Le Cam's third lemma gives the following shift (see
\citet[Chapter~6]{vanDerVaart1998} for both results):
\begin{equation}
 \Delta_n^{\mathrm{LAN}}\Rightarrow N(I_0u,I_0)
 \quad\text{under }\mathsf Q_{n,u}.
 \label{eq:rootn-common-score-shift}
\end{equation}
{Under \(\mathsf Q_{n,0}\), the score expansion in the proof of
Theorem~\ref{thm:affine-qmle} and the preceding comparison give}
\[
 \sqrt n(\widehat\theta_n-\theta_0)
 =I_0^{-1}\Delta_n^{\mathrm{LAN}}+o_p(1).
\]
{The \(o_p(1)\) remainder and the convergences
\(\widehat I_{Q,n}\to I_0\) and \(\widehat\sigma_n^2\to1\), which hold under
\(\mathsf Q_{n,0}\), carry over to \(\mathsf Q_{n,u}\) by contiguity.}
Combining this fact with
\eqref{eq:rootn-common-score-shift} yields, under \(\mathsf Q_{n,u}\),
\begin{equation}
 \sqrt n(\widehat\theta_n-\vartheta_{n,u})
 \Rightarrow N(0,I_0^{-1}),
 \qquad
 \widehat I_{Q,n}\xrightarrow{p}I_0,
 \qquad
 \widehat\sigma_n^2\xrightarrow{p}1.
 \label{eq:rootn-common-estimator}
\end{equation}
Although \(M_0^\circ\) was included to write the exact likelihood ratio,
the QMLE, the estimated information matrix \(\widehat I_{Q,n}\), and the
dispersion estimator \(\widehat\sigma_n^2\) are functions only of
\((X_0,\ldots,X_n)\). Hence \eqref{eq:rootn-common-estimator} also holds
under the observation marginal of \(\mathsf Q_{n,u}\).

\proofstep{Transfer to stationary solutions at \(\theta_n\)}
Let \(u_h=(0,0,-h)^\top\), so
\(\vartheta_{n,u_h}=\theta_n\).
Initialize the two chains with \(M_0=A_n\) and \(M_0^\circ=A\) from the
quantile coupling above, {evolve both at \(\theta_n\), and write
\((X_t)\) and \((X_t^\circ)\) for their observations.}
Choose \(\rho_\beta\in(|\beta_0|,1)\).  Since \(\beta_n\to\beta_0\),
we have \(|\beta_n|\leq\rho_\beta\) for all sufficiently large \(n\).
{Evolve the pair with the coupling kernel of
Lemma~\ref{lem:finite-activity-wasserstein-coupling} at \(\theta_n\). For
Poisson observations, \(p_0=G_1(\{0\})=e^{-1}\). Let
\(\tau=\inf\{t\geq0:X_t\ne X_t^\circ\}\). Bound
\eqref{eq:finite-activity-permanent-agreement}, which holds at \(\theta_n\)
with \(|\beta_n|\) in the role of \(|\beta|\), and \(|\beta_n|\leq\rho_\beta\)
give
\[
 \Pr(\tau<\infty\mid A_n,A)
 \leq1-\exp\left(-\frac{|A_n-A|}{1-\rho_\beta}\right)
 \leq1\wedge\frac{|A_n-A|}{1-\rho_\beta}.
\]}
{For a law \(\mu\) of the initial state, write
\(\mathcal L_{\mu,\theta}(X_0,\ldots,X_k)\) for the law of
\((X_0,\ldots,X_k)\) when \(M_0\sim\mu\) and the model evolves at
\(\theta\).}
The coupling inequality now gives, uniformly over the sample horizon,
\begin{equation*}
 \sup_{k\geq0}
 \left\|\mathcal L_{\pi_n,\theta_n}(X_0,\ldots,X_k)
       -\mathcal L_{\pi_0,\theta_n}(X_0,\ldots,X_k)
 \right\|_{\rm TV}
 \leq\E\left[1\wedge\frac{|A_n-A|}{1-\rho_\beta}\right]
 \longrightarrow0.
\end{equation*}
The convergence follows from \(A_n\to A\) almost surely and bounded
convergence.
For \(k=n\), the second law is precisely the observation marginal of
\(\mathsf Q_{n,u_h}\), whereas the first is the observation law of the stationary solution at \(\theta_n\).  Thus \eqref{eq:rootn-common-estimator} and the coupling
inequality prove \eqref{eq:rootn-local-estimator} for the stationary solutions at \(\theta_n\).

Finally, \(e_\phi^\top u_h=-h\), and hence
\[
 \sqrt n(\widehat\phi_n-1)
 =e_\phi^\top\sqrt n(\widehat\theta_n-\theta_n)-h.
\]
Equation~\eqref{eq:rootn-local-estimator} and Slutsky's lemma yield
\(T_n\Rightarrow N(-h/\sqrt{v_{\phi,0}},1)\).
Evaluating its lower tail gives
\eqref{eq:rootn-local-power}.
\end{proof}

\section{\texorpdfstring{{Selection of the illustrative crime series}}{Selection of the illustrative crime series}}
\label{app:crime-series-selection}

{The 44-series shortlist used to select the Wollondilly / Theft /
Break and enter dwelling series in Section~\ref{sec:nsw-crime-application}
was formed at \(c=1\). Admission required the full-sample numerical and
data-quality criteria stated in that section and a lag-one sample
autocorrelation of at least 0.4. In both the full sample (January
1995--March 2026) and the January 1995--December 2019 subsample, fits had
to pass the numerical screen, have a 95\% confidence interval for \(\phi\)
containing one, and yield a negative intercept and a tail index greater
than two when fitted under the restriction \(\phi=1\).}

\end{document}